\documentclass[USenglish, a4paper, 10pt]{article}
\usepackage[utf8]{inputenc} 
\usepackage{babel} 
\usepackage{lmodern} 
\usepackage[T1]{fontenc} 
\usepackage{amssymb, amsmath, mathrsfs, mathabx, bbm} 
\DeclareMathAlphabet{\mathbbold}{U}{bbold}{m}{n}
\usepackage[scr=boondoxupr]{mathalpha} 
\usepackage{stmaryrd}

\usepackage[a4paper]{geometry} 
\usepackage{titling} 
\usepackage{titlesec} 
		\titlelabel{\thetitle.\quad} 
		\titleformat*{\section}{\center\large} 
		\titleformat*{\subsection}{\sf\large} 
		\titleformat*{\subsubsection}{\sf\it} 
		
\usepackage[colorlinks=true,linkcolor=black, citecolor=GreenCite,urlcolor=GreenCite]{hyperref} 
\usepackage[usenames,dvipsnames]{color} 
\definecolor{ToDo}{RGB}{30,144,255}
\definecolor{Provisional}{RGB}{218,165,32}
\definecolor{Question}{RGB}{220,20,60}
\definecolor{GreenCite}{RGB}{47, 79, 79}

\usepackage{verbatim} 
\usepackage{array} 
\usepackage[all]{xy} 
\usepackage[usenames,dvipsnames]{color} 
\usepackage{enumerate} 
\usepackage{leftidx}
\usepackage{paralist}

\numberwithin{equation}{section} 
\usepackage{amsthm} 
\swapnumbers 
\theoremstyle{plain}
\newtheorem{theoSec}{Theorem}[section] 

\newtheorem{lemSec}[theoSec]{Lemma}
\newtheorem{proSec}[theoSec]{Proposition}
\newtheorem{corSec}[theoSec]{Corollary}
\newtheorem{defiSec}[theoSec]{Definition}
\newtheorem{remSec}[theoSec]{Remark}

\theoremstyle{remark}

\theoremstyle{plain}
\newtheorem{theo}{Theorem}[subsection] 
\newtheorem{theodefi}[theo]{Theorem-Definition}
\newtheorem{lem}[theo]{Lemma}
\newtheorem{pro}[theo]{Proposition}
\newtheorem{cor}[theo]{Corollary}
\newtheorem{defi}[theo]{Definition}
\newtheorem{rem}[theo]{Remark}

\newtheorem{ex}[theo]{Example}
\newtheorem{exs}[theo]{Examples}
\theoremstyle{remark}
\newtheorem{note}[theo]{Note}

\theoremstyle{plain}

\theoremstyle{remark}

\title{\textbf{Permanence of the torsion-freeness property \\for divisible discrete quantum subgroups}}
\author{\textsc{Rubén Martos}\thanks{Department of Mathematical Sciences, University of Copenhagen (Denmark). R.M. is supported by the European Union's Horizon 2020 research and innovation programme under the Marie Skłodowska-Curie grant agreement No 895141.}}
\date{}
\begin{document}
\maketitle
\renewcommand{\abstractname}{}
\vspace{-2.5cm}
	
\begin{abstract}
\textsc{Abstract}. We prove that torsion-freeness in the sense of Meyer-Nest is preserved under divisible discrete quantum subgroups. As a consequence, we obtain some stability results of the torsion-freeness property for relevant constructions of quantum groups (quantum (semi-)direct products, compact bicrossed products and quantum free products). We improve some stability results concerning the Baum-Connes conjecture appearing already in a previous work of the author. For instance, we show that the (resp. strong) Baum-Connes conjecture is preserved by discrete quantum subgroups (without any torsion-freeness or divisibility assumption).
	
\end{abstract}

\tableofcontents

\section{\textsc{Introduction}}
A major problem when studying the quantum counterpart of the Baum-Connes conjecture for a discrete quantum group $\widehat{\mathbb{G}}$ is the torsion structure of $\widehat{\mathbb{G}}$. Indeed, if $G$ is a discrete group, its torsion phenomena is completely described in terms of the finite subgroups of $G$ and encoded in the localizing subcategory of $\mathscr{K}\mathscr{K}^{G}$ of \emph{compactly induced $G$-C$^*$-algebras}, denoted by $\mathscr{L}_G$, according to the Meyer-Nest reformulation \cite{MeyerNest}. The notion of torsion for a genuine discrete quantum group, $\widehat{\mathbb{G}}$, has been introduced firstly by R. Meyer and R. Nest \cite{MeyerNestTorsion}, \cite{MeyerNestHomological2} in terms of ergodic actions of $\mathbb{G}$. It has been re-interpreted later by Y. Arano and K. De Commer in terms of fusion rings and module C$^*$-categories \cite{YukiKenny}. In order to apply the Meyer-Nest strategy in the quantum setting, one needs a \emph{complementary pair} of localizing subcategories, $(\mathscr{L}_{\widehat{\mathbb{G}}},\mathscr{N}_{\widehat{\mathbb{G}}})$, where $\mathscr{L}_{\widehat{\mathbb{G}}}$ must encode the torsion phenomena of $\widehat{\mathbb{G}}$. A candidate has been apparent for specific examples \cite[Section 1]{MeyerNestTorsion} and \cite[Section 5]{VoigtBaumConnesAutomorphisms} (see also \cite[Section 4.1.2]{RubenThesis} for a description for general discrete quantum groups), but it has been an open question whether the corresponding pair is complementary in $\mathscr{K}\mathscr{K}^{\widehat{\mathbb{G}}}$, which prevented from having a definition of a quantum assembly map whenever $\widehat{\mathbb{G}}$ is not torsion-free. Recently, Y. Arano and A. Skalski \cite{YukiBCTorsion} have observed that the candidates for $\mathscr{L}_{\widehat{\mathbb{G}}}$ and $\mathscr{N}_{\widehat{\mathbb{G}}}$ form indeed a complementary pair of subcategories in $\mathscr{K}\mathscr{K}^{\widehat{\mathbb{G}}}$, which allows to define a quantum assembly map for every discrete quantum group $\widehat{\mathbb{G}}$ (torsion-free or not). Moreover, following a different approach by studying the projective representation theory of a compact quantum group, the same conclusion is reached for permutation torsion-free discrete quantum groups by the author in collaboration with K. De Commer and R. Nest \cite{KennyNestRubenBCProjective}.

In this article, we study some properties in relation with the torsion phenomenon as well as the Baum-Connes property for discrete quantum subgroups.

On the one hand, we answer in the affirmative the question: \emph{is torsion-freeness preserved by divisible discrete quantum subgroups?} Let $\mathbb{G}$ and $\mathbb{H}$ be two compact quantum groups such that $\widehat{\mathbb{H}}<\widehat{\mathbb{G}}$. It is well-known that torsion-freeness in the sense of Meyer-Nest is \emph{not} preserved by discrete quantum subgroups in general. The approach of Arano-De Commer \cite{YukiKenny} allows to introduce a stronger notion of torsion-freeness, which also is \emph{not} preserved by discrete quantum subgroups. However, they show that it is preserved when $\widehat{\mathbb{H}}$ is \emph{divisible} in $\widehat{\mathbb{G}}$. Therefore it is reasonable to ask for the same stability for torsion-freeness in the sense of Meyer-Nest. The notion of divisible discrete quantum subgroup was introduced in \cite{VoigtBaumConnesUnitaryFree}. Roughly speaking, this property corresponds to the existence of a section of the canonical quotient map, so every inclusion of classical discrete groups is divisible. It turns out that this class of discrete quantum subgroups have a well-behaved representation theory, which can be taken advantage to study the torsion phenomenon. More precisely, given a torsion action of $\mathbb{H}$, $(B, \beta)$, we  can consider a torsion action of $\mathbb{G}$, $(\text{Ind}^{\mathbb{G}}_{\mathbb{H}} (B), \widetilde{\beta})$, by simply composing with the natural inclusion $C(\mathbb{H})\subset C(\mathbb{G})$ given by $\widehat{\mathbb{H}}<\widehat{\mathbb{G}}$. A result of K. De Commer and M. Yamashita \cite{KennyMakoto} states a one-to-one correspondence (up to equivariant Morita equivalence) between ergodic actions of $\mathbb{G}$ and connected $\mathscr{Rep}(\mathbb{G})$-module C$^*$-categories. Thus the main task to answer the question above is to understand the relation between the fusion rings associated to the module categories obtained from these torsion actions: $\mathscr{N}_\beta \leftrightarrow (B,\beta)$ and $\mathscr{M}_{\widetilde{\beta}}\leftrightarrow (\text{Ind}^{\mathbb{G}}_{\mathbb{H}} (B), \widetilde{\beta})$. In the framework of Arano-De Commer's approach to torsion-freeness, it is natural to look at the following construction. Given the $\mathscr{Rep}(\mathbb{H})$-module C$^*$-category $\mathscr{N}_\beta$, we can consider inducing it to obtain a $\mathscr{Rep}(\mathbb{G})$-module C$^*$-category, $\text{Ind}^{\mathscr{Rep}(\mathbb{G})}_{\mathscr{Rep}(\mathbb{H})}(\mathscr{N}_\beta)$. This involves the notion of \emph{balanced or relative tensor product of module C$^*$-categories with respect to a C$^*$-tensor category}. Such a construction is well-known in the algebraic framework (see for instance \cite{Deligne}, \cite{Tambara} or \cite{EtingofFusCatHomTh}) and foreseen by experts at the C$^*$-level (see for instance \cite{AmbrogioLimitsCCat}, \cite{MeyerColimitsCCorresp} or \cite{SergeyMakoto}), but a formal and general definition seemed to be elusive in the literature. J. Antoun and C. Voigt have defined such a tensor product without any semi-simplicity or rigidity assumption on the C$^*$-categories involved. First, we show that the divisibility property extends at the level of the representation categories (see Proposition \ref{pro.DivisibilityRepCat}), which allows to identify the fusion module of $\text{Ind}^{\mathscr{Rep}(\mathbb{G})}_{\mathscr{Rep}(\mathbb{H})}(\mathscr{N}_\beta)$ to the induced module from $\text{Fus}(\mathscr{N}_\beta)$. This allows in turn to answer the above question (see Theorem \ref{theo.TorsionFreenessDivisible}). Second, we show moreover that induction at the level of torsion actions corresponds precisely to induction at the level of module C$^*$-categories. In other words, we show that $\text{Ind}^{\mathscr{Rep}(\mathbb{G})}_{\mathscr{Rep}(\mathbb{H})}(\mathscr{N}_\beta)\cong \mathscr{M}_{\widetilde{\beta}}$ (see Theorem \ref{theo.FusionModulesInducedAction}) hence $\text{Ind}^{\text{R}(\mathbb{G})}_{\text{R}(\mathbb{H})}\Big(\text{Fus}(\mathscr{N}_\beta)\Big)\cong \text{Fus}\Big(\mathscr{M}_{\widetilde{\beta}}\Big)$. As an application, we obtain stability properties of torsion-freeness with respect to relevant constructions of quantum groups such as quantum (semi-)direct products, compact bicrossed products and quantum free products (see Proposition \ref{pro.DivisibleConstructions} and Corollary \ref{cor.StabilityTorsionFree}).

On the other hand, it is well-known that the \emph{strong} Baum-Connes property is preserved by \emph{divisible} discrete quantum subgroups whenever $\widehat{\mathbb{G}}$ is torsion-free \cite{VoigtBaumConnesUnitaryFree}. Now, thanks to the analysis carried out in the present paper we can simply drop both the torsion-freeness and the divisibility assumptions (notice that this property appeared already in \cite{RubenAmauryTorsion} as a remark). In addition, thanks to the recent general formulation of the quantum assembly map we can as well generalize this stability result for the Baum-Connes property (see Corollary \ref{cor.BCQuantumSubgroups}).

\bigskip
\textsc{Acknowledgements.} The author would like to thank R. Vergnioux for reading the preliminary versions of this paper. He is grateful to K. De Commer for his valuable comments and suggestions that allowed to improve this paper. He also wants to thank M. Yamashita for interesting discussions around the construction of relative tensor product of module C$^*$-categories. He is grateful to the anonymous referee for their valuable comments.

\section{\textsc{Preliminaries}}\label{sec.Background}
	\subsection{Notations and conventions}\label{sec.Notations}
	
	
	We denote by $\mathcal{B}(H)$ (resp. $\mathcal{K}(H)$) the space of all linear (resp. compact) operators of the Hilbert space $H$ and by $\mathcal{L}_{A}(H)$ (resp. $\mathcal{K}_A(H)$) the space of all (resp. compact) adjointable operators of the Hilbert $A$-module $H$. We use the notation $\mathcal{L}_{\mathbb{G}}(H)$ (resp. $\mathcal{K}_\mathbb{G}(H)$) for the \emph{equivariant} (resp. compact) adjointable operators of the Hilbert $A$-module $H$ with respect to a compact (quantum) group $\mathbb{G}$. All our C$^*$-algebras (except for obvious exceptions such as multiplier C$^*$-algebras and von
Neumann algebras) are supposed to be \emph{separable} and all our Hilbert modules are supposed to be \emph{countably generated}. Hilbert $A$-modules are considered to be \emph{right $A$-modules}, so that the corresponding inner products are considered to be conjugate-linear on the left and linear on the right. We use systematically the leg and Sweedler notations. The symbol $\otimes$ stands for the minimal tensor product of C$^*$-algebras and the exterior/interior tensor product of Hilbert modules depending on the context. The symbol $\underset{\max}{\otimes}$ stands for the maximal tensor product of C$^*$-algebras. If $M$ and $N$ are two $R$-modules for some ring $R$, the symbol $\underset{R}{\odot}$ stands for their algebraic tensor product over $R$, and we write $M\underset{R}{\odot} N$. If $S, A$ are C$^*$-algebras, we denote by $M(A)$ the multiplier algebra of $A$. 
	If $H$ is a finite dimensional Hilbert space and $\{\xi_1,\ldots,\xi_{dim(H)}\}$ is an orthonormal basis of $H$, the associated matrix units in $\mathcal{B}(H)$ are denoted by $\{e_{ij}\}_{i,j=1,\ldots,dim(H)}$. The coordinate linear forms on $\mathcal{B}(H)$ are denoted by $\omega_{\xi_i,\xi_j}:= \omega_{i,j}$ and defined by $\omega_{\xi_i,\xi_j}(T):=\langle  \xi_i, T(\xi_j)\rangle$, for all $i,j=1,\ldots, dim(H)$ and all $T\in\mathcal{B}(H)$.
	
	If $\mathbb{G}=(C(\mathbb{G}), \Delta)$ is a compact quantum group, the set of all unitary equivalence classes of irreducible unitary finite dimensional representations of $\mathbb{G}$ is denoted by $\text{Irr}(\mathbb{G})$. The trivial representation of $\mathbb{G}$ is denoted by $\epsilon$. If $x\in \text{Irr}(\mathbb{G})$ is such a class, we write $u^x\in\mathcal{B}(H_x)\otimes C(\mathbb{G})$ for a representative of $x$ and $H_x$ for the finite dimensional Hilbert space on which $u^x$ acts (we write $dim(x):= n_x$ for the dimension of $H_x$). The matrix coefficients of $u^x$ with respect to an orthonormal basis $\{\xi^x_1,\ldots, \xi^x_{n_x}\}$ of $H_x$ are defined by $u^x_{ij}:=(\omega_{ij}\otimes id)(u^x)$ for all $i,j=1,\ldots, n_x$. The linear span of matrix coefficients of all finite dimensional representations of $\mathbb{G}$ is denoted by $\text{Pol}(\mathbb{G})$, which is a Hopf $*$-algebra with co-multipliction $\Delta$ and co-unit and antipode denoted by $\varepsilon_\mathbb{G}$ and $S_\mathbb{G}$, respectively. Given $x,y\in \text{Irr}(\mathbb{G})$, the tensor product of $x$ and $y$ is denoted by $x\otimes y$. Given $x\in \text{Irr}(\mathbb{G})$, we denote by $Q_x:=J^*_xJ_x$ the canonical invertible positive self-adjoint operator such that $dim_q(x):=Tr(Q_x)=Tr(Q_x^{-1})$, where $J_x:H_x\longrightarrow H_{\overline{x}}$ is the antilinear isomorphism associated to a non-trivial invariant vector in $H_x\otimes H_{\overline{x}}$ by virtue of $Mor(\epsilon, x\otimes \overline{x})\neq 0\neq Mor(\epsilon, \overline{x}\otimes x)$. Let $\{\xi^x_1,\ldots, \xi^x_{n_x}\}$ be an orthonormal basis of $H_x$ that diagonalizes $Q_x$ and let $\{\omega^x_1,\ldots, \omega^x_{n_x}\}$ be its dual basis in $\overline{H}_x$. If $\{\xi^{\overline{x}}_1,\ldots, \xi^{\overline{x}}_{n_{\overline{x}}}\}$ is an orthonormal basis of $H_{\overline{x}}$ and $\{\omega^{\overline{x}}_1,\ldots, \omega^{\overline{x}}_{n_{\overline{x}}}\}$ denotes its dual basis in $\overline{H}_{\overline{x}}$ , then we identify systematically $\overline{H}_{\overline{x}}$ and $H_{x}$ via the linear map $\omega^{\overline{x}}_k\mapsto \frac{1}{\sqrt{\lambda^x_k}}\xi^x_k$, for all $k=1,\ldots, n_x$, where $\{\lambda^x_1, \ldots, \lambda^x_{n_x}\}$ are the eigenvalues of $Q_x$. 
	
	The Haar state of $\mathbb{G}$ is denoted by $h_{\mathbb{G}}$. The GNS construction corresponding to $h_{\mathbb{G}}$ is denoted by $(L^2(\mathbb{G}), \lambda, \xi_{\mathbb{G}})$. We also write $\Lambda(x) = \lambda(x)\xi_{\mathbb{G}}$ for $x\in C(\mathbb{G})$. We adopt the standard convention for the inner product on $L^2(\mathbb{G})$, which means that $\langle \Lambda(x), \Lambda(y)\rangle:=h_{\mathbb{G}}(x^*y)$ for all $x,y\in C(\mathbb{G})$. We suppress the notation $\lambda$ in computations so that we simply write $x\Lambda(y)=\Lambda(xy)$ for all $x,y\in C(\mathbb{G})$. We will make the standing assumption that $h_{\mathbb{G}}$ is faithful, so we only work with the reduced form of a compact quantum group, hence $C_r(\mathbb{G})=C(\mathbb{G})$ unless the contrary is specified; where $C_r(\mathbb{G})$ denotes the C$^*$-algebra $\lambda(C(\mathbb{G}))\subset \mathcal{B}(L^2(\mathbb{G}))$. The maximal form of $\mathbb{G}$ is given by the C$^*$-envelopping algebra of $\text{Pol}(\mathbb{G})$, denoted by $C_m(\mathbb{G})$. For more details of these definitions and constructions we refer to \cite{Woronowicz}.

	If $\mathbb{H}$ is another compact quantum group, we say that $\widehat{\mathbb{H}}$ is a discrete quantum subgroup of $\widehat{\mathbb{G}}$ if one (hence all) of the following conditions hold: $i)$ $\text{Pol}(\mathbb{H})$ is a Hopf $*$-subalgebra of $\text{Pol}(\mathbb{G})$, $ii)$ $C_r(\mathbb{H})\overset{\iota}{\subset} C_r(\mathbb{G})$ such that $\iota$ intertwines the co-multiplications, $iii)$ $C_m(\mathbb{H})\overset{\iota}{\subset} C_m(\mathbb{G})$ such that $\iota$ intertwines the co-multiplications; $iv)$ $\mathscr{Rep}(\mathbb{H})$ is a full subcategory of $\mathscr{Rep}(\mathbb{G})$ containing the trivial representation and stable by direct sums, tensor product and adjoint operations. See \cite{SoltanSubgroups} for more details. In this case we write $\widehat{\mathbb{H}}<\widehat{\mathbb{G}}$. Note that in this case we have $\epsilon:=\epsilon_\mathbb{G}=\epsilon_{\mathbb{H}}$. 
	The trivial quantum subgroup of $\widehat{\mathbb{G}}$ is denoted by $\mathbb{E}$. 
	
	\subsubsection*{Quantum semi-direct products.}
	
	Let $\Gamma$ be a discrete group, $\mathbb{G}$ a compact quantum group. Assume that $\Gamma$ acts on $\mathbb{G}$ by quantum automorphisms, $\alpha$. We recall the description of the representation theory of $\mathbb{F}:=\Gamma\underset{\alpha}{\ltimes}\mathbb{G}$ given by S. Wang \cite{WangSemidirect}. Firstly, we have $C(\mathbb{F})=\Gamma\underset{\alpha}{\ltimes} C(\mathbb{G})$ so that $\pi:C(\mathbb{G})\longrightarrow \Gamma\underset{\alpha}{\ltimes} C(\mathbb{G})$ denotes the non-degenerate faithful $*$-homomorphism and $u:\Gamma\longrightarrow \mathcal{U}(\Gamma\underset{\alpha}{\ltimes} C(\mathbb{G}))$ the group homomorphism defining this crossed product. For every irreducible representation $y\in \text{Irr}(\mathbb{F})$, take a representative $w^y\in\mathcal{B}(H_y)\otimes C(\mathbb{F})$. Then there exist (cf. \cite{WangSemidirect}) unique $\gamma\in \Gamma$ and $x\in \text{Irr}(\mathbb{G})$ such that if $w^\gamma\in\mathbb{C}\otimes C^*_r(\Gamma)$ and $w^x\in\mathcal{B}(H_x)\otimes C(\mathbb{G})$ are respective representatives of $\gamma$ and $x$, then we have $w^y\cong v^\gamma\otimes v^x\in\mathcal{B}(\mathbb{C}\otimes H_x)\otimes C(\mathbb{F})\mbox{,}$ where $v^\gamma:=(id\otimes u)(w^\gamma)\in\mathbb{C}\otimes C(\mathbb{F})$ and $v^x:=(id\otimes \pi)(w^x)\in\mathcal{B}(H_x)\otimes C(\mathbb{F})$. Hence we label $\text{Irr}(\mathbb{F})$ with couples $y:=(\gamma, x)$ where $\gamma\in\Gamma$ and $x\in\text{Irr}(\mathbb{G})$ and we define the associated representative to be $w^y:=w^{(\gamma, x)}$.
			
			Since $\alpha$ is an action of $\Gamma$ on $\mathbb{G}$ by quantum automorphisms, then for every $\gamma\in\Gamma$, we have that $(id\otimes \alpha_\gamma)(w^x)$ is an irreducible unitary finite dimensional representation of $\mathbb{G}$ on $H_x$ whenever $x\in \text{Irr}(\mathbb{G})$. Hence there exists a unique class $\alpha_\gamma(x)\in \text{Irr}(\mathbb{G})$ such that $(id\otimes \alpha_\gamma)(w^x)\cong w^{\alpha_\gamma(x)}$. Since $dim(\alpha_\gamma(x))=dim(x)$ we can assume that $w^{\alpha_\gamma(x)}\in\mathcal{B}(H_x)\otimes C(\mathbb{G})$, for all $\gamma\in\Gamma$ (if this is not the case, we might change the representative of $\alpha_\gamma(x)$ by an appropriate one in the orbit of $x$). Hence, there exists a unique, up to a multiplicative factor in $S^1$, unitary operator $V_{\gamma, x}\in \mathcal{U}(H_x)$ such that 
$(id\otimes\alpha_\gamma)(w^x)=(V_{\gamma, x}\otimes id)w^{\alpha_\gamma(x)}(V^*_{\gamma, x}\otimes id)$. Notice that it is clear that $\alpha_e(x)=x$, for all $x\in \text{Irr}(\mathbb{G})$ and that $\alpha_\gamma(\epsilon)=\epsilon$, for all $\gamma\in\Gamma$. Therefore, we can choose the multiplicative factor defining $V_{\gamma, x}$ such that $V_{e, x}=id_{H_x}$, for all $x\in \text{Irr}(\mathbb{G})$ and $V_{\gamma, \epsilon}=1_\mathbb{C}$, for all $\gamma\in\Gamma$. We keep this choice for the sequel.
			
			Let $\gamma,\gamma'\in \Gamma$ and $x,x'\in \text{Irr}(\mathbb{G})$ be irreducible representations of $\Gamma$ and $\mathbb{G}$, respectively. Consider the corresponding irreducible representations of $\mathbb{F}$, say $y:=(\gamma,x), y':=(\gamma', x')\in \text{Irr}(\mathbb{F})$. We know that $w^y=v^\gamma\otimes v^x$ and $w^{y'}=v^{\gamma'}\otimes v^{x'}$, where $v^\gamma:=(id\otimes u)(w^\gamma), v^{\gamma'}:=(id\otimes u)(w^{\gamma'})\in\mathbb{C}\otimes C(\mathbb{F})$ and $v^x:=(id\otimes \pi)(w^x), v^{x'}:=(id\otimes \pi)(w^{x'})\in\mathcal{B}(H_x)\otimes C(\mathbb{F})$. A straightforward computation yields the following formula $w^{y\otimes y'}=v^{\gamma\gamma'}\otimes\big((V_{\gamma'^{-1}}\otimes id)v^{\alpha_{\gamma'^{-1}}(x)}(V^*_{\gamma'^{-1}}\otimes id)\otimes v^{x'}\big)$.
	
	\subsubsection*{Compact bicrossed products.}
	
	Let $(\Gamma,G)$ be a matched pair of a discrete group $\Gamma$ and a compact group $G$. Following \cite{FimaBiproduit}, \cite{FimaHuaBiproduit}, this means that we have a left action $\Gamma \overset{\alpha}{\curvearrowright} G$ of $\Gamma$ on the \textit{compact space} $G$ and a (continuous) right action $\Gamma \overset{\beta}{\curvearrowleft} G$ of $G$ on the \textit{discrete space} $\Gamma$ satisfying:
\begin{equation}\label{EqBicrossed}
\begin{split}
&\alpha_r(gh)=\alpha_r(g)\alpha_{\beta_g(r)}(h)\mbox{and } \alpha_r(e)=e\mbox{, $\forall r,s\in\Gamma$}\\
&\beta_g(rs)=\beta_{\alpha_s(g)}(r)\beta_g(s)\mbox{ and }\beta_g(e)=e\mbox{, $\forall g,h\in G$;}
\end{split}
\end{equation}
where $e$ denotes either the identity element in $\Gamma$ or in $G$. Of course, we have $\alpha_{e}=id_G$ and $\beta_{e}=id_{\Gamma}$. Observe that the above relations imply that $\#[e]=1$, where $[e]\in\Gamma/G$ is the corresponding class in the orbit space. We denote by $\mathbb{F}:=\Gamma {_\alpha}\bowtie_{\beta}G$ the associated bicrossed product. Let us recall the description of the representation theory of $\mathbb{F}$ given in \cite{FimaHuaBiproduit}. For every class $\gamma\cdot G\equiv[\gamma]\in \Gamma/G$ in the orbit space, we define the \emph{clopen} subsets of $G$, $G_{r,s}:=\{g\in G : \beta_g(r)=s\}$, for every $r,s\in \gamma\cdot G$. Consider as well its characteristic function, say $\mathbbold{1}_{A_{r,s}}=: \mathbbold{1}_{r,s}$, for all $r,s\in \gamma\cdot G$. We can show that $v_\gamma:=\Big(\mathbbold{1}_{r,s}\Big)_{r,s\in \gamma\cdot G}\in \mathcal{M}_{|\gamma\cdot G|}(\mathbb{C})\otimes C(G)$ is a magic unitary and a unitary representation of $G$.
Given $\gamma\in\Gamma$, we denote by $G_\gamma:=\{g\in G\,:\,\beta_g(\gamma)=\gamma\}$ the stabilizer of $\gamma$ for the action $\Gamma \overset{\beta}{\curvearrowleft} G$. Note that $G_\gamma$ is a clopen subgroup of $G$ with index $\vert\gamma\cdot G\vert$. We view $C(G_\gamma)=v_{\gamma\gamma}C(G)\subset C(G)$ as a non-unital C*-subalgebra. Let us denote by $\nu$ the Haar probability measure on $G$ and note that $\nu(G_\gamma)=\frac{1}{\vert\gamma \cdot G\vert}$ so that the Haar probability measure $\nu_{\gamma}$ on $G_\gamma$ is given by $\nu_\gamma(A)=\vert\gamma\cdot G\vert\,\nu(A)$ for all Borel subset $A$ of $G_\gamma$. Let us note that, by the bicrossed product relations, $\alpha_\gamma$ defines a group isomorphism (and an homeomorphism) from $G_\gamma$ to $G_{\gamma^{-1}}$, for all $\gamma\in\Gamma$. Given $\gamma\in\Gamma$ we fix a section, still denoted $\gamma$, $\gamma\,:\,\gamma\cdot G\rightarrow G$ of the canonical surjection $G\longrightarrow\gamma\cdot G\,:\,g\mapsto\beta_g(\gamma)$. This means that $\gamma\,:\,\gamma\cdot G\longrightarrow G$ is an injective map such that $\gamma\cdot\gamma(r)=r$ for all $r\in\gamma\cdot G$. We chose the section $\gamma$ such that $\gamma(\gamma)=e$, for all $\gamma\in\Gamma$. For $\gamma\in\Gamma$ and $r,s\in\gamma\cdot G$, we denote by $\psi^\gamma_{r,s}$ the $\nu$-preserving homeomorphism of $G$ defined by $\psi^\gamma_{r,s}(g)=\gamma(r)g\gamma(s)^{-1}$. It follows from our choices that $\psi^\gamma_{\gamma,\gamma}=id$, for all $\gamma\in\Gamma$. Moreover, for all $g\in G$, we have $g\in G_{r,s}\Leftrightarrow\psi^\gamma_{r,s}(g)\in G_\gamma$.

Let $u\,:\,G_\gamma\longrightarrow\mathcal{U}(H)$ be a unitary representation of $G_\gamma$ and view $u$ as a continuous function $G\longrightarrow\mathcal{B}(H)$ which is zero outside $G_\gamma$ i.e. a partial isometry in $\mathcal{B}(H)\otimes C(G)$ such that $uu^*=u^*u=id_H\otimes \mathbbold{1}_{\gamma\gamma}$. 
In the sequel we view $C(G)\subset C(\mathbb{F})$ so that $u\circ\psi^\gamma_{r,s}:=(g\mapsto u(\psi^\gamma_{r,s}(g)))\in\mathcal{B}(H)\otimes C(G)\subset\mathcal{B}(H)\otimes C(\mathbb{F})$. We define:

$$\gamma(u):=\sum_{r,s\in\gamma\cdot G}e_{rs}\otimes(1\otimes u_r\mathbbold{1}_{rs})u\circ\psi^\gamma_{r,s}\in\mathcal{B}(l^2(\gamma\cdot G))\otimes \mathcal{B}(H)\otimes C(\mathbb{F}),$$
where $e_{rs}\in\mathcal{B}(l^2(\gamma\cdot G))$, $r,s\in\gamma\cdot G$, are the matrix units associated to the canonical basis of $l^2(\gamma\cdot G)$. It is known from \cite{FimaHuaBiproduit} that , for any $\gamma\in\Gamma$ and any unitary representation $u$ of $G_\gamma$, $\gamma(u)$ is a unitary representation of $\mathbb{F}$. It is easy to see that its character is given by $\chi(\gamma(u))=\underset{r\in\gamma\cdot G}{\sum}u_r\mathbbold{1}_{rr}\chi(u)\circ\psi^\gamma_{r,r}$. Using the formula for the character we find easily that, for any $\gamma,\mu\in\Gamma$ and any unitary representation $u$ and $v$ of $G_\gamma$ and $G_\mu$, respectively one has:
\begin{equation}\label{EqDim}
{\rm dim}_\mathbb{G}(\gamma(u),\mu(v))=\delta_{\gamma\cdot G,\mu\cdot G}{\rm dim}_{G_\gamma}(u,v\circ\psi_{\gamma,\gamma}^\mu).
\end{equation}
Note that the formula above makes sense since, for any $\mu\in\gamma\cdot G$, $\psi_{\gamma,\gamma}^\mu$ is an isomorphism (and homeomorphism) from $G_\gamma$ to $G_\mu$ so that $v\circ\psi_{\gamma,\gamma}^\mu$ is a unitary representation of $G_\gamma$. In particular, it implies that any representation of the form $\gamma(u)$ is irreducible if and only if $u$ is irreducible. It is also easy to deduce from the character formula that $\overline{\gamma(u)}\simeq\gamma^{-1}(\overline{u}\circ\alpha_{\gamma^{-1}})$. 
Moreover, one can show that any irreducible unitary representation of $\mathbb{F}$ is equivalent to some $\gamma(u)$. 

Finally, the fusion rules are described as follows. Let $\gamma,\mu\in\Gamma$ and $u\,:\, G_\gamma\longrightarrow\mathcal{U}(H_u)$, $v\,:\, G_\mu\longrightarrow\mathcal{U}(H_v)$ be unitary representations of $G_\gamma$ and $G_\mu$, respectively. For any $r\in(\gamma\cdot G)(\mu\cdot G)$, we define, the $r$-twisted tensor product of $u$ and $v$, denoted $u\underset{r}{\otimes} v$ (see \cite{FimaHuaBiproduit} for the precise formula) as a unitary representation of $G_r$ on $K_r\otimes H_u\otimes H_v$, where $K_r:={\rm Span}(\{e_s\otimes e_t\,:\,s\in\gamma\cdot G\text{ and }t\in\mu\cdot G\text{ such that }st=r\})\subset l^2(\gamma\cdot G)\otimes l^2(\mu\cdot G)$. 
It is shown in \cite{FimaHuaBiproduit} that the formula above defines a unitary representation of $G_r$ and the following fusion formula holds, for all $\gamma_1,\gamma_2,\gamma_2\in\Gamma$ and all $u,v,w$ unitary representations of $G_{\gamma_1}$, $G_{\gamma_2}$ and $G_{\gamma_3}$ respectively:
\begin{equation}\label{eq.FusionBicrossed}
{\rm dim}_\mathbb{G}(\gamma_1(u),\gamma_2(v)\otimes\gamma_3(w))=\left\{\begin{array}{ll}\underset{r\in\gamma_1\cdot G\cap(\gamma_2\cdot G)(\gamma_3\cdot G)}{\sum}\frac{1}{\vert r\cdot G\vert}{\rm dim}_{G_r}(u\circ\psi^{\gamma_1}_{r,r},v\underset{r}{\otimes}w)&\text{if }\gamma_1\cdot G\cap(\gamma_2\cdot G)(\gamma_3\cdot G)\neq\emptyset,\\0&\text{otherwise.}\end{array}\right.
\end{equation}
Note that the formula above makes perfect sense since the map $\psi^{\gamma_1}_{r,r}$ defines a group isomorphism and an homeomorphism from $G_r$ to $G_{\gamma_1}$ for all $r\in \gamma_1\cdot G$ so that $u\circ\psi^{\gamma_1}_{r,r}$ is a unitary representation of $G_r$.

Some further useful remarks for our purpose should be made. Since the inclusion $C(G)\subset C(\mathbb{F})$ intertwines the co-multiplications, we may and will view $\mathscr{Rep}(G)\subset\mathscr{Rep}(\mathbb{F})$. Moreover, we have a map $\Gamma\longrightarrow{\rm Rep}(\mathbb{G})$, $\gamma\mapsto\gamma(\epsilon_\gamma)$, where $\epsilon_\gamma$ denotes the trivial representation of $G_\gamma$. Note that $\gamma(\epsilon_\gamma)$ is an irreducible representation of $\mathbb{F}$ on $l^2(\gamma\cdot G)$, for all $\gamma\in\Gamma$. From Equation $(\ref{EqDim})$, we see that $\gamma(\epsilon_\gamma)\simeq\mu(\epsilon_\mu)$ if and only if $\gamma\cdot G=\mu\cdot G$. 
Note that $\overline{\gamma(\epsilon_\gamma)}\simeq\gamma^{-1}(\epsilon_\gamma)$ and, from the bicrossed product relations, $\gamma^{-1}\cdot G=(\gamma\cdot G)^{-1}$.
	
	\subsubsection*{Quantum direct products.}
	
	Let $\mathbb{G}$ and $\mathbb{H}$ be two compact quantum groups. We recall the description of the representation theory of $\mathbb{F}:=\mathbb{G}\times\mathbb{H}$ given by S. Wang \cite{WangSemidirect}: for every irreducible representation $y\in \text{Irr}(\mathbb{F})$, take a representative $w^y\in \mathcal{B}(H_y)\otimes C(\mathbb{F})$. Then there exist unique irreducible representations $x\in \text{Irr}(\mathbb{G})$ and $z\in \text{Irr}(\mathbb{H})$ such that if $w^x\in \mathcal{B}(H_x)\otimes C(\mathbb{G})$ and $w^z\in \mathcal{B}(H_z)\otimes C(\mathbb{H})$ are respective representatives of $x$ and $z$, then we have $w^y\cong \big[w^x\big]_{13}\big[w^z\big]_{24}\in\mathcal{B}(H_x\otimes H_z)\otimes C(\mathbb{F})\mbox{,}$ where $\big[w^x\big]_{13}$ and $\big[w^z\big]_{24}$ are the corresponding legs of $w^x$ and $w^z$, respectively inside $\mathcal{B}(H_x)\otimes\mathcal{B}(H_z)\otimes C(\mathbb{G})\otimes C(\mathbb{H})$. In this case we write $w^{y}:=w^{(x,z)}$.
			
			Moreover, let $x,x'\in \text{Irr}(\mathbb{G})$ and $z,z'\in \text{Irr}(\mathbb{H})$ be irreducible representations of $\mathbb{G}$ and $\mathbb{H}$, respectively. Consider the corresponding irreducible representations of $\mathbb{F}$, say $y:=(x,z), y':=(x', z')\in \text{Irr}(\mathbb{F})$. We know that $w^y= \big[w^x\big]_{13}\big[w^z\big]_{24}$ and $w^{y'}= \big[w^{x'}\big]_{13}\big[w^{z'}\big]_{24}$, where the legs are considered inside $\mathcal{B}(H_x)\otimes\mathcal{B}(H_z)\otimes C(\mathbb{G})\otimes C(\mathbb{H})$ and $\mathcal{B}(H_{x'})\otimes\mathcal{B}(H_{z'})\otimes C(\mathbb{G})\otimes C(\mathbb{H})$, respectively. A straightforward computation yields that the flip map $H_z\otimes H_{x'}\longrightarrow H_{x'}\otimes H_z$ yields the following obvious identification $w^{y\otimes y'}=\big[w^{x}\otimes w^{x'}\big]_{13}\big[w^z\otimes w^{z'}\big]_{24}$.
	
	\subsubsection*{Quantum free products.}
	
	Let $\mathbb{G}$ and $\mathbb{H}$ be two compact quantum groups. We recall the description of the representation theory of $\mathbb{F}:=\mathbb{G}\ast \mathbb{H}$ given by S. Wang \cite{WangFreeProduct}: for every irreducible representation $y\in \text{Irr}(\mathbb{F})$, take a representative $w^y\in \mathcal{B}(H_y)\otimes C(\mathbb{F})$. Then there exist a natural number $n\in\mathbb{N}$ and irreducible representations $\zeta_1,\ldots, \zeta_n$ either in $\text{Irr}(\mathbb{G})$ or in $\text{Irr}(\mathbb{H})$ such that if $w^{\zeta_i}\in\mathcal{B}(H_i)\otimes C(\mathbb{G} \mbox{ or }\mathbb{H})$ are respective representatives of $\zeta_i$, for all $i=1,\ldots, n$, then we have $w^y\cong w^{\zeta_{1}}\otimes w^{\zeta_{2}}\otimes\ldots\otimes w^{\zeta_{n}}\in\mathcal{B}(H_{1}\otimes\ldots H_{n})\otimes C(\mathbb{F})$ with the $\zeta_i$ alternating between $\text{Irr}(\mathbb{G})$ and $\text{Irr}(\mathbb{H})$.
	
		In addition, the fusion rules are described in the following way.
			\begin{enumerate}[a)]
				\item If $y, y'\in \text{Irr}(\mathbb{F})=\text{Irr}(\mathbb{G})\ast \text{Irr}(\mathbb{H})$ are words such that $y$ ends in $\text{Irr}(\mathbb{G})$ and $y'$ starts in $\text{Irr}(\mathbb{H})$ (or \emph{vice-versa}), then $y\otimes y'=yy'$ is an irreducible representation of $\mathbb{F}$.  
				\item If $y=\zeta x, y'=x'\zeta'\in \text{Irr}(\mathbb{F})=\text{Irr}(\mathbb{G})\ast \text{Irr}(\mathbb{H})$ are words such that $x,x'\in \text{Irr}(\mathbb{G})$ (or in $\text{Irr}(\mathbb{H})$), then $y\otimes y'=\underset{t\subset x\otimes x'}{\bigoplus} \zeta t\zeta'\oplus \delta_{\overline{x}, x'}(\zeta\otimes \zeta')$, where the sum runs over all non-trivial irreducible representations $t\in \text{Irr}(\mathbb{G})$ (or in $\text{Irr}(\mathbb{H})$) contained in $x\otimes x'$ with multiplicity.
			\end{enumerate}

	\subsection{C$^*$-categories}\label{sec.ModCTensorCat}
	Let us gather together some basic material about C$^*$-categories used for our purposes. We refer to \cite{Kassel}, \cite{Sergey} or \cite{Gelaki} for further precisions about C$^*$-tensor categories, module C$^*$-categories and examples. The data defining a C$^*$-tensor category is denoted by $(\mathscr{C}, *, \otimes, \mathbbold{1}, \alpha, l, r)$, where $(\mathscr{C}, *)$ is a C$^*$-category, $\mathbbold{1}$ is the unit object, $\alpha$ is the associativity constraint, $l$ is the left unit constraint and $r$ is the right unit constraint. Then, the data defining a left (resp. right) $\mathscr{C}$-module C$^*$-category is denoted by $(\mathscr{M}, \bullet, \mu, e)$, where $\mathscr{M}$ is a C$^*$-category, $\bullet$ denotes the left (resp. right) action of $\mathscr{C}$ on $\mathscr{M}$, $\mu$ is the associativity constraint with respect to $\bullet$ and $e$ is the unit constraint with respect to $\bullet$. A linear functor between two C$^*$-categories that preserves the $*$-operation is called \emph{C$^*$-functor}. A \emph{module functor} is a C$^*$-functor $F$ between two module C$^*$-categories that preserves the module action, i.e. equipped with unitary natural isomorphisms $\phi_{U, X}:F(U\bullet X)\cong U\bullet F(X)$ for all $U\in\text{Obj}(\mathscr{C})$ and $X\in\text{Obj}(\mathscr{M})$ satisfying coherence diagrams with respect to $e$ and $\mu$.
	
	\begin{note}\label{note.AssumptionsCtensorcat}
		If $(\mathscr{C}, *, \otimes, \mathbbold{1}, \alpha, l, r)$ is a C$^*$-tensor category, we assume the following properties.
	\begin{enumerate}[i)]
		\item The unit object $\mathbbold{1}$ is simple (or irreducible), i.e. $End_{\mathscr{C}}(\mathbbold{1})=\mathbb{C}$.
		\item	It is well-known that every C$^*$-tensor category is unitarily monoidally equivalent to a strict C$^*$-tensor category (result due to Mac Lane and we refer to \cite[Theorem XI.5.3]{Kassel} for a proof). Hence, from now on we assume that $\mathscr{C}$ is strict meaning that the natural equivalences $\alpha$, $l$ and $r$ are identities.
		\item\label{assump.DirectSum} $\mathscr{C}$ has orthogonal finite direct sums. More precisely, given objects $U_1,\ldots, U_n\in Obj(\mathscr{C})$, there exists an object $S\in Obj(\mathscr{C})$ and isometries $u_i\in \text{Hom}_{\mathscr{C}}(U_i, S)$ for each $i=1,\ldots, n$ such that $\overset{n}{\underset{i=1}{\sum}}u_iu^*_i=id_S$ and $u_iu^*_j=\delta_{ij}$, for all $i,j=1,\ldots, n$.
		\item\label{assump.Subobjects} $\mathscr{C}$ has subobjects or retracts. More precisely, for any object $U\in Obj(\mathscr{C})$ and for any projection $p\in End_{\mathscr{C}}(U)$, there exists an object $V\in Obj(\mathscr{C})$ and an isometry $u\in \text{Hom}_{\mathscr{C}}(V, U)$ such that $p=uu^*$. In particular, $\mathscr{C}$ has a zero object. Namely, the object defined by the zero projection.
		
		\emph{Observe that different terminologies are used in the category theory literature for this property. For instance, we also say that $\mathscr{C}$ is subobject/idempotent/Cauchy/Karoubi complete}.
	\end{enumerate}
	\end{note}
	\begin{rem}\label{rem.AssumptionsSemiSimple}			
		Assume that $\mathscr{C}$ is \emph{rigid}, i.e. every object of $\mathscr{C}$ has a conjugate object. This is the case for the C$^*$-categories used for our purpose. In this case, one can show that assumptions (\ref{assump.DirectSum}) and (\ref{assump.Subobjects}) in Note \ref{note.AssumptionsCtensorcat} are equivalent to say that all homomorphism spaces are finite-dimensional and that every object in $\mathscr{C}$ is isomorphic to a finite direct sum of simple objects (cf. \cite{Sergey} for more details). In other words, $\mathscr{C}$ is semi-simple. We mainly work with semi-simple C$^*$-categories, so we assume from now on that $\mathscr{C}$ is semi-simple unless otherwise stated. A module C$^*$-category is called \emph{semi-simple} if the underlying C$^*$-category is semi-simple. The (module) C$^*$-categories used for our purpose are semi-simple.
	\end{rem}
	
	The following are relevant examples of C$^*$-tensor categories for our purpose.
	\begin{exs}
		\begin{enumerate}
			\item The cateogory of all Hilbert spaces and bounded linear maps is a C$^*$-tensor category for the ordinary tensor product of Hilbert spaces, $\mathbb{C}$ being the unit object. It is denoted by $\mathscr{H}ilb$. The corresponding rigid subcategory consists of all finite dimensional Hilberts spaces, denoted by $\mathscr{H}ilb_f$.
			
			\item If $\mathbb{G}$ is a compact quantum group, the category of all its finite dimensional unitary representations and intertwiners is a C$^*$-tensor category for the usual tensor product of representations of $\mathbb{G}$, $\epsilon$ being the unit object. It is denoted by $\mathscr{Rep}(\mathbb{G})$ and it is rigid \cite{Sergey}. The category $\mathscr{Rep}(\mathbb{G})$ satisfies all assumptions considered in Note \ref{note.AssumptionsCtensorcat}; in particular, it is semi-simple (cf. Remark \ref{rem.AssumptionsSemiSimple}).
			\begin{note}
				Given a subset $\mathcal{S}\subset \text{Irr}(\mathbb{G})$, we denote by $\mathscr{C}:=\langle\mathcal{S}\rangle$ the smallest full subcategory of $\mathscr{Rep}(\mathbb{G})$ containing $\mathcal{S}$. If, in addition, $\mathscr{C}$ contains the trivial representation and it is closed under taking tensor product and contragredient representations, by Tannaka-Krein-Woronowicz duality, there is an associated C$^*$-subalgebra $C(\mathbb{H})$ such that restricting the coproduct to $C(\mathbb{H})$ endows it with the structure of compact quantum group $\mathbb{H}$. Moreover, $\mathscr{Rep}(\mathbb{H})$ naturally identifies with $\mathscr{C}$ and we say that \emph{$\widehat{\mathbb{H}}$ is the quantum subgroup of $\widehat{\mathbb{G}}$ generated by $\mathcal{S}$}.
			\end{note}
			
			\item Let $\mathscr{C}$, $\mathscr{D}$ be C$^*$-categories. Denote by $\mathscr{F}un(\mathscr{C}, \mathscr{D})$ the category whose objects are the C$^*$-functors between $\mathscr{C}$ and $\mathscr{D}$ and whose homomorphisms are the natural transformations between C$^*$-functors, which are uniformly bounded.
			Then $\mathscr{F}un(\mathscr{C}, \mathscr{D})$ is a C$^*$-category in the following way. The involution $*:\mathscr{F}un(\mathscr{C}, \mathscr{D})\longrightarrow \mathscr{F}un(\mathscr{C}, \mathscr{D})$ given by the identity on objects and $(\eta^*)_U:=(\eta_U)^*$, for all \\$\eta\in \text{Hom}_{\mathscr{F}un(\mathscr{C}, \mathscr{D})}(F, F')$ with $F, F'\in Obj\Big(\mathscr{F}un(\mathscr{C}, \mathscr{D})\Big)$ and all $U\in Obj(\mathscr{C})$. For every objects $F, F'\in Obj\Big(\mathscr{F}un(\mathscr{C}, \mathscr{D})\Big)$, the homomorphism space $\text{Hom}_{\mathscr{F}un(\mathscr{C}, \mathscr{D})}(F, F')$ is equipped with the following norm $||\eta||:=\sup\{||\eta_U||\ |\ U\in Obj(\mathscr{C})\}\mbox{,}$ for all $\eta\in \text{Hom}_{\mathscr{F}un(\mathscr{C}, \mathscr{D})}(F, F')$ and all $U\in Obj(\mathscr{C})$.
			
			In particular, if $\mathscr{C}=\mathscr{D}=:\mathscr{M}$, then we write $\mathscr{E}nd(\mathscr{M}):=\mathscr{F}un(\mathscr{M}, \mathscr{M})$. In this case, it is a C$^*$-tensor category with tensor product $\otimes$ given by the composition of functors, the identity functor being the unit object. The corresponding rigid subcategory consists of adjointable functors whose unit and co-unit maps are uniformly bounded.
		\end{enumerate}
	\end{exs}
	
	The following are relevant examples of module C$^*$-categories, which are used in Section \ref{sec.StabilityTorsionFreeness}.
	\begin{exs}\label{exs.ModuleCategories}
		\begin{enumerate}
			\item If $\mathscr{C}$ is a C$^*$-tensor category, then it is a $\mathscr{C}$-bimodule C$^*$category with left and right actions given simply by its own tensor product $\underset{l}{\bullet}:=\otimes=:\underset{r}{\bullet}$.
			\item\label{ex.ModCatQG} Let $\mathscr{C}$ and $\mathscr{D}$ be C$^*$-tensor categories. If $J:\mathscr{C}\longrightarrow \mathscr{D}$ is a C$^*$-tensor functor, then $\mathscr{D}$ is a right $\mathscr{C}$-module C$^*$-category with the following action $P\bullet U:=P\otimes J(U)\mbox{,}$ for all objects $U\in\mathscr{C}$ and $P\in\mathscr{D}$ and analogously defined on homomorphisms. In particular, let $\mathbb{G}$ and $\mathbb{H}$ be compact quantum groups.
			\begin{enumerate}[a)]
				\item If $\mathbb{H}\leq \mathbb{G}$ with canonical surjection $\rho: C_m(\mathbb{G})\twoheadrightarrow C_m(\mathbb{H})$, then we have a restriction functor between C$^*$-tensor categories, $J:\mathscr{Rep}(\mathbb{G})\longrightarrow\mathscr{Rep}(\mathbb{H})$. In this way, $\mathscr{Rep}(\mathbb{H})$ is a right $\mathscr{Rep}(\mathbb{G})$-module C$^*$-category.
				\item If $\widehat{\mathbb{H}}\leq \widehat{\mathbb{G}}$, then we have a fully faithful functor between C$^*$-tensor categories, $J:\mathscr{Rep}(\mathbb{H})\longrightarrow\mathscr{Rep}(\mathbb{G})$ given by the natural inclusion of $\mathscr{Rep}(\mathbb{H})$ inside $\mathscr{Rep}(\mathbb{G})$ as a full subcategory. In this way, $\mathscr{Rep}(\mathbb{G})$ is a right $\mathscr{Rep}(\mathbb{H})$-module C$^*$-category.
			\end{enumerate}
			\item Let $\mathscr{C}$ is a C$^*$-tensor category. If $\mathscr{M}$ and $\mathscr{N}$ are a left $\mathscr{C}$-module C$^*$categories, then we denote by $\mathscr{F}un_{\mathscr{C}}(\mathscr{M}, \mathscr{N})$ the category whose objects are the $\mathscr{C}$-module functors between $\mathscr{M}$ and $\mathscr{N}$ and whose homomorphisms are the $\mathscr{C}$-module natural transformations, which are uniformly bounded. Then \\$\mathscr{F}un_{\mathscr{C}}(\mathscr{M}, \mathscr{N})$ is a left $\mathscr{C}$-module C$^*$-category in the following way. The involution $*:\mathscr{F}un_{\mathscr{C}}(\mathscr{M}, \mathscr{N})\longrightarrow\mathscr{F}un_{\mathscr{C}}(\mathscr{M}, \mathscr{N})$ given by the identity on objects and $(\eta^*)_X:=(\eta_X)^*$, for all $\eta\in \text{Hom}_{\mathscr{F}un_{\mathscr{C}}(\mathscr{M}, \mathscr{N})}(F, F')$ with $F, F'\in Obj\Big(\mathscr{F}un_{\mathscr{C}}(\mathscr{M}, \mathscr{N})\Big)$ and all $X\in Obj(\mathscr{M})$. For every objects $F, F'\in Obj\Big(\mathscr{F}un_{\mathscr{C}}(\mathscr{M}, \mathscr{N})\Big)$, the homomorphism space $\text{Hom}_{\mathscr{F}un_{\mathscr{C}}(\mathscr{M}, \mathscr{N})}(F, F')$ is equipped with the following norm $||\eta||:=\sup\{||\eta_X||\ |\ X\in Obj(\mathscr{M})\}\mbox{,}$ for all $\eta\in \text{Hom}_{\mathscr{F}un_{\mathscr{C}}(\mathscr{M}, \mathscr{N})}(F, F')$ and all $X\in Obj(\mathscr{M})$. The left action of $\mathscr{C}$ on $\mathscr{F}un_{\mathscr{C}}(\mathscr{M}, \mathscr{N})$ is given by $U\bullet F\ (X):=U\bullet\ F(X)\mbox{,}$ for all $F\in Obj(\mathscr{F}un_{\mathscr{C}}(\mathscr{M}, \mathscr{N}))$, all $U\in Obj(\mathscr{C})$ and all $X\in Obj(\mathscr{M})$; and analogously defined on homomorphisms. Both the associativity constraint and the unit constraint with respect to $\bullet$ are inherited from those on $\mathscr{N}$.
				
			In particular, we write $\mathscr{M}^*:=\mathscr{F}un_{\mathscr{C}}(\mathscr{M}, \mathscr{C})$ and it is called the \emph{dual category of $\mathscr{M}$}.			
		\end{enumerate}
	\end{exs}

	\subsection{Fusion rings and (\emph{strong}) torsion-freeness}\label{sec.Torsion}
		\begin{defi}
			Let $\mathbb{G}=(C(\mathbb{G}),\Delta)$ be a compact quantum group. A right $\mathbb{G}$-C$^*$-algebra is a C$^*$-algebra $A$ together with an injective non-degenerate $*$-homomorphism $\delta: A\longrightarrow M(A\otimes C(\mathbb{G}))$ such that: $i)$ $(\delta\otimes id_{C(\mathbb{G})})\circ\delta = (id_A\otimes \Delta)\circ\delta$ and $ii)$ $[\delta(A)(1\otimes C(\mathbb{G}))]=A\otimes C(\mathbb{G})$.
		Such homomorphism is called a \emph{right action of $\mathbb{G}$ on $A$} or a \emph{right co-action of $C(\mathbb{G})$ on $A$}.
		
		 In particular, if $A$ is finite dimensional and $\delta$ is ergodic, meaning that $A^\delta:=\{a\in A\ |\ \delta(a)=a\otimes 1_{C(\mathbb{G})}\}=\mathbb{C}1_A$, then we say that $(A,\delta)$ is a torsion action of $\mathbb{G}$. The set of all equivariant Morita equivalence classes of torsion actions of $\mathbb{G}$ is denoted by $\text{Tor}(\widehat{\mathbb{G}})$.
	\end{defi}
	\begin{note}
		Firstly, similar definitions can be made for \emph{left} actions of $\mathbb{G}$. In the present paper we will consider all actions to be right ones unless the contrary is explicitly indicated, so we will refer to them simply as \emph{action of $\mathbb{G}$}. Similar considerations are made for the following definitions. Secondly, since we are mainly interested in studying torsion actions, we implicitly assume that our C$^*$-algebras are unital unless the contrary is explicitly indicated.
	\end{note}
	\begin{rem}
		Given a $\mathbb{G}$-C$^*$-algebra $(A, \delta)$, there always exists a non-degenerate $\delta$-invariant conditional expectation $E_\delta:A\longrightarrow A^\delta$ given by $a\mapsto (id\otimes h_\mathbb{G})\delta(a)$ for all $a\in A$. $E_\delta$ is a state whenever $\delta$ is ergodic. Recall that we only work with the reduced form of $\mathbb{G}$, so $E_\delta$ is automatically faithful.
	\end{rem}
	\begin{defi}
		Let $\mathbb{G}=(C(\mathbb{G}),\Delta)$ be a compact quantum group and $\mathcal{A}$ a unital $*$-algebra. A right co-action of $\text{Pol}(\mathbb{G})$ on $\mathcal{A}$ is an algebra homomorphism $\delta:\mathcal{A}\longrightarrow \mathcal{A} \odot \text{Pol}(\mathbb{G})$ such that: $i)$ $(\delta\otimes id_{\text{Pol}(\mathbb{G})})\circ\delta = (id_\mathcal{A}\otimes \Delta)\circ\delta$ and $ii)$ $(id_{\mathcal{A}}\otimes \varepsilon_{\mathbb{G}})\circ \delta=id_{\mathcal{A}}$.
	\end{defi}
	\begin{defi}
		Let $\mathbb{G}=(C(\mathbb{G}),\Delta)$ be a compact quantum group. An algebraic right $\mathbb{G}$-$*$-algebra is a unital $*$-algebra $\mathcal{A}$ together with a right co-action of $\text{Pol}(\mathbb{G})$ on $\mathcal{A}$, $\delta:\mathcal{A}\longrightarrow \mathcal{A} \odot \text{Pol}(\mathbb{G})$, such that $\mathcal{A}^\delta$ is a unital C$^*$-algebra, and the canonical $\delta$-invariant conditional expectation $E_\delta:\mathcal{A}\longrightarrow \mathcal{A}^\delta$ is completely positive. Such co-action is called a \emph{(right) algebraic action of $\mathbb{G}$ on $\mathcal{A}$}.
	\end{defi}
	
	\begin{theo}[{\cite[Proposition 4.5]{KennyMakoto}}]\label{theo.AlgebraicActionsG}
		Let $\mathbb{G}=(C(\mathbb{G}),\Delta)$ be a compact quantum group.
		\begin{enumerate}[i)]
			\item If $(A,\delta)$ is a unital right $\mathbb{G}$-C$^*$-algebra, then $\mathcal{A}:=span\{(id_A\otimes h_\mathbb{G})\big(\delta(a)(1_A\otimes x)\big)\ |\ a\in A, x\in \text{Pol}(\mathbb{G})\}$ is a dense unital $*$-subalgebra of $A$ on which $\delta$ restricts to an algebraic action of $\mathbb{G}$. This association defines a functor from the category of $\mathbb{G}$-C$^*$-algebras to the one of algebraic $\mathbb{G}$-$*$-algebras. We denote it by $\textbf{Alg}$. 
			\item If $(\mathcal{A}, \delta)$ is an algebraic right $\mathbb{G}$-$*$-algebra, then there exists a unique C$^*$-completion of $\mathcal{A}$, say $A$, to which $\delta$ extends as a right action of $\mathbb{G}$. Moreover, $\mathcal{A}^\delta=A^\delta$. This association defines a functor from the category of algebraic $\mathbb{G}$-$*$-algebras and the one of $\mathbb{G}$-C$^*$-algebras. We denote it by $\textbf{Comp}$.
		\end{enumerate}
		
		In addition, the composition $\textbf{Comp}\circ \textbf{Alg}$ is naturally equivalent to the identity functor. The composition $\textbf{Alg}\circ \textbf{Comp}$ is naturally equivalent to the identity whenever we restrict to actions with finite dimensional fixed points algebras.
	\end{theo}
	
	\begin{exs}\label{ex.TorsionActions}
		\begin{enumerate}
			\item The trivial action $(\mathbb{C}, trv.)$ is of course a torsion action of any compact quantum group $\mathbb{G}$. The co-multiplication of any compact quantum group $\mathbb{G}$ defines an action of $\mathbb{G}$ on its defining C$^*$-algebra. This action is called the \emph{regular action of $\mathbb{G}$}.
			\item If $\widehat{\mathbb{H}}<\widehat{\mathbb{G}}$ is a discrete quantum subgroup of $\widehat{\mathbb{G}}$, we have by definition an inclusion of C$^*$-algebras $C(\mathbb{H})\overset{\iota}{\subset} C(\mathbb{G})$ intertwining the corresponding co-multiplications. Therefore, if $(B,\beta)$ is a $\mathbb{H}$-C$^*$-algebra, we can obviously extend $\beta$ (by composing with $\iota$) into an action of $\mathbb{G}$ on $B$, which is denoted by $\widetilde{\beta}$. We denote by $\text{Ind}^{\mathbb{G}}_{\mathbb{H}}(B)$ the same C$^*$-algebra $B$ but equipped with the composition $\widetilde{\beta}:=(id_B\otimes\iota)\circ\beta$ as an action of $\mathbb{G}$. Observe that if $(B,\beta)$ is a torsion action of $\mathbb{H}$, then $(\text{Ind}^{\mathbb{G}}_{\mathbb{H}}(B), \widetilde{\beta})$ is a torsion action of $\mathbb{G}$. In particular, if $\widehat{\mathbb{H}}$ is finite, then $(C(\mathbb{H}), \Delta_{\mathbb{H}})$ defines a non-trivial torsion action of $\mathbb{G}$.
			\item If $\widehat{\mathbb{H}}<\widehat{\mathbb{G}}$ is a discrete quantum subgroup of $\widehat{\mathbb{G}}$, we have by definition an inclusion of C$^*$-algebras $C(\mathbb{H})\overset{\iota}{\subset} C(\mathbb{G})$ intertwining the corresponding co-multiplications. Given a $\mathbb{G}$-C$^*$-algebra $(A, \delta)$ we denote by $\text{Res}^\mathbb{G}_{\mathbb{H}}(A)$ the C$^*$-algebra obtained as $\mathcal{A}':=\overline{span}\{(id_A\otimes h_\mathbb{G})\big(\delta(a)(1_A\otimes y)\big)\ |\ a\in A, y\in \text{Pol}(\mathbb{H})\}$. As in Theorem \ref{theo.AlgebraicActionsG}, $\delta$ ``restricts'' to an action of $\mathbb{H}$ on $\text{Res}^\mathbb{G}_{\mathbb{H}}(A)$, which we denote by $\delta'$. We will show in Proposition \ref{pro.CorrespondenceTorsionActions} that if $(A, \delta)$ is a torsion action of $\mathbb{G}$, then $(\text{Res}^\mathbb{G}_{\mathbb{H}}(A), \delta')$ splits as a direct sum of actions of $\mathbb{H}$ which are $\mathbb{H}$-Morita equivalent to torsion actions of $\mathbb{H}$.
			\item If $u\in \mathcal{B}(H_u)\otimes C(\mathbb{G})$ is a unitary representation of $\mathbb{G}$ on a finite dimensional Hilbert space $H_u$, then it defines an action of $\mathbb{G}$ on $\mathcal{B}(H_u)$ given by $Ad_u:\mathcal{B}(H_u) \longrightarrow \mathcal{B}(H_u)\otimes C(\mathbb{G})$, $T \longmapsto Ad_u(T):=u(T\otimes 1_{C(\mathbb{G})})u^*$. It is clear that $\mathcal{B}(H)^{Ad_u}=End(u)$. Hence, the pair $(\mathcal{B}(H_u), Ad_u)$ is a torsion action of $\mathbb{G}$ if and only if $u$ is irreducible.
		\end{enumerate}
	\end{exs}
	
	\begin{defi}\label{defi.TorsionFree}
		Let $\mathbb{G}$ be a compact quantum group. We say that $\widehat{\mathbb{G}}$ is torsion-free if any torsion action of $\mathbb{G}$ is $\mathbb{G}$-equivariantly Morita equivalent to the trivial $\mathbb{G}$-C$^*$-algebra $\mathbb{C}$.
	\end{defi}

	The notion of equivariant Hilbert module with respect to a compact quantum group is central for our purpose.
	\begin{defi}\label{defi.EquivariantModule}
		Let $\mathbb{G}$ be a compact quantum group and $(A, \delta)$ a $\mathbb{G}$-C$^*$-algebra. A right $\mathbb{G}$-equivariant Hilbert $A$-module is a right $A$-module $E$ together with an injective linear map $\delta_E: E\longrightarrow E\otimes C(\mathbb{G})$ such that: $i)$ $\delta_E(\xi\cdot a)=\delta_E(\xi)\circ\delta(a)$ for all $\xi\in E$ and $a\in A$; $ii)$ $\delta\big(\langle \xi, \eta\rangle\big)=\langle \delta_E(\xi), \delta_E(\eta) \rangle$ for all $\xi, \eta\in E$; $iii)$ $(\delta_E\otimes id)\circ \delta_E=(id\otimes \Delta)\circ \delta_E$; $iv)$ $[\delta_E(E)(A\otimes C(\mathbb{G}))]=E\otimes C(\mathbb{G})$. Such map is called a \emph{right action of $\mathbb{G}$ on $E$} or a \emph{right co-action of $C(\mathbb{G})$ on $E$}.
	\end{defi}
	\begin{ex}\label{ex.ActHilbSpaceUnitRep}
		If $u\in \mathcal{B}(H_u)\otimes C(\mathbb{G})$ is a unitary representation of $\mathbb{G}$ on a finite dimensional Hilbert space $H_u$, then it defines an action of $\mathbb{G}$ on $H_u$ given by $\delta_u(\xi):=u(\xi\otimes id)$, for all $\xi\in H_u$.
	\end{ex}
	\begin{defi}\label{defi.EquivariantModule}
		Let $\mathbb{G}$ be a compact quantum group and $(A, \delta)$ a $\mathbb{G}$-C$^*$-algebra. Let $(E, \delta_E)$ be a $\mathbb{G}$-equivariant Hilbert $A$-module. We say that $E$ is irreducible if the space of equivariant adjointable operators of $E$, $\mathcal{L}_{\mathbb{G}}(E):=\{T\in\mathcal{L}_A(E)\ |\ \delta_E(T(\xi))=(T\otimes 1)\delta_E(\xi)\mbox{, for all $\xi\in E$}\},$ is one-dimensional.
	\end{defi}
	\begin{rem}\label{rem.AdmissibleUnitary}
		If $(E, \delta_E)$ is a $\mathbb{G}$-equivariant Hilbert $A$-module as above, then $\mathcal{K}_A(E)$ is a $\mathbb{G}$-C$^*$-algebra with action $\delta_{\mathcal{K}_A(E)}$ defined by $\delta_{\mathcal{K}_A(E)}(\theta_{\xi, \eta})=\delta_E(\xi)\delta_E(\eta)^*\in \mathcal{K}_A(E)\otimes C(\mathbb{G})$, for all $\xi, \eta\in E$ where $\theta_{\xi, \eta}$ denotes the corresponding rank one operator in $E$. By abuse of notation, we still denote by $\delta_{\mathcal{K}_A(E)}$ the extension of this homomorphism to $\mathcal{L}_A(E) = M(\mathcal{K}_A(E)) \rightarrow M(\mathcal{K}_{A}(E) \otimes C(\mathbb{G}))$. The latter is however not in general an action of $\mathbb{G}$ on $\mathcal{L}_A(E)$. Recall further that giving an action $\delta_E$ is equivalent to giving a unitary operator $V_E\in\mathcal{L}_{A\otimes C(\mathbb{G})}\big(E\underset{\delta}{\otimes}(A\otimes C(\mathbb{G})), E\otimes C(\mathbb{G})\big)$ such that $\delta_E(\xi)=V_E\circ T_\xi$ for all $\xi\in E$ where $T_\xi\in\mathcal{L}_{A\otimes C(\mathbb{G})}(A\otimes C(\mathbb{G}), E\underset{\delta}{\otimes}(A\otimes C(\mathbb{G})))$ is such that $T_\xi(x)=\xi\underset{\delta}{\otimes} x$, for all $x\in A\otimes C(\mathbb{G})$. One calls $V_E$ the \emph{admissible operator for $(E, \delta_E)$}. Moreover, we have $\delta_{\mathcal{K}_A(E)}=Ad_{V_E}$. We refer to \cite{BaajSkandalisQuantumKK} for more details. Note that $\mathcal{L}_{\mathbb{G}}(E)=\mathcal{L}_A(E)^{Ad_{V_E}}$. So, if $(E, \delta_E)$ is irreducible, then $\mathcal{L}_A(E)=\mathcal{K}_A(E)$ together with $Ad_{V_E}$ defines an ergodic action of $\mathbb{G}$.
	\end{rem}

	Y. Arano and K. De Commer \cite{YukiKenny} have re-interpreted the notion of torsion-freeness for discrete quantum groups in terms of fusion rings giving a stronger version of torsion-freeness. Let us recall briefly elementary notions about fusion rings theory (we refer to \cite{YukiKenny} or \cite{Gelaki} for further details and properties).
		
	Let $(I, \mathbbm{1})$ be an involutive pointed set, i.e. a set $I$ equipped with an involution $I\rightarrow I$, $i\mapsto \overline{i}$, such that $\mathbbm{1}=\overline{\mathbbm{1}}$. Let $J$ be any set. Denote by $\mathbb{Z}_I$ the \emph{free $\mathbb{Z}$-module with basis $I$}, that is, every element in $\mathbb{Z}_I$ is a unique finite $\mathbb{Z}$-linear combination of elements of $I$. The addition operation in $\mathbb{Z}_I$ is denoted by $\oplus$. One extends the involution on $I$ to a $\mathbb{Z}$-linear involution on $\mathbb{Z}_I$. A \emph{ring structure $\otimes$} on $\mathbb{Z}_I$ is given by constants $N_{\alpha, i'}^{i}\in \mathbb{N}\cup \{0\}$ for all $\alpha, i, i'\in I$, called \emph{fusion rules}, such that $\alpha\otimes i'=\underset{i\in I}{\sum}N_{\alpha, i'}^{i}\cdot i\mbox{,}$ where all but finitely many terms vanish. This rule extends obviously to any element of $\mathbb{Z}_I$ and it can be regarded as an action of $\mathbb{Z}_I$ on itself; we call it \emph{regular action of $\mathbb{Z}_I$}. We write $i\subset \alpha\otimes i'$ whenever $N_{\alpha, i'}^{i}\neq 0$. Denote by $\mathbb{Z}_J$ the \emph{free $\mathbb{Z}$-module with basis $J$}. The addition operation in $\mathbb{Z}_J$ is still denoted by $\oplus$. A \emph{(left) $\mathbb{Z}_I$-module structure} on $\mathbb{Z}_J$, still denoted by $\otimes$, is given by constants $N^{j}_{\alpha, j'}\in\mathbb{N}\cup \{0\}$ for all $\alpha\in I$, $j, j'\in J$ such that $\alpha\otimes j'=\underset{j\in J}{\sum}N_{\alpha, j'}^{j}\cdot j\mbox{,}$ where all but finitely many terms vanish. This rule extends obviously to any element of $\mathbb{Z}_I$ and $\mathbb{Z}_J$ and it can be regarded as an action of $\mathbb{Z}_I$ on $\mathbb{Z}_j$; we say that $\mathbb{Z}_J$ is a $\mathbb{Z}_I$-module. We write $j\subset \alpha\otimes j'$ whenever $N_{\alpha, j'}^{j}\neq 0$.

	\begin{defi}\label{defi.FusionModules}
		Let $(I, \mathbbm{1})$ be an involutive pointed set and $J$ any set. Let $R:=(\mathbb{Z}_I, \oplus, \otimes)$ be the free $\mathbb{Z}$-module with basis $I$ endowed with a ring structure and let $M:=(\mathbb{Z}_J, \oplus, \otimes)$ be the free $\mathbb{Z}$-module with basis $J$ endowed with a $\mathbb{Z}_I$-module structure.
		\begin{itemize}
			\item[-] We say that $R$ is a $I$-based ring if $\overline{\alpha\otimes\alpha'}=\overline{\alpha'}\otimes\overline{\alpha}$, for all $\alpha, \alpha'\in I$ and $\mathbbold{1}\subset \overline{\alpha}\otimes \alpha'$ if and only if $\alpha=\alpha'$, for all $\alpha, \alpha'\in I$.
			\item[-] We say that $R$ is a fusion ring if it is a $I$-based ring equipped with a dimension function, that is, a unital ring homomorphism $d:\mathbb{Z}_I\longrightarrow\mathbb{R}$ such that $d(\alpha)>0$, for all $\alpha\in I$ and $d(\overline{\alpha})=d(\alpha)$, for all $\alpha\in I$.
			\item[-] We say that $M$ is a $J$-based module if $j\subset \alpha\otimes j' \Leftrightarrow j'\subset \overline{\alpha}\otimes j$, for all $\alpha\in I$, $j, j'\in J$.
			\item[-] A $J$-based module $M$ is said to be co-finite if for all $j, j'\in J$, the set $\{\alpha\in I\ |\ j\subset\alpha\otimes j' \}$ is finite.
			\item[-] A $J$-based module $M$ is said to be connected if for all $j, j'\in J$, there exists $\alpha\in I$ such that $j\subset \alpha\otimes j'$.
			\item[-] A $J$-based module is said to be a torsion module if it is co-finite and connected.
			\item[-] We say that $M$ is a fusion $R$-module if it is $J$-based and it is equipped with a compatible dimension function, that is, a linear map $d_J:\mathbb{Z}_J\longrightarrow\mathbb{R}$ such that $d_J(j)>0$, for all $j\in J$ and $d_J(\alpha\otimes j')=d_I(\alpha)d_J(j')$, for all $\alpha\in I$ and all $j'\in J$.
		\end{itemize}
	\end{defi}
	\begin{note}
		An isomorphism of based modules is assumed to take basis elements to basis elements.
	\end{note}
	\begin{rem}\label{rem.BilinearForm}
		We keep the terminology and notations from the previous definition. We refer to \cite{YukiKenny} for further details about the following observations, which will be useful for later computations. If $R$ is a fusion ring with dimension function $d$ and $M$ is a $J$-based $R$-module, then we can define an $R$-valued $\mathbb{Z}$-bilinear form on $M$ by $\langle j, j'\rangle := \underset{i\in I}{\sum} N^{j'}_{\overline{i}, j}\cdot i\mbox{,}$ for all $j, j'\in J$. Notice that $\langle \alpha\otimes j, j'\rangle=\alpha\otimes \langle j, j'\rangle$ and $\langle j, j'\rangle=\overline{\langle j', j \rangle}$, for any $\alpha\in I$, $j, j'\in J$. Then for any $j_0\in J$, the map $d_J:M\longrightarrow \mathbb{R}$ defined by $d_J(j):=d(\langle j, j_0\rangle)$, for all $j\in J$ is a compatible dimension function on $M$. In particular, for $M=R$ one has $\langle \alpha, \alpha' \rangle = \alpha\otimes \overline{\alpha'}$, for any $\alpha, \alpha'\in I$.
		
		Observe that if we equip $R$ with the $\mathbb{Z}$-linear functional $\tau$ (sometimes denotes by $\tau_R$ to specify the ring on which one considers this functional) such that $\tau(\alpha)=\delta_{\alpha, \mathbbold{1}}$ for all $\alpha \in I$, then one can check that $R$ is a $I$-based ring if and only if $\overline{\alpha\otimes\alpha'}=\overline{\alpha'}\otimes\overline{\alpha}$, $\tau(\overline{\alpha}\otimes \alpha')=\delta_{\alpha, \alpha'}$ and $\tau(\alpha \otimes \beta\otimes \gamma)\in\mathbb{N}$, for all $\alpha, \alpha', \beta, \gamma\in I$. In this case, the structural constants of $R$ are recovered as $N_{\alpha, i'}^{i}=\tau(\alpha \otimes i'\otimes \overline{i})$, for all $\alpha, i, i'\in I$. Moreover, for a $J$-based $R$-module $M$ one has $\tau(\langle j, j'\rangle)=\delta_{j, j'}$, for all $j, j'\in J$ and the structural constants of $M$ are recovered as $N_{\alpha, j'}^{j}=\tau(\alpha\otimes \langle j', j\rangle)$, for all $\alpha\in I$, $j, j'\in J$. In particular, $N^{j'}_{\mathbbold{1}_R, j}=\delta_{j,j'}$.
		
		More generally, one can check that $M$ is a $J$-based $R$-module if and only if there exists a $R$-valued $\mathbb{Z}$-bilinear form on $M$, say $\langle \cdot, \cdot \rangle$, such that $\langle \alpha\otimes j, j'\rangle=\alpha\otimes \langle j, j'\rangle$, $\langle j, j'\rangle=\overline{\langle j', j \rangle}$, $\tau(\langle j, j'\rangle)=\delta_{j, j'}$ and $\tau(\alpha\otimes \langle j, j'\rangle)\in\mathbb{N}$, for all $\alpha\in I$ and all $j, j'\in J$.
	\end{rem}
	\begin{defi}
		Let $R$ be a $I$-based ring and $M$ a $J$-based $R$-module. The stabiliser of an element $j\in J$ is defined by $\text{Stab}(j):=\{\alpha\in I\ |\ j\subset \alpha\otimes j\}$.
	\end{defi}
	
	\begin{exs}
		\begin{enumerate}
			\item The trivial fusion ring is the fusion ring $\mathbb{Z}_I$ with $I=\{\mathbbold{1}\}$.
			\item Any fusion ring $R$ is a fusion $R$-module with its regular action. It is automatically co-finite and connected by definition. The $R$-valued bilinear form on $R$ as introduced in Remark \ref{rem.BilinearForm} is given by $\langle \alpha, \alpha'\rangle =\alpha\otimes \overline{\alpha'}\mbox{,}$ for all $\alpha, \alpha'\in I$. In this way, we say that $R$ is equipped with the \emph{standard fusion module structure}.
				
				A fusion $R$-module is said to be \emph{standard} if it is isomorphic to $R$ with its standard fusion module structure.
			\item Let $R:=(\mathbb{Z}_I, \oplus, \otimes, d)$ be a fusion ring. If $L\subset I$ is subset such that $(L, \mathbbold{1})$ is an involutive pointed set such that $N^{i}_{\alpha, i'}=0$, for all $\alpha, i'\in L$ and all $i\in I\backslash L$, then we obtain by restriction of $\otimes$ and $d$ a fusion ring $S:=(\mathbb{Z}_L, \oplus, \otimes_{|}, d_{|})$. It is called \emph{fusion subring of $R$} and we write $S\subset R$.
				
				For instance, given any basis element $\alpha\in I$ we can consider the \emph{fusion ring generated by $\alpha$}, which is the smallest fusion subring of $R$ containing $\alpha$.
			\item Let $R_1:=(\mathbb{Z}_{I_1}, \oplus, \otimes, d_1)$ and $R_2:=(\mathbb{Z}_{I_2}, \oplus, \otimes, d_2)$ be two fusion rings. We define the tensor product of $R_1$ and $R_2$, denoted by $R_1\otimes R_2$, as the free $\mathbb{Z}$-module $\mathbb{Z}_{I_1}\underset{\mathbb{Z}}{\odot} \mathbb{Z}_{I_2}$. It is a fusion ring with basis $I_{1}\underset{\mathbb{Z}}{\odot}I _{2}$, unit $\mathbbold{1}_1\odot \mathbbold{1}_2$, multiplication such that $(i_1\odot i_2)\otimes(i'_1\odot i'_2)=(i_1\otimes i'_1)\odot (i_2\otimes i'_2)$, for all $i_1, i'_1\in I_1$, $i_2, i'_2\in I_2$; involution $\overline{x\odot y}=\overline{x}\odot \overline{y}$, for all $x\in\mathbb{Z}_{I_1}$ and all $y\in\mathbb{Z}_{I_2}$; and dimension function $d(i_1\odot i_2)=d_1(i_1)d_2(i_2)$, for all $i_1\in I_1$ and all $i_2\in I_2$.
			\item Let $\widehat{\mathbb{G}}$ be a discrete quantum group. Define $(I,\mathbbold{1}):=(\text{Irr}(\mathbb{G}), \epsilon)$ as the pointed set with involution given by the adjoint representation. We define the \emph{fusion ring of $\widehat{\mathbb{G}}$} as the $\text{Irr}(\mathbb{G})$-based ring $\mathbb{Z}_{\text{Irr}(\mathbb{G})}$ with fusion rules and dimension function given by $N_{x, y}^{z}:=dim\Big(Mor(z, x\otimes y)\Big)\mbox{ and } d(x):=dim(H_x)\mbox{ ,}$ for all $x,y,z\in \text{Irr}(\mathbb{G})$. In other words, the ring structure is given simply by the tensor product of representations and so by the corresponding fusion rules. This ring is denoted by $\text{R}(\mathbb{G})$ and we refer to it as the \emph{representation ring of $\mathbb{G}$}. If $\widehat{\mathbb{H}}<\widehat{\mathbb{G}}$, then $\mathscr{Rep}(\mathbb{H})$ is a full subcategory of $\mathscr{Rep}(\mathbb{G})$, so that $\text{R}(\mathbb{H})$ is a fusion subring of $\text{R}(\mathbb{G})$.
			\item Let $R$ be a fusion ring and let $S\subset R$ be a fusion subring of $R$. By restriction, it is clear that $R$ can be viewed as a fusion $S$-module. If $N$ is a fusion $S$-module, the tensor product $M:=R\underset{S}{\odot} N$ is a $R$-module relevant for our purpose. We denote it by $\text{Ind}^{R}_{S}(N)$ and we call it \emph{induced $R$-module from $N$}. Notice that it is not clear a priori whether it is a \emph{based module} in the sense of the Definition \ref{defi.FusionModules}.
		\end{enumerate}
	\end{exs}
		
	\begin{defi}\label{defi.StrongTorsionFree}
		A fusion ring $R$ is called torsion-free if any non-trivial torsion $R$-module is standard. In particular, given a compact quantum group $\mathbb{G}$, we say that $\widehat{\mathbb{G}}$ is strong torsion-free if $\text{R}(\mathbb{G})$ is torsion-free.
	\end{defi}
	\begin{defi}
		Let $R$ be a fusion ring and let $S\subset R$ be a fusion subring. We say that $S$ is divisible in $R$ if $R\cong \underset{\Omega}{\bigoplus} S$ as (right) based $S$-modules. In particular, given two compact quantum groups $\mathbb{G}$ and $\mathbb{H}$ such that $\widehat{\mathbb{H}}<\widehat{\mathbb{G}}$, we say that $\widehat{\mathbb{H}}$ is divisible in $\widehat{\mathbb{G}}$ if $\text{R}(\mathbb{H})$ is divisible in $\text{R}(\mathbb{G})$.	
	\end{defi}
	\begin{rem}\label{rem.DivisibilityDecomp}
		Observe that given a basis element $i\in I$, it is decomposed as a sum of elements in $\underset{\Omega}{\bigoplus} L$ through the identification $R\cong \underset{\Omega}{\bigoplus} S$; but since the latter is an isomorphism of \emph{based} modules, the element $i\in I$ corresponds to a basis element in only one component of this direct sum, say $l_{i}\in L$ at position $t_i\in\Omega$. Let us denote by $\mathbbold{1}_t\in I\subset R$ the copy of $\mathbbold{1}_S$ at position $t\in\Omega$. Then one can write $i\cong \mathbbold{1}_{t_i}\otimes l_i$. In the decomposition $R\cong \underset{\Omega}{\bigoplus} S$ we choose an index $t_0\in \Omega$ such that the copy of $S$ at position $t_0$ is precisely the component $S\subset R$. In particular, $\mathbbold{1}=\mathbbold{1}_R=\mathbbold{1}_S=\mathbbold{1}_{t_0}$.
	\end{rem}
	\begin{rem}	
		In the particular case of compact quantum groups, we can give an alternative definition of divisibility. Namely, we define the following equivalence relation on $\text{Irr}(\mathbb{G})$: $x,y\in \text{Irr}(\mathbb{G})$, $x\sim y\Leftrightarrow$ there exists an irreducible representation $z\in \text{Irr}(\mathbb{H})$ such that $y\otimes \overline{x}\supset z$. We say that $\widehat{\mathbb{H}}$ is divisible in $\widehat{\mathbb{G}}$ if for each $\alpha\in \sim\backslash \text{Irr}(\mathbb{G})$ there exists a representation $l_{\alpha}\in \alpha$ such that $s\otimes l_{\alpha}$ is irreducible for all $s\in \text{Irr}(\mathbb{H})$ and $s\otimes l_\alpha\cong s'\otimes l_\alpha$ implies $s\cong s'$, for all $s,s'\in \text{Irr}(\mathbb{H})$. This is equivalent to say that for each $\alpha\in \text{Irr}(\mathbb{G})/\sim$ there exists a representation $r_{\alpha}\in \alpha$ such that $r_{\alpha}\otimes s$ is irreducible for all $s\in \text{Irr}(\mathbb{H})$ and $r_\alpha\otimes s\cong r_\alpha\otimes s'$ implies $s\cong s'$, for all $s,s'\in \text{Irr}(\mathbb{H})$. This is again equivalent to say that there exists a $\widehat{\mathbb{H}}$-equivariant $*$-isomorphism $c_0(\widehat{\mathbb{G}})\cong c_0(\widehat{\mathbb{H}})\otimes c_0(\widehat{\mathbb{H}}\backslash\widehat{\mathbb{G}})$ (see \cite{VoigtBaumConnesUnitaryFree} for a proof). 
	\end{rem}
		
	The approach of Y. Arano and K. De Commer meets the notion of torsion-freeness in the sense of R. Meyer and R. Nest when we work in the context of module C$^*$-categories. Let us recall the main definitions and results that make possible this connection. We refer to \cite{Longo}, \cite{YukiKenny} and \cite{SergeyMakoto} for more precisions and details. Let $\mathscr{C}$ be a rigid C$^*$-tensor category and $\mathscr{M}$ a $\mathscr{C}$-module C$^*$-category. We associate to $\mathscr{C}$ a fusion ring, denoted by $\text{Fus}(\mathscr{C})$, with basis given by irreducible objects of $\mathscr{C}$, fusion rules analogous to the fusion rules of a discrete quantum group and dimension function defined in \cite{Longo} (called \emph{intrinsic dimension}). We associate to $\mathscr{M}$ a based $\text{Fus}(\mathscr{C})$-module, denoted by $\text{Fus}(\mathscr{M})$, with basis given by irreducible objects in $\mathscr{M}$.
		\begin{defi}
			Let $\mathscr{C}$ be a rigid C$^*$-tensor category and $\mathscr{M}$ a $\mathscr{C}$-module C$^*$-category. We say that $\mathscr{M}$ is \emph{co-finite} (resp. connected, resp. torsion) if $\text{Fus}(\mathscr{M})$ is co-finite (resp. connected, resp. torsion) as $\text{Irr}(\mathscr{M})$-based $\text{Fus}(\mathscr{C})$-module.
		\end{defi}
	\begin{rem}
		More precisely, the previous definitions give the following. Recall that we always assume that our C$^*$-categories are semi-simple. $\mathscr{M}$ is co-finite if and only if for all non-zero objects $X, Y\in Obj(\mathscr{M})$, the set $\{U\in Obj(\mathscr{C})\ |\ \text{Hom}_{\mathscr{M}}(U\bullet Y, X)\neq 0\}$ is finite. $\mathscr{M}$ is connected if and only if for all non-zero objects $X, Y\in Obj(\mathscr{M})$, there exists an object $U\in Obj(\mathscr{C})$ such that $\text{Hom}_{\mathscr{M}}(U\bullet Y, X)\neq 0$.
	\end{rem}

	\begin{theodefi}[{\cite[Lemma 3.10 \& Lemma 3.11]{YukiKenny}}]\label{theo.TorsionFreeYukiKenny}
		Let $\mathscr{C}$ be a rigid C$^*$-tensor category. We say that $\mathscr{C}$ is torsion-free if one (hence all) of the following equivalent condition holds:
		\begin{enumerate}[i)]
			\item For every torsion $\mathscr{C}$-module C$^*$-category $\mathscr{M}$, $\text{Fus}(\mathscr{M})\cong \text{Fus}(\mathscr{C})$ as based modules.
			\item Every non-trivial torsion $\mathscr{C}$-module C$^*$-category is equivalent to $\mathscr{C}$ as $\mathscr{C}$-module C$^*$-categories.
		\end{enumerate}
		
		In particular, a discrete quantum group $\widehat{\mathbb{G}}$ is torsion-free if and only if for every torsion $\mathscr{Rep}(\mathbb{G})$-module C$^*$-category $\mathscr{M}$, $\text{Fus}(\mathscr{M})\cong \text{R}(\mathbb{G})$ as based modules.
	\end{theodefi}
	\begin{rem}
		Observe that the fusion ring associated to the C$^*$-tensor category $\mathscr{Rep}(\mathbb{G})$ is simply the representation ring of $\mathbb{G}$, $\text{R}(\mathbb{G})$. Hence, the above characterization of torsion-freeness of $\widehat{\mathbb{G}}$ is weaker than the strong torsion-freeness of $\widehat{\mathbb{G}}$ since here we work with a more restricted class of $\text{R}(\mathbb{G})$-modules. Namely, those arising from $\mathscr{Rep}(\mathbb{G})$-module C$^*$-categories.
	\end{rem}

\label{sec.DiscussionAboutKennyMakoto}

	Finally, K. De Commer and M. Yamashita \cite[Theorem 6.4]{KennyMakoto} have obtained a one-to-one correspondence between ergodic actions of $\mathbb{G}$ and connected $\mathscr{Rep}(\mathbb{G})$-module C$^*$-categories. Taking into account Theorem-Definition \ref{theo.TorsionFreeYukiKenny}, this allows in practice to replace torsion actions of a compact quantum group by torsion $\text{R}(\mathbb{G})$-modules. This correspondence is central for our purpose, so let us describe it more precisely. 

	If $(A,\delta)$ is a $\mathbb{G}$-C$^*$-algebra, then the category $\mathscr{M}_\delta$ of all $\mathbb{G}$-equivariant Hilbert $A$-modules is a left $\mathscr{Rep}(\mathbb{G})$-module C$^*$-category with left action of $\mathscr{Rep}(\mathbb{G})$ given by $u\bullet E:=H_u\otimes E,$ for all $u\in Obj(\mathscr{Rep}(\mathbb{G}))$ and all $(E,\delta_E)\in Obj(\mathscr{M}_\delta)$ where $\delta_{u\bullet E}: u\bullet E \longrightarrow u\bullet E\otimes C(\mathbb{G})$ is such that $\delta_{u\bullet E}(\xi\otimes\eta)=u_{13}(\xi\otimes \delta_E(\eta)),$ for all $\xi\in H_u$ and all $\eta\in E$. Moreover, $\mathscr{M}_\delta$ is semi-simple and connected whenever $\delta$ is ergodic. 
	\begin{rem}\label{rem.IrrCatHilbMod}
		In this case, it is also known (cf. \cite{KennyMakoto}) that every irreducible $\mathbb{G}$-equivariant Hilbert $A$-module arises as a $\mathbb{G}$-equivariant Hilbert submodule of $H_x\otimes A$ for some irreducible representation $x$ of $\mathbb{G}$.
	\end{rem}
	Conversely, if $\mathscr{M}$ is a connected $\mathscr{Rep}(\mathbb{G})$-module C$^*$-category, for every object $X\in Obj(\mathscr{M})$ the vector space \\$\mathcal{B}^X_X:=\underset{x\in \text{Irr}(\mathbb{G})}{\bigoplus} \text{Hom}_{\mathscr{M}}(u^x\bullet X, X)\otimes H_x$ is a unital $*$-algebra with an algebraic action of $\mathbb{G}$ denoted by $\delta_{X, \mathscr{M}}$ and given by $\delta_{X, \mathscr{M}}=\underset{x\in\text{Irr}(\mathbb{G})}{{\bigoplus}} (id\otimes \delta_{u^x})$ (cf. Example \ref{ex.ActHilbSpaceUnitRep} and see \cite[Section 5]{KennyMakoto} for more details). If $B^X_X$ denotes its C$^*$-completion, then $(B^X_X, \delta_{X, \mathscr{M}})$ is a $\mathbb{G}$-C$^*$-algebra (recall Theorem \ref{theo.AlgebraicActionsG}). Moreover, $\delta_{X, \mathscr{M}}$ is ergodic whenever $X$ is irreducible in $\mathscr{M}$. 
		
	These associations give rise to a one-to-one correspondence between ergodic actions of $\mathbb{G}$ (up to equivariant Morita equivalence) and connected $\mathscr{Rep}(\mathbb{G})$-module C$^*$-categories (up to equivalence of $\mathscr{Rep}(\mathbb{G})$-module C$^*$-categories). More precisely, if $(A, \delta)$ is an ergodic $\mathbb{G}$-C$^*$-algebra, then it can be viewed as an irreducible object in $\mathscr{M}_\delta$ and one has \cite[Proposition 6.2]{KennyMakoto} $(B^A_A, \delta_{A, \mathscr{M}_\delta})\underset{\mathbb{G}}{\sim} (A, \delta)$. Conversely, if $\mathscr{M}$ is a connected $\mathscr{Rep}(\mathbb{G})$-module C$^*$-category and $X\in \text{Irr}(\mathscr{\mathscr{M}})$, then one has \cite[Proposition 6.3]{KennyMakoto} $\mathscr{M}_{\delta_{X, \mathscr{M}}}\cong \mathscr{M}$. In particular, $\mathscr{M}_\delta$ is a torsion $\mathscr{Rep}(\mathbb{G})$-module C$^*$-category if and only if $\delta$ is a torsion action of $\mathbb{G}$. In this correspondence we have that $\mathscr{M}_\delta\cong \mathscr{Rep}(\mathbb{G})$ if and only if $(A,\delta)\underset{\mathbb{G}}{\sim}(\mathbb{C}, trv.)$.

\section{\textsc{Torsion-freeness for divisible discrete quantum subgroups}}\label{sec.StabilityTorsionFreeness}


In this section we are going to prove that torsion-freeness is preserved by \emph{divisible} discrete quantum subgroups. 

It is important to observe that torsion-freeness in the sense of Meyer-Nest is \emph{not} preserved, in general, by discrete quantum subgroups. For instance, consider $\widehat{SO_q(3)}<\widehat{SU_q(2)}$. While $\widehat{SU_q(2)}$ is torsion-free by \cite{VoigtBaumConnesFree}, $\widehat{SO_q(3)}$ is \emph{not} torsion-free by \cite{SoltanSO3}. In relation with the results obtained in \cite{RubenAmauryTorsion}, we can consider an other more elaborated example. Let $\mathbb{G}$ be a compact quantum group such that $\widehat{\mathbb{G}}$ is torsion-free. Then the dual of the free product $\mathbb{G}\ast SU_q(2)$ is torsion-free (because $\widehat{SU_q(2)}$ is torsion-free for all $q\in (-1,1)\backslash \{0\}$ as it is shown in \cite{VoigtBaumConnesFree} and torsion-freeness is preserved by free product as it is shown in \cite{YukiKenny}). Consider the Lemeux-Tarrago's discrete quantum subgroup $\widehat{\mathbb{H}}_q<\widehat{\mathbb{G}\ast SU_q(2)}$ which is such that $\mathbb{H}_q$ is monoidally equivalent to the free wreath product $\mathbb{G}\wr_*S^+_N$ (see \cite{TarragoWreath} for more details). It is explained in \cite{RubenAmauryTorsion} that the dual of $\mathbb{G}\wr_*S^+_N$ is \emph{never} torsion-free. Hence $\widehat{\mathbb{H}}_q$ neither (because torsion-freeness is preserved under monoidally equivalence as it is shown in \cite{VoigtBaumConnesFree} or \cite{RijdtMonoidalProbabilistic}). 
	
	Observe that strong torsion-freeness (cf. Definition \ref{defi.StrongTorsionFree}) is \emph{not} preserved, in general, by discrete quantum subgroups as it is pointed out in \cite{YukiKenny}. However, it is whenever the fusion subring $\text{R}(\mathbb{H})$ is divisible in the fusion ring $\text{R}(\mathbb{G})$ \cite{YukiKenny}. Thus, it is reasonable to expect that torsion-freeness (cf. Definition \ref{defi.TorsionFree}) is preserved under \emph{divisible} discrete quantum subgroups. Inspired by the study carried out in \cite{RubenAmauryTorsion}, our strategy consists in applying techniques from \cite{YukiKenny} for proving the following stability result: given a compact quantum group $\mathbb{G}$, $\widehat{\mathbb{G}}$ is torsion-free if and only if every divisible discrete quantum subgroup $\widehat{\mathbb{H}}<\widehat{\mathbb{G}}$ is torsion-free.
	
	\subsection{Preparatory observations}\label{sec.PreparatoryObs}
	\begin{lem}\label{lem.InducedTorsionModules}
		Let $R:=(\mathbb{Z}_I, \oplus, \otimes)$ be a $I$-based ring and $S:=(\mathbb{Z}_L, \oplus, \otimes)$ be a $L$-based ring. Let $N:=(\mathbb{Z}_J, \oplus, \otimes)$ be a based $S$-module with basis $J$. Assume that $S$ is a based subring of $R$. If $S$ is divisible in $R$, say $R\cong \underset{\Omega}{\bigoplus} S$ as (right) based $S$-modules, then the induced $R$-module $\text{Ind}^R_S(N)$ is a based $R$-module with basis: $$\widetilde{J}:=\{\mathbbold{1}_t\odot j\ |\ t\in\Omega, j\in J\},$$
		where $\mathbbold{1}_t\in R$ denotes the copy of $\mathbbold{1}_S$ at position $t\in \Omega$. Moreover, $\text{Ind}^R_S(N)$ is a torsion $R$-module whenever $N$ is a torsion $S$-module.
	\end{lem}
	\begin{proof}
	Since $S$ is a divisible based subring of $R$, we have $R\cong \underset{\Omega}{\bigoplus} S$ as (right) based $S$-modules. Through this isomorphism we write $\mathbbold{1}_t\in I\subset R$ for the copy of $\mathbbold{1}_S$ at position $t\in \Omega$ in this decomposition. For more clarity in the exposition, we put $t:=\mathbbm{1_t}$. The isomorphism $R\cong \underset{\Omega}{\bigoplus} S$ as (right) based $S$-modules yields straightforwardly that $\widetilde{J}:=\{t\odot j\ |\ t\in\Omega, j\in J\}$ is a $\mathbb{Z}$-basis for $\text{Ind}^R_S(N)=R\underset{S}{\odot}N$. 
	Moreover, $R\underset{S}{\odot} N$ is $\widetilde{J}$-based $R$-module because $R$ is a based $I$-module and $N$ is $J$-based module. To see this, we are going to use Remark \ref{rem.BilinearForm}. Given $\widetilde{j}, \widetilde{j}'\in \widetilde{J}$, say $\widetilde{j}:=t\odot j$ and $\widetilde{j}':=t'\odot j'$ with $t, t'\in\Omega$ and $j,j'\in J$, we define a $\mathbb{Z}$-bilinear form $\langle \cdot, \cdot\rangle: \text{Ind}^R_S(N)\times \text{Ind}^R_S(N)\rightarrow R$ such that $\langle \widetilde{j}, \widetilde{j}'\rangle:=t\otimes \langle j, j'\rangle \otimes \overline{t'}$.  Then we check the following for all $t, t'\in\Omega$, $j,j'\in J$ and $\alpha\in I$:
	\begin{enumerate}
		\item $\langle \alpha\otimes \widetilde{j}, \widetilde{j}'\rangle=\alpha\otimes\langle \widetilde{j}, \widetilde{j}'\rangle$. Indeed, by definition of the $R$-module structure of $\text{Ind}^R_S(N)=R\underset{S}{\odot}N$, the fusion decompositions and using the isomorphism $R\cong \underset{\Omega}{\bigoplus} S$ one has:
		\begin{equation*}
		\begin{split}
			\langle \alpha\otimes \widetilde{j}, \widetilde{j}'\rangle&=\langle \alpha\otimes t\odot j, t'\odot j'\rangle=\langle (\alpha\otimes t)\odot j, t'\odot j'\rangle=\underset{i\in I}{\sum} N^{i}_{\alpha, t}\cdot\langle i\odot j, t'\odot j' \rangle\\
			&\cong \underset{i\in I}{\sum} N^{i}_{\alpha, t}\cdot\langle t_i\odot (l_i\otimes j), t'\odot j' \rangle=\underset{i\in I, k\in J}{\sum} N^{i}_{\alpha, t} N^{k}_{l_i, j}\cdot\langle t_i\odot k, t'\odot j' \rangle\\
			&=\underset{i\in I, k\in J}{\sum} N^{i}_{\alpha, t} N^{k}_{l_i, j}\cdot t_i\otimes \langle k, j' \rangle\otimes \overline{t'}=\underset{i\in I}{\sum} N^{i}_{\alpha, t} \cdot t_i\otimes \langle l_i\otimes j, j' \rangle\otimes \overline{t'}\\
			&=\underset{i\in I}{\sum} N^{i}_{\alpha, t} \cdot t_i\otimes l_i\otimes \langle j, j' \rangle\otimes \overline{t'}\cong \underset{i\in I}{\sum} N^{i}_{\alpha, t} \cdot i\otimes \langle j, j' \rangle\otimes \overline{t'}\\
			&=\alpha\otimes t \otimes \langle j, j' \rangle\otimes \overline{t'}=\alpha\otimes\langle \widetilde{j}, \widetilde{j}'\rangle.
		\end{split}
		\end{equation*}
		
		\item $\langle \widetilde{j}, \widetilde{j}'\rangle=\overline{\langle \widetilde{j}', \widetilde{j}\rangle}$. This is clear.
		
		\item $\tau_R\big(\langle \widetilde{j}, \widetilde{j}'\rangle\big)=\delta_{\widetilde{j}, \widetilde{j}'}=\delta_{t, t'}\delta_{j, j'}$. By definition one has $\langle \widetilde{j}, \widetilde{j}'\rangle= t\otimes \langle j, j'\rangle \otimes \overline{t'}$, where $t, \overline{t'}\in R$ and $\langle j, j'\rangle \in S$. The fusion decomposition of $\langle j, j'\rangle$ as an element in $S$ is given by $\underset{l\in L}{\sum} N^{j'}_{\overline{l}, j}\cdot l$ (cf. Remark \ref{rem.BilinearForm}). Next, recall that $t$ (resp. $t'$) denotes the copy of $\mathbbold{1}_S$ at position $t\in\Omega$ (resp. $t'\in\Omega$) in the decomposition $R\cong \underset{\Omega}{\bigoplus} S$ as (right) based $S$-modules. In particular, one has $t\cong t\otimes \mathbbold{1}_{S}$ (resp. $t'\cong t'\otimes \mathbbold{1}_S$) as explained in Remark \ref{rem.DivisibilityDecomp}. Hence:
		$$ t\otimes \langle j, j'\rangle\cong t\otimes \mathbbold{1}_S\otimes \langle j, j'\rangle=\underset{l\in L}{\sum} N^{j'}_{\overline{l}, j}\cdot t\otimes l.$$
		
		At this point, notice that we are considering $S$ as a subring of $R$, hence $\mathbbold{1}_S=\mathbbold{1}_R$ and the element $t\otimes l$ with $l\in L$ is by assumption a basis element of $R$. Then one can write the following (cf. Remark \ref{rem.BilinearForm}):
		$$\tau_R\big(\langle \widetilde{j}, \widetilde{j}'\rangle\big)=\tau_R\Big(\underset{l\in L}{\sum} N^{j'}_{\overline{l}, j}\cdot (t\otimes l)\otimes (\overline{t'}\otimes \mathbbold{1}_S)\Big)=\underset{l\in L}{\sum} N^{j'}_{\overline{l}, j}\tau_R\big((t\otimes l)\otimes (\overline{t'\otimes \overline{\mathbbold{1}}_S})\big)=\underset{l\in L}{\sum} N^{j'}_{\overline{l}, j}\delta_{t\otimes l, t'\otimes \mathbbold{1}_S}=\delta_{t, t'}\delta_{j, j'},$$
		which yields the claim.
		
		\item $\tau\big(\alpha\otimes\langle \widetilde{j}, \widetilde{j}'\rangle\big)\in\mathbb{N}$. This is clear.
	\end{enumerate}
	
	According to Remark \ref{rem.BilinearForm} one has $N^{\widetilde{j}}_{\alpha, \widetilde{j'}}=\tau\big(\alpha\otimes\langle \widetilde{j}, \widetilde{j}'\rangle\big)$ and $\langle \widetilde{j}, \widetilde{j}'\rangle=\underset{i\in I}{\sum} N^{\widetilde{j'}}_{\overline{i}, \widetilde{j}}\cdot i$, for all $\widetilde{j}, \widetilde{j}'\in \widetilde{J}$ and all $\alpha\in I$. In particular, we deduce from the relation $\langle \widetilde{j}', \widetilde{j}\rangle=\overline{\langle \widetilde{j}, \widetilde{j}'\rangle}$ that, given  and $\alpha\in I$, $\widetilde{j}\subset \alpha\otimes \widetilde{j}'$ if and only if $\widetilde{j}'\subset \overline{\alpha}\otimes \widetilde{j}$.		
		
	Let us show that $R\underset{S}{\odot} N$ is torsion as soon as $N$ is torsion. We will follow a direct argment without using the bilinear form on $R\underset{S}{\odot} N$ above. 
	
	We start by showing that $R\underset{S}{\odot} N$ is co-finite. Given $\widetilde{j}, \widetilde{j}'\in \widetilde{J}$, we have to show that the set $A:=\{\alpha\in I\ |\ \widetilde{j}\subset \alpha\otimes \widetilde{j}'\}$ is finite. Put $\widetilde{j}:=t\odot j$ and $\widetilde{j}':=t'\odot j'$ with $t, t'\in\Omega$ and $j,j'\in J$. Take $\alpha\in A$, that is, $t\odot j\subset \alpha\otimes (t'\odot j')=(\alpha\otimes t')\odot j'$ (where the last equality follows from the $R$-module structure of $R\underset{S}{\odot} N$). Let us consider the fusion decomposition of $\alpha\otimes t'$, say $\alpha\otimes t'=\underset{i'\in I}{\sum} N^{i'}_{\alpha, t'}\cdot i'$. Given $i'\in I$, consider the corresponding basis element in $L$ through the identification $R\cong \underset{\Omega}{\bigoplus} S$, say $l_{i'}\in L$. Hence $(\alpha\otimes t')\odot j'=\underset{\underset{k'\in J}{i'\in I}}{\sum} N^{i'}_{\alpha, t'}N^{k'}_{l_{i'}, j'}\cdot t_{i'}\odot k'.$
	Saying that $\alpha$ is such that $t\odot j\subset (\alpha\otimes t')\odot j'$ means that there exists $i'\in I$ such that $t_{i'}=t$ (which allows to write $i'\cong t\otimes l_{i'}$, cf. Remark \ref{rem.DivisibilityDecomp}) and $N^{i'}_{\alpha, t'}N^{j}_{l_{t_{i'}}, j'}\neq 0$, i.e. $N^{i'}_{\alpha, t'}\neq 0$ and $N^{j}_{l_{i'}, j'}\neq 0$. Put differently in order to simplify notations, the fact that $\alpha\in A$ implies that there exists some $l\in L$ such that $j\subset l\otimes j'$ and $t\otimes l\subset \alpha\otimes t'$. Since $N$ is co-finite by assumption, there can exist only a finite set $L_0$ of elements $l$ with $j\subset l\otimes j'$. Next, co-finiteness of $R$ (as a left based module over itself) shows that there can exist only finitely many $\alpha$ with $t\otimes l\subset \alpha\otimes t'$, for some $l\in L_0$. In conclusion, $A$ is finite.
	
	Finally, we show that $R\underset{S}{\odot} N$ is connected. Given $\widetilde{j}, \widetilde{j}'\in \widetilde{J}$, we have to show that there exists $\alpha\in I$ such that $\widetilde{j}\subset \alpha\otimes \widetilde{j}'$. Put $\widetilde{j}:=t\odot j$ and $\widetilde{j}':=t'\odot j'$ with $t, t'\in\Omega$ and $j,j'\in J$. Since $N$ is connected, there exists $\gamma\in L$ such that $j\subset \gamma\otimes j'$. Then we also have that $t\odot j\subset t\odot (\gamma\otimes j')=(t\otimes \gamma)\odot j'$. Next, consider the fusion decomposition of $t\otimes \gamma$ in $R$, say $\underset{i\in I(t, \gamma)}{\sum} N^{i}_{t, \gamma}\cdot i$, where $I(t, \gamma)$ is the subset of $I$ formed by those basis elements $i$ such that $N^{i}_{t, \gamma}\neq 0$. Given $t'$, connectedness of $R$ yields that for each $i\in I(t, \gamma)$ appearing in this decomposition, there exists $\alpha_i\in I(t, \gamma)$ such that $i\subset \alpha_i\otimes t'$. Put $\alpha:=\underset{i\in I(t, \gamma)}{\oplus}\ \alpha_i$. By construction, one has that $t\otimes \gamma\subset \alpha\otimes t'$, hence $t\odot j\subset (\alpha\otimes t')\odot j'=\alpha\otimes (t'\odot j')$, which yields connectedness of $R\underset{S}{\odot} N$ as claimed.
	\end{proof}
	\begin{rem}
		Note that the previous proof shows that, under the divisibility assumption, $R\underset{S}{\odot} N$ is connected (resp. co-finite) as soon as $N$ is connected (resp. co-finite).
	\end{rem}
	
	\begin{lem}\label{lem.RestrictedTorsionModules}
		Let $R:=(\mathbb{Z}_I, \oplus, \otimes)$ be a $I$-based ring and $S:=(\mathbb{Z}_L, \oplus, \otimes)$ be a $L$-based ring. Let $M:=(\mathbb{Z}_J, \oplus, \otimes)$ be a based $R$-module with basis $J$. Assume that $S$ is a based subring of $R$. If $M$ is a $R$-module, then $\text{Res}^{R}_{S}(M)$ decomposes as a direct sum of connected $S$-modules. Moreover, if $M$ is co-finite, then $\text{Res}^{R}_{S}(M)$ decomposes as a direct sum of torsion $S$-modules.
	\end{lem}
	\begin{proof}
		Let $M$ be a $R$-module and denote by $N:=\text{Res}^{R}_{S}(M)$ the module $M$ equipped with the restriction action from $R$ to $S$, so that $N$ is a $S$-module. We are going to show that $N$ decomposes as a direct sum of connected $S$-modules. Let $J$ denote the basis of $M$ (which is also the basis for $N$). Next, take any $j_1\in J$ and put $N_1:=\langle \beta\otimes j_1\ |\ \beta\in L \rangle$ the based $S$-submodule generated by all $j\subset \beta \otimes j_1$ with $\beta\in L$. Then it is clear that $N_1$ is connected. If $N_1=N$, then we obtain already the decomposition claimed above. Otherwise, take $j_2\in J\backslash N_1$ and put $N_2:=\langle \beta\otimes j_2\ |\ \beta\in L \rangle$, which is connected again. If $N=N_1\oplus N_2$, then we obtain already the claim. Otherwise, we continue this process so that we obtain the decomposition $N=\bigoplus N_i$, where $\{N_i\}$ is a collection of connected $S$-modules and the claim is proven.
		
		If moreover $M$ is co-finite, then $N$ is obviously co-finite since $\{\beta\in L\ |\ j'\subset \beta\otimes j\}\subset \{\alpha\in I\ |\ j'\subset \alpha\otimes j\}$, for all $j,j'\in J$. In the same way, the $S$-submodules $N_i$ constructed before are still co-finite. Finally, it is straightforward to check that co-finiteness is preserved under direct summands. In other words, if $M$ is co-finite, then $N$ decomposes as a direct sum of connected and co-finite (i.e. torsion) $S$-modules.
	\end{proof}

	A direct consequence of Lemma \ref{lem.InducedTorsionModules} and Lemma \ref{lem.RestrictedTorsionModules} is the following:
	\begin{cor}\label{cor.IndResTorsion}
		Let $\mathbb{G}$ and $\mathbb{H}$ be two compact quantum groups such that $\widehat{\mathbb{H}}<\widehat{\mathbb{G}}$. If $N$ is a torsion $\text{R}(\mathbb{H})$-module, then the induced module, $\text{Ind}^{\text{R}(\mathbb{G})}_{\text{R}(\mathbb{H})}(N)$, is a torsion $\text{R}(\mathbb{G})$-module whenever $\widehat{\mathbb{H}}$ is divisible in $\widehat{\mathbb{G}}$. If $M$ is a torsion $\text{R}(\mathbb{G})$-module, then $\text{Res}^{\text{R}(\mathbb{G})}_{\text{R}(\mathbb{H})}(M)$ decomposes as a direct sum of torsion $\text{R}(\mathbb{H})$-modules.
	\end{cor}

	\begin{pro}\label{pro.CorrespondenceTorsionActions}
		Let $\mathbb{G}$ and $\mathbb{H}$ be two compact quantum groups such that $\widehat{\mathbb{H}}<\widehat{\mathbb{G}}$. If $(B,\beta)$ is a torsion action of $\mathbb{H}$, then the induced action, $(\text{Ind}^{\mathbb{G}}_{\mathbb{H}}(B), \widetilde{\beta})$, is a torsion action of $\mathbb{G}$. If $(A, \delta)$ is a torsion action of $\mathbb{G}$, then $(\text{Res}^\mathbb{G}_{\mathbb{H}}(A), \delta')$ decomposes as a direct sum of actions of $\mathbb{H}$ which are $\mathbb{H}$-Morita equivalent to torsion actions of $\mathbb{H}$.
	\end{pro}
	\begin{proof}
		The first part of the statement has been explained in Examples \ref{ex.TorsionActions}. Next, let $(A,\delta)$ be a torsion action of $\mathbb{G}$ and consider the corresponding $\mathscr{R}$ep$(\mathbb{G})$-module C$^*$-category $\mathscr{M}_{\delta}$ by \cite{KennyMakoto}. Denote by $\mathscr{N}_{\delta}$ the C$^*$-category $\mathscr{M}_{\delta}$ equipped with the restriction action from $\mathscr{R}$ep$(\mathbb{G})$ to $\mathscr{R}$ep$(\mathbb{H})$, so that $\mathscr{N}_{\delta}$ is a $\mathscr{R}$ep$(\mathbb{H})$-module C$^*$-category. 
		For any irreducible object $X\in Obj(\mathscr{N}_{\delta})$, denote by $\mathscr{N}^X_{\mathbb{H}}$ the module C$^*$-subcategory of $\mathscr{N}_{\delta}$ generated by $X$ and the action of $\mathscr{R}$ep$(\mathbb{H})$ (by restriction), which is a $\mathscr{R}$ep$(\mathbb{H})$-module C$^*$-category. By Lemma \ref{lem.RestrictedTorsionModules}, there exists a collection of irreducible objects $\{X_i\}\subset Obj(\mathscr{N}_{\delta})$ such that $\mathscr{N}_{\delta}=\bigoplus \mathscr{N}^{X_i}_{\mathbb{H}}$, where each $\mathscr{N}^{X_i}_{\mathbb{H}}$ is a torsion $\mathscr{R}$ep$(\mathbb{H})$-module C$^*$-category. In other words, we obtain a collection of torsion actions $\{(A_i,\delta_i)\}$ of $\mathbb{H}$ such that $(\text{Res}^\mathbb{G}_{\mathbb{H}}(A), \delta')\cong \bigoplus (A_i, \delta_i)$ as we wanted to show.
	\end{proof}
	
	\begin{lem}\label{lem.StabilisersDivisibleInd}
			Let $R:=(\mathbb{Z}_I, \oplus, \otimes)$ be a $I$-based ring and $S:=(\mathbb{Z}_L, \oplus, \otimes)$ be a $L$-based ring. Let $N:=(\mathbb{Z}_J, \oplus, \otimes)$ be a based $S$-module with basis $J$. Assume that $S$ is a based subring of $R$. If $S$ is divisible in $R$ and we consider $\text{Ind}^R_S(N)=R\underset{S}{\odot}N$ as a based $R$-module with basis $\widetilde{J}$ as in Lemma \ref{lem.InducedTorsionModules}, then $\text{Stab}(\mathbbold{1}_{t_0}\odot j)\subset L$, for all $j\in J$.
		\end{lem}
		\begin{proof}
			Recall that we write $\mathbbold{1}_t\in I\subset R$ for the copy of $\mathbbold{1}_S$ at position $t\in \Omega$ in the decomposition $R\cong \underset{\Omega}{\bigoplus} S$ as (right) based $S$-modules given by the divisibility assumption. For more clarity in the exposition, we put $t:=\mathbbm{1_t}$. Recall from Remark \ref{rem.DivisibilityDecomp} that we put $t_0=\mathbbold{1}=\mathbbold{1}_S=\mathbbold{1}_R$. Next, given $j\in J$ consider the basis element $t_0\odot j\in \widetilde{J}$. By definition we have:
			$$\text{Stab}(t_0\odot j)=\{\alpha\in I\ |\ t_0\odot j\subset \alpha \otimes (t_0\odot j)\}.$$
			
			Notice that given $\alpha\in I$, the $R$-module structure of $R\underset{S}{\odot}N$ allows to write $\alpha \otimes (t_0\odot j)=(\alpha\otimes t_0)\odot j=\alpha\odot j$. Now, we express $\alpha\odot j\in R\underset{S}{\odot}N$ in terms of the basis $\widetilde{J}$. Namely, using the divisibility assumption, one can write $\alpha\cong t_\alpha\otimes l_\alpha$, for some $t_\alpha\in I$ and $l_\alpha\in L$ (cf. Remark \ref{rem.DivisibilityDecomp}). Hence:
			\begin{equation*}
			\begin{split}
				\alpha\odot j&\cong t_\alpha\odot (l_\alpha\otimes j)=\underset{k\in J}{\sum} N_{l_\alpha, j}^k\cdot t_\alpha\odot k.
			\end{split}
			\end{equation*}
			
			Accordingly, $\alpha\in \text{Stab}(t_0\odot j)$ if and only if $t_0\odot j\subset \underset{k\in J}{\sum} N_{l_\alpha, j}^k\cdot t_\alpha\odot k$, which is equivalent to say that $N_{l_\alpha, j}^j\neq 0$ and $t_\alpha=t_0$. The latter implies that $\alpha\cong t_0\otimes l_\alpha=l_\alpha\in L$, which finishes the proof.
		\end{proof}
		\begin{rem}\label{rem.StabilisersDivisibleInd}
			We keep the previous notations. Notice that from the proof of the previous lemma we can conclude that $\text{Stab}(t_0\odot j)=\text{Stab}(j)$, for all $j\in J$.
		\end{rem}
	
	\subsection{Induction for module C$^*$-categories}\label{sec.AlternativeApproachTorsionFree}
		First of all, in the purely algebraic setting, the notion of relative tensor product of module categories over tensor categories goes back to constructions by P. Deligne \cite{Deligne} and D. Tambara \cite{Tambara} (see as well \cite{MugerSubfactorsCat}, \cite{EtingofFusCatHomTh}, \cite{IgnacioTensorProdCat}, \cite{DouglasSchommerPriesSnyder} for further related developments). The corresponding construction for module C$^*$-categories is certainly known to experts or, at least, foreseen by experts (see for instance \cite[Section 4.1]{SergeyMakoto}, \cite{AmbrogioLimitsCCat} or \cite{MeyerColimitsCCorresp}), but a formal and general definition seemed to be elusive in the literature. The construction of balanced tensor products of module categories over C$^*$-categories (without any semisimplicity or rigidity assumption) has been established in a recent work of J. Antoun and C. Voigt \cite{VoigtBalancedTensorProductCategories} as a byproduct of a more general category related considerations.
	
	First, let us introduce some terminology.
	\begin{defi}
	Let $\mathscr{C}$ be a C$^*$-tensor category. Let $(\mathscr{M}, \bullet, \mu, e)$ be a right $\mathscr{C}$-module C$^*$-category, $(\mathscr{N}, \bullet, \mu', e')$ a left $\mathscr{C}$-module C$^*$-category and $\mathscr{A}$ a C$^*$-category (resp. $\mathbb{C}$-linear category). A C$^*$-bifunctor (resp. $\mathbb{C}$-linear bifunctor) $F: \mathscr{M}\times\mathscr{N}\longrightarrow \mathscr{A}$ is called $\mathscr{C}$-balanced if it is equipped with a natural equivalence $b: F\circ (\bullet\times id_{\mathscr{N}})\longrightarrow F\circ (id_{\mathscr{M}}\times \bullet)$ such that
	\begin{enumerate}[i)]
		\item the diagram:
			$$
 				\xymatrix@!C=13mm@R=12mm{
 					& &\mbox{$F(X\bullet (U\otimes V), Y)$}\ar[dll]_{\mbox{$b_{X, U\otimes V, Y}$}}\ar[drr]^{\mbox{$F\circ (\mu_{X, U, V}\times id_{Y})$}} & &\\
 					\mbox{$F(X, (U\otimes V)\bullet Y)$}\ar[dr]_{\mbox{$F\circ (id_{X}\times \mu'_{U, V, Y})$}}& & & &\mbox{$F((X\bullet U)\bullet V, Y)$}\ar[dl]^{\mbox{$b_{X\bullet U, V, Y}$}}\\
 					&\mbox{$F(X, U\bullet (V\bullet Y))$}& &\mbox{$F(X\bullet U,V\bullet Y)$}\ar[ll]^{\mbox{$b_{X, U, V\bullet Y}$}}&}
 			$$
			is commutative for all objects $U, V\in Obj(\mathscr{C})$, $X\in Obj(\mathscr{M})$ and $Y\in Obj(\mathscr{N})$.
		\item the diagram:
			$$
				\xymatrix{
					\mbox{$F(X\bullet \mathbbold{1},  Y)$}\ar[dr]_{\mbox{$F\circ (e_X\times id_Y)$}}\ar[rr]^{\mbox{$b_{X,\mathbbold{1},Y}$}}&&\mbox{$F(X, \mathbbold{1}\bullet Y)$}\ar[dl]^{\mbox{$F\circ (id_X\times e'_Y)$}}\\
					&\mbox{$F(X, Y)$}&}
			$$
			is commutative for all objects $X\in Obj(\mathscr{M})$ and $Y\in Obj(\mathscr{N})$.
	\end{enumerate}
	
	The collection of all $\mathscr{C}$-balanced functors from $\mathscr{M}\times\mathscr{N}$ to $\mathscr{A}$ together with natural transformations between them forms a C$^*$-category (resp. involutive category) denoted by $\mathscr{B}il_{\mathscr{C}}(\mathscr{M}, \mathscr{N}; \mathscr{A})$.
\end{defi}
\begin{rem}\label{rem.MultiplierCategory}
	On the one hand, for the sake of the presentation we omit the definition of natural transformation between $\mathscr{C}$-balanced functors. On the other hand, it is important to notice that the notion of $\mathscr{C}$-balanced functor appearing in the algebraic context \cite{Tambara} must be adapted to the C$^*$-level. Namely, if the C$^*$-categories involved are countably additive, we will have to deal with \emph{semi-categories}, that is, categories who do not necessarily contain identity homomorphisms. This entails the definition of a \emph{multiplier C$^*$-category} together with a notion of \emph{non-degenerate C$^*$-functor} by imitating the standard C$^*$-algebraic case. This has been done in \cite{VoigtBalancedTensorProductCategories}. For our purpose we do not need such general considerations in view of assumptions of Note \ref{note.AssumptionsCtensorcat} and Remark \ref{rem.AssumptionsSemiSimple}, so that the previous definition is useful enough (compare with \cite[Definition 5.8]{VoigtBalancedTensorProductCategories}).
\end{rem}
	
	Let $\mathscr{C}$ be a countably additive C$^*$-tensor category, $\mathscr{M}$ a right $\mathscr{C}$-module C$^*$-category and $\mathscr{N}$ a left $\mathscr{C}$-module C$^*$-category. Then there exists a countably additive C$^*$-category $\mathscr{M}\underset{\mathscr{C}}{\boxtimes} \mathscr{N}$ together with a $\mathscr{C}$-balanced functor $Q_{\mathscr{C}}:\mathscr{M}\times\mathscr{N}\longrightarrow \mathscr{M}\underset{\mathscr{C}}{\boxtimes} \mathscr{N}$ such that composition with $Q_{\mathscr{C}}$ induces an equivalence of C$^*$-categories $\mathscr{F}un(\mathscr{M}\underset{\mathscr{C}}{\boxtimes} \mathscr{N}, \mathscr{A})\cong\mathscr{B}il_{\mathscr{C}}(\mathscr{M}, \mathscr{N}; \mathscr{A})$, for all C$^*$-category $\mathscr{A}$ (see \cite[Theorem 5.10]{VoigtBalancedTensorProductCategories} for a proof). The category $\mathscr{M}\underset{\mathscr{C}}{\boxtimes} \mathscr{N}$ is called \emph{relative tensor product of $\mathscr{M}$ and $\mathscr{N}$ with respect to $\mathscr{C}$}. In this way, if $\mathscr{D}$ is another countable additive C$^*$-tensor category such that there exists a C$^*$-tensor functor $J: \mathscr{D}\longrightarrow \mathscr{C}$ (so that $\mathscr{C}$ can be considered as a right $\mathscr{D}$-module C$^*$-category, recall Examples \ref{exs.ModuleCategories}) and $\mathscr{N}$ is a left $\mathscr{D}$-module C$^*$-category, we define the \emph{induced $\mathscr{C}$-module C$^*$-category of $\mathscr{N}$} to be the following $\mathscr{C}$-bimodule C$^*$-category $\text{Ind}^{\mathscr{C}}_{\mathscr{D}}(\mathscr{N}):=\mathscr{C}\underset{\mathscr{D}}{\boxtimes} \mathscr{N}$. On the one hand, the $\mathscr{C}$-bimodule structure on $\text{Ind}^{\mathscr{C}}_{\mathscr{D}}(\mathscr{N})$ is simply given by the one of $\mathscr{C}$ itself. On the other hand, observe that thanks to \cite[Lemma 5.11]{VoigtBalancedTensorProductCategories}, we have $\text{Ind}^{\mathscr{D}}_{\mathscr{D}}(\mathscr{N})\cong\mathscr{N}$ as left $\mathscr{D}$-module C$^*$-categories.
	
	In particular, let $\mathbb{G}$ and $\mathbb{H}$ be compact quantum groups such that $\widehat{\mathbb{H}}<\widehat{\mathbb{G}}$. As already pointed out in Examples \ref{exs.ModuleCategories}(\ref{ex.ModCatQG}), we have a fully faithful functor between C$^*$-tensor categories, $J:\mathscr{Rep}(\mathbb{H})\longrightarrow\mathscr{Rep}(\mathbb{G})$, given by the natural inclusion of $\mathscr{Rep}(\mathbb{H})$ inside $\mathscr{Rep}(\mathbb{G})$ as a full subcategory. In this way, $\mathscr{Rep}(\mathbb{G})$ can be viewed as a right $\mathscr{Rep}(\mathbb{H})$-module C$^*$-category. So, given a left $\mathscr{Rep}(\mathbb{H})$-module C$^*$-category $\mathscr{N}$ we put:
	$$\mbox{\text{Ind}}^{\mathscr{Rep}(\mathbb{G})}_{\mathscr{Rep}(\mathbb{H})}(\mathscr{N}):=\mathscr{Rep}(\mathbb{G})\underset{\mathscr{Rep}(\mathbb{H})}{\boxtimes} \mathscr{N}.$$

	\begin{rem}
			Let $\mathscr{C}$ and $\mathscr{D}$ be two countable additive C$^*$-categories. One defines their \emph{minimal} and \emph{maximal} tensor product, denoted by $\mathscr{C}\underset{min}{\boxtimes} \mathscr{D}$ and $\mathscr{C}\underset{max}{\boxtimes} \mathscr{D}$, respectively as the subobject completion of the corresponding algebraic minimal and maximal tensor products, respectively. It is shown in \cite[Proposition 3.4]{VoigtBalancedTensorProductCategories} that $\mathscr{C}\underset{min}{\boxtimes} \mathscr{D}$ and $\mathscr{C}\underset{max}{\boxtimes} \mathscr{D}$ are again countably additive. These definitions make use of the notion of multiplier C$^*$-category. But, as we have already mentioned in Remark \ref{rem.MultiplierCategory}, assumptions of Note \ref{note.AssumptionsCtensorcat} and Remark \ref{rem.AssumptionsSemiSimple} allow us to disregard such general considerations. More precisely, if $\mathscr{C}$ and $\mathscr{D}$ are semi-simple C$^*$-categories, then there is no completions involved in the definition of their minimal or maximal tensor products, so that $\mathscr{C}\underset{min}{\boxtimes} \mathscr{D}\cong \mathscr{C}\underset{max}{\boxtimes} \mathscr{D}$. Moreover, one shows  \cite[Proposition 3.10]{VoigtBalancedTensorProductCategories} that this tensor product coincides with the Deligne tensor product, which we denote simply by $\boxtimes$. The proof of \cite[Proposition 3.10]{VoigtBalancedTensorProductCategories} shows that $\mathscr{C}\boxtimes \mathscr{D}$ is again semi-simple with a complete set of irreducible objects given by $\{X_i\boxtimes Y_j\}_{i\in I, j\in J}$, where $\{X_i\}_{i\in I}$ and $\{Y_j\}_{j\in J}$ are complete sets of irreducible objects in $\mathscr{C}$ and $\mathscr{D}$, respectively.
		\end{rem}

	Next, given a torsion action $(B, \beta)$ of $\mathbb{H}$, consider  the corresponding torsion $\mathscr{R}$ep$(\mathbb{H})$-module C$^*$-category $\mathscr{N}_{\beta}$ as explained at the end of Section \ref{sec.DiscussionAboutKennyMakoto}. We want to show that the induced category $\mbox{\text{Ind}}^{\mathscr{Rep}(\mathbb{G})}_{\mathscr{Rep}(\mathbb{H})}(\mathscr{N}_\beta)$ is torsion as soon as $\widehat{\mathbb{H}}$ is divisible in $\widehat{\mathbb{G}}$. It follows from the following proposition, which shows that the divisibility property extends to the level of the representation categories.
	
	\begin{pro}\label{pro.DivisibilityRepCat}
		Let $\mathbb{G}$ and $\mathbb{H}$ be two compact quantum groups such that $\widehat{\mathbb{H}}<\widehat{\mathbb{G}}$. If $\widehat{\mathbb{H}}$ is divisible in $\widehat{\mathbb{G}}$, then $\mathscr{Rep}(\mathbb{G})\cong \underset{\Omega}{\bigoplus}\mathscr{Rep}(\mathbb{H})$ as $\mathscr{Rep}(\mathbb{H})$-module C$^*$-categories.
	\end{pro}
	\begin{proof}
		Since $\widehat{\mathbb{H}}$ is divisible in $\widehat{\mathbb{G}}$ by assumption, then we have $\text{R}(\mathbb{G})\cong\underset{\Omega}{\bigoplus} \text{R}(\mathbb{H})$ as based $\text{R}(\mathbb{H})$-modules. For each $t\in\Omega$, we denote by $\mathbbold{1}_t\in \text{R}(\mathbb{G})$ the copy of $\mathbbold{1}_{\text{R}(\mathbb{H})}=\epsilon$ at position $t\in \Omega$ in the decomposition $\text{R}(\mathbb{G})\cong \underset{\Omega}{\bigoplus} \text{R}(\mathbb{H})$. Next, given $t\in\Omega$ denote by $\mathscr{Rep}(\mathbb{G})_t$ the $\mathscr{Rep}(\mathbb{H})$-module subcategory of $\mathscr{Rep}(\mathbb{G})$ generated by $\mathbbold{1}_t$. It is clear that $\mbox{\text{Fus}}(\mathscr{Rep}(\mathbb{G})_t)=\text{R}(\mathbb{H})$ as based $\text{R}(\mathbb{H})$-modules, which implies that $\mathscr{Rep}(\mathbb{G})_t\cong \mathscr{Rep}(\mathbb{H})$ as $\mathscr{Rep}(\mathbb{H})$-module C$^*$-categories by virtue of \cite[Lemma 3.10]{YukiKenny}. Since this is true for all $t\in\Omega$, we deduce that $\mathscr{Rep}(\mathbb{G})$ can be viewed as a direct sum of $\mathscr{Rep}(\mathbb{H})$-module C$^*$-categories as in the statement.
	\end{proof}
	
	\begin{cor}\label{cor.TorsionInducedCateg}
		Let $\mathbb{G}$ and $\mathbb{H}$ be two compact quantum groups such that $\widehat{\mathbb{H}}<\widehat{\mathbb{G}}$. If $\widehat{\mathbb{H}}$ is divisible in $\widehat{\mathbb{G}}$, then we have:
		$$\text{Ind}^{\text{R}(\mathbb{G})}_{\text{R}(\mathbb{H})}(\mbox{\text{Fus}}(\mathscr{N}_\beta))\cong \mbox{\text{Fus}}\Big(\text{Ind}^{\mathscr{Rep}(\mathbb{G})}_{\mathscr{Rep}(\mathbb{H})}(\mathscr{N}_\beta)\Big),$$
		as based $\text{R}(\mathbb{G})$-modules, for all torsion action $(B,\beta)$ of $\mathbb{H}$. Therefore, $\text{Ind}^{\mathscr{Rep}(\mathbb{G})}_{\mathscr{Rep}(\mathbb{H})}(\mathscr{N}_\beta)$ is a torsion $\mathscr{Rep}(\mathbb{G})$-module C$^*$-category (hence semi-simple), for all torsion action $(B,\beta)$ of $\mathbb{H}$.
	\end{cor}
	\begin{proof}
		On the one hand, by virtue of Lemma \ref{lem.InducedTorsionModules}, the induced $\text{R}(\mathbb{G})$-module $\text{Ind}^{\text{R}(\mathbb{G})}_{\text{R}(\mathbb{H})}(\mbox{\text{Fus}}(\mathscr{N}_\beta))=\text{R}(\mathbb{G})\underset{\text{R}(\mathbb{H})}{\odot} \text{Fus}(\mathscr{N}_\beta)$ is a (left) based $\text{R}(\mathbb{G})$-module with basis $\widetilde{J}:=\{\mathbbold{1}_t\odot j\ |\ t\in\Omega, (Y, \delta_Y)\in J\}$, where $\mathbbold{1}_t\in \text{Irr}(\mathbb{G})$ denotes the copy of $\mathbbold{1}_{\text{R}(\mathbb{H})}$ at position $t\in \Omega$ and $J:=\text{Irr}(\mathscr{N}_\beta)$. Moreover, $\text{Ind}^{\text{R}(\mathbb{G})}_{\text{R}(\mathbb{H})}(\mbox{\text{Fus}}(\mathscr{N}_\beta))$ is a torsion $\text{R}(\mathbb{G})$-module because $\text{Fus}(\mathscr{N}_\beta)$ is a torsion $\text{R}(\mathbb{H})$-module. On the other hand, by virtue of the isomorphism $\mathscr{Rep}(\mathbb{G})\cong \underset{\Omega}{\bigoplus}\mathscr{Rep}(\mathbb{H})$ as $\mathscr{Rep}(\mathbb{H})$-module C$^*$-categories from Proposition \ref{pro.DivisibilityRepCat}, the set $\widetilde{J}$ provides a complete set of irreducible objects in $\text{Ind}^{\mathscr{Rep}(\mathbb{G})}_{\mathscr{Rep}(\mathbb{H})}(\mathscr{N}_\beta)=\mathscr{Rep}(\mathbb{G})\underset{\mathscr{Rep}(\mathbb{H})}{\boxtimes} \mathscr{N}_{\beta}$. It follows that $\text{Ind}^{\text{R}(\mathbb{G})}_{\text{R}(\mathbb{H})}(\mbox{\text{Fus}}(\mathscr{N}_\beta))\cong \mbox{\text{Fus}}\Big(\text{Ind}^{\mathscr{Rep}(\mathbb{G})}_{\mathscr{Rep}(\mathbb{H})}(\mathscr{N}_\beta)\Big)$ as based $\text{R}(\mathbb{G})$-modules, for all torsion action $(B,\beta)$ of $\mathbb{H}$
	\end{proof}
	\begin{rem}
		Note that the previous corollary remains true for ergodic actions of $\mathbb{H}$ (not necessarily finite dimensional).
	\end{rem}
	
	\begin{theo}\label{theo.TorsionFreenessDivisible}
		Let $\mathbb{G}$ and $\mathbb{H}$ be two compact quantum groups such that $\widehat{\mathbb{H}}<\widehat{\mathbb{G}}$ is divisible. If $\widehat{\mathbb{G}}$ is torsion-free, then $\widehat{\mathbb{H}}$ is torsion-free.
	\end{theo}
	\begin{proof}
		Let $(B, \beta)$ be a torsion action of $\mathbb{H}$ and consider the corresponding torsion $\mathscr{R}$ep$(\mathbb{H})$-module C$^*$-category $\mathscr{N}_{\beta}$. Put $\mathscr{M}:=\text{Ind}^{\mathscr{Rep}(\mathbb{G})}_{\mathscr{Rep}(\mathbb{H})}(\mathscr{N}_\beta)$ for the induced C$^*$-module category. By Corollary \ref{cor.TorsionInducedCateg}, $\mathscr{M}$ is a torsion $\mathscr{Rep}(\mathbb{G})$-module C$^*$-category. Since $\widehat{\mathbb{G}}$ is torsion-free by hypothesis, we have that $\mathscr{M}\cong \mathscr{Rep}(\mathbb{G})$ as $\mathscr{Rep}(\mathbb{G})$-module C$^*$-categories, which is equivalent to say that $\mbox{\text{Fus}}(\mathscr{M})\cong \text{R}(\mathbb{G})$ as based $\text{R}(\mathbb{G})$-modules (recall Theorem \ref{theo.TorsionFreeYukiKenny}). 
		
		Corollary \ref{cor.TorsionInducedCateg} gives as well that $\text{Fus}(\mathscr{M})\cong \text{Ind}^{\text{R}(\mathbb{G})}_{\text{R}(\mathbb{H})}(\text{Fus}(\mathscr{N}_\beta))=\text{R}(\mathbb{G})\underset{\text{R}(\mathbb{H})}{\odot} \text{Fus}(\mathscr{N}_\beta)$ as based $\text{R}(\mathbb{G})$-modules. Next, since $\widehat{\mathbb{H}}$ is divisible in $\widehat{\mathbb{G}}$ by assumption, then we have $\text{R}(\mathbb{G})\cong\underset{\Omega}{\bigoplus} \text{R}(\mathbb{H})$ as based $\text{R}(\mathbb{H})$-modules. In this situation, we can apply verbatim the proof of \cite[Proposition 1.28]{YukiKenny} to show that $\text{Fus}(\mathscr{N}_{\beta})\cong \text{R}(\mathbb{H})$ as $\text{R}(\mathbb{H})$-modules, which yields $\mathscr{N}_{\beta}\cong \mathscr{Rep}(\mathbb{H})$ thanks to Theorem \ref{theo.TorsionFreeYukiKenny} again. Since this is true for all torsion action $\beta$ of $\mathbb{H}$, this will imply that $\widehat{\mathbb{H}}$ is torsion-free. 
		
		Namely, if $J$ denotes the basis for the torsion $\text{R}(\mathbb{H})$-module $\text{Fus}(\mathscr{N}_{\beta})$, then the basis for $\text{Ind}^{\text{R}(\mathbb{G})}_{\text{R}(\mathbb{H})}\Big(\text{Fus}(\mathscr{N}_\beta)\Big)$ taking into account the divisibility condition is given by $\widetilde{J}:=\{\alpha\otimes j\ |\ \alpha\in I_0\mbox{ and } j\in J\}\mbox{,}$ where $I_0$ denotes the set of elements in $\text{Irr}(\mathbb{G})$ corresponding to the units of the components $\text{R}(\mathbb{H})$ in the decomposition $\text{R}(\mathbb{G})=\oplus \text{R}(\mathbb{H})$ (cf. Lemma \ref{lem.InducedTorsionModules}). By the above discussion we have $\text{Ind}^{\text{R}(\mathbb{G})}_{\text{R}(\mathbb{H})}\Big(\text{Fus}(\mathscr{N}_\beta)\Big)\cong \text{R}(\mathbb{G})$ as based $\text{R}(\mathbb{G})$-modules. Hence, take $\alpha_0\otimes j_0\in \widetilde{J}$ corresponding to $\epsilon\in \text{Irr}(\mathbb{G})$. Then $\alpha\otimes (\alpha_0\otimes j_0)\in \widetilde{J}$, for each $\alpha\in \text{Irr}(\mathbb{G})$ because the isomorphism $\text{Ind}^{\text{R}(\mathbb{G})}_{\text{R}(\mathbb{H})}\Big(\text{Fus}(\mathscr{N}_\beta)\Big)\cong \text{R}(\mathbb{G})$ sends basis to basis. In particular, $\alpha\otimes \alpha_0\in \text{Irr}(\mathbb{G})$, for each $\alpha\in \text{Irr}(\mathbb{G})$ and so $\overline{\alpha}_0\otimes \alpha_0= \epsilon$. We may thus assume $\alpha_0=\epsilon$ and we find that $\text{Fus}(\mathscr{N}_{\beta})\cong \text{R}(\mathbb{H})\underset{\text{R}(\mathbb{H})}{\odot} \text{Fus}(\mathscr{N}_{\beta})\cong \text{R}(\mathbb{H})$ as based $\text{R}(\mathbb{H})$-modules as we wanted to show.		
	\end{proof}

	\subsection{Relation between induced torsion actions and induced torsion module C$^*$-categories}
		In this section we show that the induction at the level of torsion actions corresponds precisely to induction at the level of torsion modules C$^*$-categories under the divisibility assumption.
		
		Let $\mathbb{G}$ and $\mathbb{H}$ be compact quantum groups such that $\widehat{\mathbb{H}}<\widehat{\mathbb{G}}$. Given a torsion action $(B, \beta)$ of $\mathbb{H}$, consider  the induced torsion action of $\mathbb{G}$, $(\text{Ind}^{\mathbb{G}}_{\mathbb{H}}(B), \widetilde{\beta})$, where $\text{Ind}^{\mathbb{G}}_{\mathbb{H}}(B)$ is the same C$^*$-algebra $B$ equipped with the composition $\widetilde{\beta}:=(id_B\otimes \iota)\circ\beta$ as action of $\mathbb{G}$, where $C(\mathbb{H})\overset{\iota}{\subset} C(\mathbb{G})$ is the canonical embedding defining $\widehat{\mathbb{H}}<\widehat{\mathbb{G}}$. 
		
		On the one hand, we can consider the torsion $\mathscr{R}$ep$(\mathbb{H})$-module C$^*$-category $\mathscr{N}_{\beta}$ associated to $\beta$ and the corresponding torsion $\text{R}(\mathbb{H})$-module, $\text{Fus}(\mathscr{N}_{\beta})$. On the other hand, we can consider the torsion $\mathscr{R}$ep$(\mathbb{G})$-module C$^*$-category $\mathscr{M}_{\widetilde{\beta}}$ associated to $\widetilde{\beta}$ and the corresponding torsion $\text{R}(\mathbb{G})$-module, $\text{Fus}(\mathscr{M}_{\widetilde{\beta}})$. We want to show that $\text{Ind}^{\mathscr{Rep}(\mathbb{G})}_{\mathscr{Rep}(\mathbb{H})}(\mathscr{N}_\beta)\cong \mathscr{M}_{\widetilde{\beta}}$ as soon as $\widehat{\mathbb{H}}$ is divisible in $\widehat{\mathbb{G}}$. Note that if $(B, \beta)$ is the trivial action of $\mathbb{H}$, then $\mathscr{N}_\beta\cong \mathscr{Rep}(\mathbb{H})$, $\mathscr{M}_{\widetilde{\beta}}\cong \mathscr{Rep}(\mathbb{G})$ and the property is obviously satisfied. We assume then that $(B, \beta)$ is a non-trivial torsion action of $\mathbb{H}$.
		
		In order to fix our notations, recall that $\text{Fus}(\mathscr{N}_\beta)$ is by definitoin the free $\mathbb{Z}$-module with basis $J:=\text{Irr}(\mathscr{N}_{\beta})$, where $\text{Irr}(\mathscr{N}_{\beta})$ is formed by $\mathbb{H}$-equivariant Hilbert $B$-submodules of $H_y\otimes B$ with $y\in \text{Irr}(\mathbb{H})$ (cf. Remark \ref{rem.IrrCatHilbMod}). Similarly, $\text{Fus}(\mathscr{M}_{\widetilde{\beta}})$ is by definition the free $\mathbb{Z}$-module with basis $J':=\mbox{\text{Irr}}(\mathscr{M}_{\widetilde{\beta}})$, where $\text{Irr}(\mathscr{M}_{\widetilde{\beta}})$ is formed by $\mathbb{G}$-equivariant Hilbert $\text{Ind}^{\mathbb{G}}_{\mathbb{H}}(B)$-submodules of $H_x\otimes \text{Ind}^{\mathbb{G}}_{\mathbb{H}}(B)$ with $x\in \mbox{\text{Irr}}(\mathbb{G})$.
		
		\begin{lem}\label{lem.NBetaFullSubMBetaTilde}
			We keep the previous notations. $\mathscr{N}_{\beta}$ is a full $\mathscr{Rep}(\mathbb{H})$-module C$^*$-subcategory of $\mathscr{M}_{\widetilde{\beta}}$.
		\end{lem}
		\begin{proof}
		
		By definition, $\mathscr{M}_{\widetilde{\beta}}$ is the $\mathscr{R}$ep$(\mathbb{G})$-module C$^*$-category whose objects are the $\mathbb{G}$-equivariant Hilbert $\text{Ind}^{\mathbb{G}}_{\mathbb{H}}(B)$-modules. 
		
		On the one hand, note that given a $\mathbb{H}$-equivariant Hilbert $B$-module $(E,\delta_E)$, the modules of the form $(\text{Ind}^{\mathbb{G}}_{\mathbb{H}}(E), \widetilde{\delta}_E)$ where $\text{Ind}^{\mathbb{G}}_{\mathbb{H}}(E):=E$ equipped with the composition $\widetilde{\delta}_E:=(id_E\otimes \iota)\circ\delta_E$ as action of $\mathbb{G}$ are in $\mathscr{M}_{\widetilde{\beta}}$. Indeed, given an $\mathbb{H}$-equivariant Hilbert $B$-module $(E,\delta_E)$, let $V_E\in \mathcal{L}_{B\otimes C(\mathbb{H})}\big(E\underset{\beta}{\otimes} (B\otimes C(\mathbb{H})), B\otimes C(\mathbb{H})\big)$ be the admissible operator associated to the action $\delta_E$ of $\mathbb{H}$ on $E$ (cf. Remark \ref{rem.AdmissibleUnitary}). Denote by $\widetilde{V}_E\in \mathcal{L}_{B\otimes \iota(C(\mathbb{H}))}\big(E\underset{\widetilde{\beta}}{\otimes} (B\otimes \iota(C(\mathbb{H}))), B\otimes \iota(C(\mathbb{H}))\big)$ the admissible unitary associated to the action $\widetilde{\delta}_E$ of $\mathbb{G}$ on $E$. We have an obvious commutative diagram:
		$$
			\xymatrix{
				\mbox{$E\underset{\beta}{\otimes} (B\otimes C(\mathbb{H}))$}\ar[d]_{\mbox{$id_E\underset{\beta}{\otimes} (id_B\otimes \iota)$}}\ar[r]^{\mbox{$V_E$}}&\mbox{$B\otimes C(\mathbb{H})$}\ar[d]^{\mbox{$id_B\otimes \iota$}}\\
				\mbox{$E\underset{\widetilde{\beta}}{\otimes} (B\otimes \iota(C(\mathbb{H})))$}\ar[r]_{\mbox{$\widetilde{V}_E$}}&\mbox{$B\otimes \iota(C(\mathbb{H}))$}&}
		$$
		which allows to write the following:
		$$\widetilde{\delta}_E(\xi)(\iota (x))=\widetilde{V}_E(\xi\underset{\widetilde{\beta}}{\otimes} \iota(x))=\widetilde{V}_E\circ (id_E\underset{\beta}{\otimes} (id_B\otimes \iota))(\xi\underset{\beta}{\otimes} x)=(id_B\otimes \iota)\circ V_E(\xi\underset{\beta}{\otimes} x)=(id_B\otimes \iota)\circ\delta_E(\xi)(x)\mbox{,}$$
		for all $\xi\in E$ and $x\in B\otimes C(\mathbb{H})$. It is thus straightforward to check that $(\text{Ind}^{\mathbb{G}}_{\mathbb{H}}(H), \widetilde{\delta}_H)\in Obj(\mathscr{M}_{\widetilde{\beta}})$. 
		
		On the other hand,  observe that if $(Y,\delta_Y)\in J$, then the $\mathbb{G}$-equivariant Hilbert $\text{Ind}^{\mathbb{G}}_{\mathbb{H}}(B)$-module $(\text{Ind}^{\mathbb{G}}_{\mathbb{H}}(Y), \widetilde{\delta}_Y)$ is again irreducible in $\mathscr{M}_{\widetilde{\beta}}$. Namely, we have $\mathcal{L}_{\mathbb{G}}\big(\text{Ind}^{\mathbb{G}}_{\mathbb{H}}(Y)\big)\subset \mathcal{L}_{\mathbb{H}}\big(Y\big)=\mathbb{C}\cdot 1$ because for every $T\in \mathcal{L}_{\mathbb{G}}\big(\text{Ind}^{\mathbb{G}}_{\mathbb{H}}(Y)\big)$ and every $\xi\in Y$ we write: 
		$$(id_Y\otimes \iota)\delta_Y(T(\xi))=\widetilde{\delta}_Y(T(\xi))=(T\otimes id)\widetilde{\delta}_Y(\xi)=(T\otimes id)(id_Y\otimes \iota)\delta_Y(\xi)=(id_Y\otimes \iota)(T\otimes id)\delta_Y(\xi),$$
		which implies that $\delta_Y(T(\xi))=(T\otimes id)\delta_Y(\xi)$ for all $\xi\in Y$, that is, $T\in \mathcal{L}_{\mathbb{H}}\big(Y\big)$. In other words, $J\subset J'$. 
		
		Since both  $\mathscr{N}_\beta$ and $\mathscr{M}_{\widetilde{\beta}}$ are semi-simple because they are constructed from ergodic actions of $\mathbb{G}$, the previous discussion yields a fully faithful functor between $\mathscr{N}_\beta$ and $\mathscr{M}_{\widetilde{\beta}}$, so that we view $\mathscr{N}_\beta$ as a (full) $\mathscr{R}$ep$(\mathbb{H})$-module C$^*$-subcategory of $\mathscr{M}_{\widetilde{\beta}}$.
		\end{proof}
		
		\begin{lem}\label{lem.NBetaFullSubIndCat}
			We keep the previous notations. $\mathscr{N}_{\beta}$ is a full $\mathscr{Rep}(\mathbb{H})$-module C$^*$-subcategory of $\text{Ind}^{\mathscr{Rep}(\mathbb{G})}_{\mathscr{Rep}(\mathbb{H})}(\mathscr{N}_\beta)$.	
		\end{lem}
		\begin{proof}
			 The association $(E, \delta_E)\mapsto \epsilon\boxtimes (E,\delta_E)$ from $\mathscr{N}_\beta$ to $\mathscr{Rep}(\mathbb{G})\boxtimes \mathscr{N}_\beta$ defines a C$^*$-functor which induces a fully faithful $\mathscr{Rep}(\mathbb{H})$-module functor between $\mathscr{N}_\beta$ and $\mathscr{Rep}(\mathbb{G})\underset{\mathscr{Rep}(\mathbb{H})}{\boxtimes} \mathscr{N}_\beta$, so that $\mathscr{N}_\beta$ is viewed as a (full) $\mathscr{R}$ep$(\mathbb{H})$-module C$^*$-subcategory of $\text{Ind}^{\mathscr{Rep}(\mathbb{G})}_{\mathscr{Rep}(\mathbb{H})}(\mathscr{N}_\beta)=\mathscr{Rep}(\mathbb{G})\underset{\mathscr{Rep}(\mathbb{H})}{\boxtimes} \mathscr{N}_\beta$.
		\end{proof}
		
		If $\widehat{\mathbb{H}}$ is divisible in $\widehat{\mathbb{G}}$, then $\text{R}(\mathbb{G})\cong \underset{\Omega}{\bigoplus} \text{R}(\mathbb{H})$ as (right) based $\text{R}(\mathbb{H})$-modules and Lemma \ref{lem.InducedTorsionModules} assures that $\text{Ind}^{\text{R}(\mathbb{G})}_{\text{R}(\mathbb{H})}\Big(\text{Fus}(\mathscr{N}_\beta)\Big)=\text{R}(\mathbb{G})\underset{\text{R}(\mathbb{H})}{\odot}\text{Fus}(\mathscr{N}_\beta)$ is a based $\text{R}(\mathbb{G})$-module with basis $\widetilde{J}:=\{\mathbbold{1}_t\odot (Y, \delta_Y) \ |\ t\in \Omega, (Y,\delta_Y)\in J\}$. Here, $\mathbbold{1}_t\in \text{R}(\mathbb{G})$ denotes the copy of $\mathbbold{1}_{\text{R}(\mathbb{H})}=\epsilon$ at position $t\in \Omega$ in the decomposition $\text{R}(\mathbb{G})\cong \underset{\Omega}{\bigoplus} \text{R}(\mathbb{H})$. Recall from Remark \ref{rem.DivisibilityDecomp} that $\epsilon=\mathbbold{1}_{t_0}$. So $J$ can be embedded into $\widetilde{J}$ by sending $(Y, \delta_Y)$ to $\epsilon\odot (Y, \delta_Y)$.
		
		By Corollary \ref{cor.TorsionInducedCateg} we know that, if $\widehat{\mathbb{H}}$ is divisible in $\widehat{\mathbb{G}}$, then $\text{Ind}^{\mathscr{Rep}(\mathbb{G})}_{\mathscr{Rep}(\mathbb{H})}(\mathscr{N}_\beta)$ is a torsion $\mathscr{Rep}(\mathbb{G})$-module C$^*$-category (hence semi-simple). As such it is attached, as explained at the end of Section \ref{sec.DiscussionAboutKennyMakoto}, to a torsion action of $\mathbb{G}$, say $(A,\delta)$. In other words, we have $\text{Ind}^{\mathscr{Rep}(\mathbb{G})}_{\mathscr{Rep}(\mathbb{H})}(\mathscr{N}_\beta)\cong \mathscr{M}_{\delta}$. We want to show that $\mathscr{M}_{\delta}\cong \mathscr{M}_{\widetilde{\beta}}$ as $\mathscr{Rep}(\mathbb{G})$-module C$^*$-categories as soon as $\widehat{\mathbb{H}}$ is divisible in $\widehat{\mathbb{G}}$.
		
		On the one hand, by Lemma \ref{lem.NBetaFullSubMBetaTilde} we view $\mathscr{N}_\beta$ as a (full) $\mathscr{R}$ep$(\mathbb{H})$-module C$^*$-subcategory of $\mathscr{M}_{\widetilde{\beta}}$. In particular, we identify $(B,\beta)\in \text{Irr}(\mathscr{N}_\beta)$ to $(\text{Ind}^{\mathbb{G}}_{\mathbb{H}}(B), \widetilde{\beta})\in \text{Irr}(\mathscr{M}_{\widetilde{\beta}})$ inside $\mathscr{M}_{\widetilde{\beta}}$. Let us denote by $B_{\widetilde{\beta}}$ the object $(\text{Ind}^{\mathbb{G}}_{\mathbb{H}}(B), \widetilde{\beta})$ as an (irreducible) object in the category $\mathscr{M}_{\widetilde{\beta}}$. On the other hand, by Lemma \ref{lem.NBetaFullSubIndCat} we view $\mathscr{N}_{\beta}$ as a (full) $\mathscr{Rep}(\mathbb{H})$-module C$^*$-subcategory of $\text{Ind}^{\mathscr{Rep}(\mathbb{G})}_{\mathscr{Rep}(\mathbb{H})}(\mathscr{N}_\beta)$. In particular, we identify $(B,\beta)\in \text{Irr}(\mathscr{N}_\beta)$ to $\epsilon\odot (B,\beta)\in \widetilde{J}\cong \text{Irr}\Big(\text{Ind}^{\mathscr{Rep}(\mathbb{G})}_{\mathscr{Rep}(\mathbb{H})}(\mathscr{N}_\beta)\Big)$ inside $\mathscr{Rep}(\mathbb{G})\underset{\mathscr{Rep}(\mathbb{H})}{\boxtimes} \mathscr{N}_\beta$. Let us denote by $B_{\text{Ind}}$ the object $\epsilon\odot (B,\beta)$ as an (irreducible) object in the category $\text{Ind}^{\mathscr{Rep}(\mathbb{G})}_{\mathscr{Rep}(\mathbb{H})}(\mathscr{N}_\beta)\cong \mathscr{M}_{\delta}$. 
		
		\begin{lem}\label{lem.IsomAlgObjCat}
			We keep the previous notations and those from the end of Section \ref{sec.DiscussionAboutKennyMakoto}. $(\mathcal{B}^{B_{\text{Ind}}}_{B_{\text{Ind}}}, \delta_{B_{\text{Ind}}, \mathscr{M}_\delta})\cong (\mathcal{B}^{B_{\widetilde{\beta}}}_{B_{\widetilde{\beta}}}, \delta_{B_{\widetilde{\beta}}, \mathscr{M}_{\widetilde{\beta}}})$ as soon as $\widehat{\mathbb{H}}$ is divisible in $\widehat{\mathbb{G}}$.
		\end{lem}
		\begin{proof}
			Assume that $\widehat{\mathbb{H}}$ is divisible in $\widehat{\mathbb{G}}$. First of all, we claim that we have an isomorphism between the homomorphism spaces:
			\begin{equation}\label{eq.IdentificationAlgbActions}
		\begin{split}
			\text{Hom}_{\mathscr{M}_\delta}(u^x\bullet B_{\text{Ind}}, B_{\text{Ind}})\cong \text{Hom}_{\mathscr{M}_{\widetilde{\beta}}}(u^x\bullet B_{\widetilde{\beta}}, B_{\widetilde{\beta}}),
		\end{split}
		\end{equation}
		for all $x\in \text{Irr}(\mathbb{G})$. 
		
		On the one hand, observe that, by construction, $\text{Stab}(B_{\text{Ind}})$ in the induced category can only be formed by irreducible $\mathbb{H}$-representations. Indeed, this follows from Lemma \ref{lem.StabilisersDivisibleInd} since $B_{\text{Ind}}=\epsilon\odot (B,\beta)\in \widetilde{J}\cong \text{Irr}\Big(\text{Ind}^{\mathscr{Rep}(\mathbb{G})}_{\mathscr{Rep}(\mathbb{H})}(\mathscr{N}_\beta)\Big)$. On the other hand, $\text{Stab}(B_{\widetilde{\beta}})$ in $\mathscr{M}_{\widetilde{\beta}}$ can only be formed by irreducible $\mathbb{H}$-representations too. The argument for this is similar to the one for Lemma \ref{lem.StabilisersDivisibleInd}. Namely, by definition we have:
		$$\text{Stab}(B_{\widetilde{\beta}})=\{x\in \text{Irr}(\mathbb{G})\ |\ B_{\widetilde{\beta}}\subset H_x\otimes B_{\widetilde{\beta}}\}.$$
		
		Notice that $B_{\widetilde{\beta}}\cong H_{\epsilon}\otimes B$ as the irreducible object $(B, \beta)$ inside $\mathscr{M}_{\widetilde{\beta}}$. Given $x\in \text{Irr}(\mathbb{G})$ and using the divisibility assumption, one can write $x\cong t_x\otimes y_x$, for some $t_x\in\text{Irr}(\mathbb{G})$ and $y_x\in\text{Irr}(\mathbb{H})$ (cf. Remark \ref{rem.DivisibilityDecomp}). Hence, $H_x\otimes B_{\widetilde{\beta}}\cong (H_{t_x}\otimes H_{y_x})\otimes B\cong H_{t_x}\otimes (H_{y_x}\otimes B)$. Next, we decompose the object $H_{y_x}\otimes B\in \text{Obj}(\mathscr{N_\beta})$ as a direct sum of irreducibles inside $\mathscr{N_\beta}$, say $H_{y_x}\otimes B\cong \underset{k\in I(y_x, B)}{\bigoplus} Y_k$, where $\{(Y_{k}, \delta_{Y_{k}})\}_{k\in I(y_x, B)}\subset J=\text{Irr}(\mathscr{N}_{\beta})$ and $I(y_x, B)$ denotes the index set for the irreducible objects in $\mathscr{N_\beta}$ appearing in the decomposition of $H_{y_x}\otimes B$ (in particular, $Y_k\neq 0$, for all $k\in I(y_x, B)$). Hence, $H_{t_x}\otimes (H_{y_x}\otimes B)\cong \underset{k\in I(y_x, B)}{\bigoplus} H_{t_x}\otimes Y_k$. Accordingly, $x\in \text{Stab}(B_{\widetilde{\beta}})$ if and only if $H_{\epsilon}\otimes B\subset \underset{k\in I(y_x, B)}{\bigoplus} H_{t_x}\otimes Y_k$. Since $H_{\epsilon}\otimes B$ is irreducible, it must be $H_{\epsilon}\otimes B\subset H_{t_x}\otimes Y_{k_0}$, for some $k_0\in I(y_x, B)$. By Frobenius reciprocity, this is equivalent to say that $Y_{k_0}\subset \overline{H_{t_x}}\otimes B$. However, $(Y_{k_0}, \delta_{Y_{k_0}})\in J$, which means that it arises as a submodule of $H_{y_0}\otimes B$, for some $y_0\in \text{Irr}(\mathbb{H})$. In other words, we have that $Y_{k_0}\subset (\overline{H_{t_x}}\otimes B)\cap (H_{y_0}\otimes B)=(\overline{H_{t_x}}\cap H_{y_0})\otimes B$. Since both $t_x$ and $y_0$ are irreducible representations, one has $\overline{H_{t_x}}\cap H_{y_0}=0$ if $t_x\neq y_0$. In this case, we would have $Y_{k_0}=0$, which is not possible by our choice of $k_0$. Consequently, it must be $t_x= y_0\in \text{Irr}(\mathbb{H})$. It follows that $t_x=t_0=\epsilon$ (cf. Remark \ref{rem.DivisibilityDecomp}). This implies that $x\cong \epsilon\otimes y_x=y_x\in\text{Irr}(\mathbb{H})$, which yields the claim.\footnote{Similarly as in Remark \ref{rem.StabilisersDivisibleInd}, this argument shows that $\text{Stab}(B_{\widetilde{\beta}})=\text{Stab}(B)$.}
		 
		Therefore, our claim comes down to show the isomorphism (\ref{eq.IdentificationAlgbActions}) above for all $y\in\text{Irr}(\mathbb{H})$, which is straightforward. In other words, $\text{Hom}_{\mathscr{M}_\delta}(u^x\bullet B_{\text{Ind}}, B_{\text{Ind}})$ reduces to $\text{Hom}_{\mathscr{N}_\beta}(u^y\bullet B, B)$ inside $\mathscr{M}_{\delta}$ and $\text{Hom}_{\mathscr{M}_{\widetilde{\beta}}}(u^x\bullet B_{\widetilde{\beta}}, B_{\widetilde{\beta}})$ reduces to $\text{Hom}_{\mathscr{N}_{\beta}}(u^y\bullet B, B)$ inside $\mathscr{M}_{\widetilde{\beta}}$. Consequently, the $*$-algebra $\mathcal{B}^{B_{\text{Ind}}}_{B_{\text{Ind}}}$ reduces to $\mathcal{B}^{B}_{B}$ inside $\mathscr{M}_{\delta}$ and the $*$-algebra $\mathcal{B}^{B_{\widetilde{\beta}}}_{B_{\widetilde{\beta}}}$ reduces to $\mathcal{B}^{B}_{B}$ inside $\mathscr{M}_{\widetilde{\beta}}$ together with the corresponding $\mathbb{G}$-actions, which yields the claim.
		\end{proof}

	In conclusion, we obtain the following:
	\begin{theo}\label{theo.FusionModulesInducedAction}
		Let $\mathbb{G}$ and $\mathbb{H}$ be two compact quantum groups such that $\widehat{\mathbb{H}}<\widehat{\mathbb{G}}$. If $\widehat{\mathbb{H}}$ is divisible in $\widehat{\mathbb{G}}$, then we have:
		$$\text{Ind}^{\mathscr{Rep}(\mathbb{G})}_{\mathscr{Rep}(\mathbb{H})}(\mathscr{N}_\beta)\cong \mathscr{M}_{\widetilde{\beta}},$$
		
		as $\mathscr{R}$ep$(\mathbb{G})$-module C$^*$-categories, for all torsion action $(B, \beta)$ of $\mathbb{H}$. In particular, $\text{Ind}^{\text{R}(\mathbb{G})}_{\text{R}(\mathbb{H})}\Big(\text{Fus}(\mathscr{N}_\beta)\Big)\cong \text{Fus}\big(\mathscr{M}_{\widetilde{\beta}}\big)$ as based $\text{R}(\mathbb{G})$-modules, for all torsion action $(B, \beta)$ of $\mathbb{H}$.
	\end{theo}
	\begin{proof}
		By virtue of Lemma \ref{lem.IsomAlgObjCat}, one has $(\mathcal{B}^{B_{\text{Ind}}}_{B_{\text{Ind}}}, \delta_{B_{\text{Ind}}, \mathscr{M}_\delta})\cong (\mathcal{B}^{B_{\widetilde{\beta}}}_{B_{\widetilde{\beta}}}, \delta_{B_{\widetilde{\beta}}, \mathscr{M}_{\widetilde{\beta}}})$ as soon as $\widehat{\mathbb{H}}$ is divisible in $\widehat{\mathbb{G}}$, for all torsion action $(B, \beta)$ of $\mathbb{H}$. According to the one-to-one correspondence between ergodic actions of $\mathbb{G}$ and connected $\mathscr{Rep}(\mathbb{G})$-module C$^*$-categories discussed at the end of Section \ref{sec.DiscussionAboutKennyMakoto}, one deduces that $(A, \delta) \underset{\mathbb{G}}{\sim}(B^{B_{\text{Ind}}}_{B_{\text{Ind}}}, \delta_{B_{\text{Ind}}, \mathscr{M}_\delta})\cong (B^{B_{\widetilde{\beta}}}_{B_{\widetilde{\beta}}}, \delta_{B_{\widetilde{\beta}}, \mathscr{M}_{\widetilde{\beta}}})\underset{\mathbb{G}}{\sim} (\text{Ind}^{\mathbb{G}}_{\mathbb{H}}(B), \widetilde{\beta})$, which implies $\mathscr{M}_{\delta}\cong \mathscr{M}_{\widetilde{\beta}}$ as we wanted to show. Finally, the identification $\text{Ind}^{\text{R}(\mathbb{G})}_{\text{R}(\mathbb{H})}\Big(\text{Fus}(\mathscr{N}_\beta)\Big)\cong \text{Fus}\big(\mathscr{M}_{\widetilde{\beta}}\big)$ follows from Corollary \ref{cor.TorsionInducedCateg}.
	\end{proof}
	\begin{rem}
		Note that the previous discussion and so the previous proposition remains valid for ergodic actions of $\mathbb{H}$ (not necessarily finite dimensional).
	\end{rem}
	
	This theorem allows to recover Theorem \ref{theo.TorsionFreenessDivisible} by adapting the same argument appearing there at the level of torsion modules associated to torsion actions.
	\begin{cor}\label{cor.}
		Let $\mathbb{G}$ and $\mathbb{H}$ be two compact quantum groups such that $\widehat{\mathbb{H}}<\widehat{\mathbb{G}}$ is divisible. If $\widehat{\mathbb{G}}$ is torsion-free, then $\widehat{\mathbb{H}}$ is torsion-free.
	\end{cor}
	
	In view of the examples discussed in the begining of this section, when the discrete quantum subgroup $\widehat{\mathbb{H}}$ is not divisible in $\widehat{\mathbb{G}}$, a non-trivial torsion action of $\mathbb{H}$ can be realized as the trivial torsion action of $\mathbb{G}$ even when $\widehat{\mathbb{G}}$ is torsion-free, which prevents $\widehat{\mathbb{H}}$ to be torsion-free. The previous proposition implies that, under the divisibility assumption, induction of torsion actions preserves non-trivial actions.
	\begin{cor}
		Let $\mathbb{G}$ and $\mathbb{H}$ be two compact quantum groups such that $\widehat{\mathbb{H}}<\widehat{\mathbb{G}}$ is divisible. If $(B, \beta)$ is a non-trivial torsion action of $\mathbb{H}$, then $(\text{Ind}^{\mathbb{G}}_{\mathbb{H}}(B), \widetilde{\beta})$ is a non-trivial torsion action of $\mathbb{G}$.
	\end{cor}
	\begin{proof}
		Let $(B, \beta)$ be a non-trivial torsion action of $\mathbb{H}$. Assume that $(\text{Ind}^{\mathbb{G}}_{\mathbb{H}}(B), \widetilde{\beta})$ is the trivial torsion action of $\mathbb{G}$. This means, according to the discussion at the end of Section \ref{sec.DiscussionAboutKennyMakoto}, that $\mathscr{M}_{\widetilde{\beta}}\cong \mathscr{Rep}(\mathbb{G})$. By Theorem \ref{theo.FusionModulesInducedAction} this implies that $\text{Ind}^{\mathscr{Rep}(\mathbb{G})}_{\mathscr{Rep}(\mathbb{H})}(\mathscr{N}_\beta)\cong \mathscr{Rep}(\mathbb{G})$ hence $\text{Ind}^{\text{R}(\mathbb{G})}_{\text{R}(\mathbb{H})}\Big(\text{Fus}(\mathscr{N}_\beta)\Big)\cong \text{R}(\mathbb{G})$. At this point, the same argument as the one of Theorem \ref{theo.TorsionFreenessDivisible} would yield $\text{Fus}(\mathscr{N}_\beta)\cong \text{R}(\mathbb{H})$ hence $\mathscr{N}_\beta\cong \mathscr{Rep}(\mathbb{H})$. This means, thanks again to the discussion at the end of Section \ref{sec.DiscussionAboutKennyMakoto}, that $(B,\beta)$ would be equivariantly Morita equivalent to the trivial torsion action of $\mathbb{H}$, which contradicts the choice of $\beta$. Therefore, $(\text{Ind}^{\mathbb{G}}_{\mathbb{H}}(B), \widetilde{\beta})$ must be non-trivial.
	\end{proof}

	\subsection{Some permanence properties of torsion-freeness}
	
	\begin{proSec}\label{pro.DivisibleConstructions}
		The following properties hold.
		\begin{enumerate}[i)]
			\item Let $\Gamma$ be a discrete group, $\mathbb{G}$ a compact quantum group. Assume that $\Gamma$ acts on $\mathbb{G}$ by quantum automorphisms, $\alpha$. Then both $\Gamma$ and $\widehat{\mathbb{G}}$ are divisible discrete quantum subgroups in $\widehat{\Gamma\underset{\alpha}{\ltimes}\mathbb{G}}$.
			\item Let $\Gamma$ be a discrete group, $G$ a compact group. Assume that $(\Gamma, G)$ is a matched pair with actions $\Gamma \overset{\alpha}{\curvearrowright} G$ and $\Gamma \overset{\beta}{\curvearrowleft} G$. Then $\widehat{G} \mbox{ is a divisible in } \widehat{\Gamma {_\alpha}\bowtie_{\beta}G}\Leftrightarrow \beta \mbox{ is trivial.}$ In this case, $\Gamma$ is also a divisible discrete quantum subgroup in $ \widehat{\Gamma {_\alpha}\bowtie_{\beta}G}$.
			\item Let $\mathbb{G}$ and $\mathbb{H}$ be two compact quantum groups. Then both $\widehat{\mathbb{G}}$ and $\widehat{\mathbb{H}}$ are divisible discrete quantum subgroups in $\widehat{\mathbb{G}\times \mathbb{H}}$.
			\item Let $\mathbb{G}$ and $\mathbb{H}$ be two compact quantum groups. Then both $\widehat{\mathbb{G}}$ and $\widehat{\mathbb{H}}$ are divisible discrete quantum subgroups in $\widehat{\mathbb{G}\ast \mathbb{H}}$. 
		\end{enumerate}
	\end{proSec}
	\begin{proof}
		\begin{enumerate}[i)]
			\item The description of the representation theory of $\mathbb{F}:=\Gamma\underset{\alpha}{\ltimes}\mathbb{G}$ recalled in Section \ref{sec.Notations} yields that $\widehat{\mathbb{G}}$ and $\Gamma$ are divisible in $\widehat{\mathbb{F}}$. Namely, take an irreducible representation $y:=(\gamma, x)\in \text{Irr}(\mathbb{F})$ with $\gamma\in \Gamma$ and $x\in \text{Irr}(\mathbb{G})$. Then $\gamma=(\gamma, \epsilon_{\mathbb{G}})\in [y]$ in $\text{Irr}(\mathbb{F})/ \text{Irr}(\mathbb{G})$ because $(\gamma^{-1}, \epsilon_{\mathbb{G}})\otimes (\gamma, x)=(e, x)=x\in \text{Irr}(\mathbb{G})$. Likewise, we have that $x=(e,x)\in [y]$ in $\Gamma\backslash \text{Irr}(\mathbb{F})$ because $(\gamma,x)\otimes (e, \overline{x})=(\gamma, \epsilon_{\mathbb{G}})=\gamma\in \Gamma$. Consequently, $\widehat{\mathbb{G}}$ is divisible in $\widehat{\mathbb{F}}$ because for all $s\in \text{Irr}(\mathbb{G})$ we have that $(\gamma, \epsilon_{\mathbb{G}})\otimes s=(\gamma, \epsilon_{\mathbb{G}})\otimes (e,s)=(\gamma,s)\in \text{Irr}(\mathbb{F})$. Likewise, $\Gamma$ is divisible in $\widehat{\mathbb{F}}$ because for all $s\in \Gamma$ we have that $s\otimes (e, x)=(s, \epsilon_{\mathbb{G}})\otimes (e, x)=(s, x)\in \text{Irr}(\mathbb{F})$.
			\item On the one hand, if $\Gamma \overset{\beta}{\curvearrowleft} G$ is the trivial action, then we simply have $\Gamma {_\alpha}\bowtie_{\beta}G= \Gamma\underset{\alpha}{\ltimes} G$; so that property $(i)$ above yields that both $\Gamma$ and $\widehat{G}$ are divisible discrete quantum subgroups in $\widehat{\Gamma\underset{\alpha}{\ltimes} G}$.
			
			On the other hand, in order to show that divisibility of $\widehat{G}$ implies triviality of $\beta$, we are going to use the description of the representation theory of $\mathbb{F}:= \Gamma {_\alpha}\bowtie_{\beta}G$ recalled in Section \ref{sec.Notations}. Let $y=\gamma(x)\in \text{Irr}(\mathbb{F})$ an irreducible representation of the bicrossed product with $\gamma\in \Gamma$ and $x\in \text{Irr}(G_\gamma)$. If $y'=\gamma'(x')\in \text{Irr}(\mathbb{F})$ is another irreducible representation of $\mathbb{F}$ with $\gamma'\in \Gamma$ and $x'\in \text{Irr}(G_{\gamma'})$, then, by virtue of the fusion rules of the bicrossed product (recall Equation (\ref{eq.FusionBicrossed})), $y'\in [y]$ in $\text{Irr}(\mathbb{F})/ \text{Irr}(G)$ if and only if $y'=\gamma(x')$ for some $x'\in \text{Irr}(G_\gamma)$ (recall in addition that $\#[e]=1$ in the orbit space $\Gamma/G$ and that $(\gamma\cdot G)^{-1}=\gamma^{-1}\cdot G$). 
			
			Now, assume that $\widehat{G}$ is divisible in $\widehat{\mathbb{F}}$. This means that given $[y]=[\gamma(x)]\in \text{Irr}(\mathbb{F})/\text{Irr}(G)$ we can find a representative, say $\gamma(x')\in [y]$ for some $x'\in \text{Irr}(G)$ such that for all irreducible $s\in \text{Irr}(G)$, $\gamma(x')\otimes s=\gamma(x')\otimes e(s)$ is an irreducible representation of $\mathbb{F}$. For this recall again the fusion rules for a bicrossed product (Equation (\ref{eq.FusionBicrossed})): given $\mu(z)\in \text{Irr}(\mathbb{F})$ with $\mu\in \Gamma$ and $z\in \text{Irr}(G_\mu)$, we have:
			$${\rm dim}_\mathbb{F}(\mu(z),\gamma(x')\otimes e(s))=\left\{\begin{array}{ll}\underset{r\in\mu\cdot G\cap\gamma\cdot G}{\sum}\frac{1}{\vert r\cdot G\vert}{\rm dim}_{G_r}(z\circ\psi^{\mu}_{r,r},x'\underset{r}{\otimes}s)&\text{ if }\mu\cdot G\cap\gamma\cdot G\neq\emptyset,\\0&\text{ otherwise.}\end{array}\right.$$
			
			Then, in order to $\gamma(x')\otimes e(s)$ to be irreducible, it is necessary that $|\mu\cdot G\cap\gamma\cdot G|=1$, for all $\mu,\gamma\in \Gamma$. This implies that all classes in the orbit space $\Gamma/G$ are trivial, that is, $\beta$ is the trivial action as we wanted to show. Observe moreover that in this case, a representative of $y=\gamma(x)$ in $\text{Irr}(\mathbb{F})/\text{Irr}(G)$ can be chosen as $\gamma(\epsilon_G)$ and the above fusion rules become:
			$${\rm dim}_\mathbb{F}(\mu(z),\gamma(\epsilon_G)\otimes e(s))=\left\{\begin{array}{ll}{\rm dim}_{G}(z\circ\psi^{\mu}_{\gamma,\gamma}, s)&\text{ if }\mu=\gamma, \\0&\text{ otherwise,}\end{array}\right.$$
			where recall that $\psi^\gamma_{\gamma, \gamma}=id$. In other words, $\gamma(\epsilon_G)\otimes e(s)=\gamma(s)$; recovering then the same computations as done for a quantum semi-direct product in $(i)$ above.
			
			\item The description of the representation theory of $\mathbb{F}:=\mathbb{G}\times \mathbb{H}$ recalled in Section \ref{sec.Notations} yields that $\widehat{\mathbb{G}}$ and $\widehat{\mathbb{H}}$ are divisible in $\widehat{\mathbb{F}}$. Namely, take an irreducible representation $y:=(x, z)\in \text{Irr}(\mathbb{F})$ with $x\in \text{Irr}(\mathbb{G})$ and $z\in \text{Irr}(\mathbb{H})$. Then $x=(x, \epsilon_{\mathbb{H}})\in [y]$ in $\text{Irr}(\mathbb{H})\backslash \text{Irr}(\mathbb{F})$ because $(x, z)\otimes (\overline{x}, \epsilon_{\mathbb{H}})=(\epsilon_{\mathbb{G}}, z)=z\in \text{Irr}(\mathbb{H})$. Likewise, we have that $z=(\epsilon_{\mathbb{G}},z)\in [y]$ in $\text{Irr}(\mathbb{G})\backslash \text{Irr}(\mathbb{F})$ because $(x,z)\otimes (\epsilon_{\mathbb{G}}, \overline{z})=(x, \epsilon_{\mathbb{H}})=x\in \text{Irr}(\mathbb{G})$. Consequently, $\widehat{\mathbb{G}}$ is divisible in $\widehat{\mathbb{F}}$ because for all $s\in \text{Irr}(\mathbb{G})$ we have that $s\otimes (\epsilon,z)=(s,\epsilon_{\mathbb{H}})\otimes (\epsilon_{\mathbb{G}},z)=(s,z)\in \text{Irr}(\mathbb{F})$. Likewise, $\widehat{\mathbb{H}}$ is divisible in $\widehat{\mathbb{F}}$ because for all $s\in \text{Irr}(H)$ we have that $(x, \epsilon_{\mathbb{H}})\otimes s=(x, \epsilon_{\mathbb{H}})\otimes (\epsilon_{\mathbb{G}},s)=(x, s)\in \text{Irr}(\mathbb{F})$.
			\item The description of the representation theory of $\mathbb{F}:=\mathbb{G}\ast \mathbb{H}$ recalled in Section \ref{sec.Notations} yields that both $\widehat{\mathbb{G}}$ and $\widehat{\mathbb{H}}$ are divisible in $\widehat{\mathbb{F}}$. Let us show that $\widehat{\mathbb{G}}$ is divisible in $\widehat{\mathbb{F}}$ (the proof for $\widehat{\mathbb{H}}$ is analogous). Take any irreducible representation of $\mathbb{F}$, which is given by an alternating word in $\text{Irr}(\mathbb{G})$ and $\text{Irr}(\mathbb{H})$, say $y:=x_{i_1}z_{i_2}\ldots x_{i_{n-1}}z_{i_n}$. By definition, a representative of $y$ in $\text{Irr}(\mathbb{G})\backslash \text{Irr}(\mathbb{F})$ (resp. in $\text{Irr}(\mathbb{F})/\text{Irr}(\mathbb{G})$) is an irreducible representation $y'\in \text{Irr}(\mathbb{F})$ such that $y\otimes \overline{y'}$ (resp. $\overline{y'}\otimes y$) contains an irreducible representation of $\mathbb{G}$ (inside $\mathbb{F}$). The latter is possible if and only if the tensor product $y\otimes \overline{y'}$ (resp. $\overline{y'}\otimes y$) reduces to a single letter in $\text{Irr}(\mathbb{G})$. 
		
		Assume that $y$ starts in $\text{Irr}(\mathbb{G})$, then it is enough to put $y':=z_{i_2}\ldots x_{i_{n-1}}z_{i_n}$. The fusion rules of a quantum free product yield that $y\otimes \overline{y'}=x_{i_1}\in \text{Irr}(\mathbb{G})$. 
		
		In this situation, given any $s\in \text{Irr}(\mathbb{G})$ we have to prove that $s\otimes y'\in \text{Irr}(\mathbb{F})$. Since $y$ starts in $\text{Irr}(\mathbb{G})$, then $y'$ starts in $\text{Irr}(\mathbb{H})$, so that the fusion rules of a quantum free product yield that $s\otimes y'=sy'\in \text{Irr}(\mathbb{F})$.
		
		Assume that $y$ starts in $\text{Irr}(\mathbb{H})$ and ends in $\text{Irr}(\mathbb{G})$, then it is enough to put $y':=x_{i_{1}}z_{i_2}\ldots x_{i_{n-1}}$. The fusion rules of a quantum free product yield that $\overline{y'}\otimes y=z_{i_n}\in \text{Irr}(\mathbb{G})$. Since $y$ ends in $\text{Irr}(\mathbb{G})$, then $y'$ ends in $\text{Irr}(\mathbb{H})$ and the fusion rules of a quantum free product yield that $y'\otimes s=y's\in \text{Irr}(\mathbb{F})$, for every $s\in \text{Irr}(\mathbb{G})$.
		
		Assume that $y$ starts and ends in $\text{Irr}(\mathbb{H})$, then we can not choose any representative $y'$ of $y$ either in $\text{Irr}(\mathbb{G})\backslash \text{Irr}(\mathbb{F})$ or in $\text{Irr}(\mathbb{F})/\text{Irr}(\mathbb{G})$ such that either $y\otimes \overline{y'}$ or $\overline{y'}\otimes y$ reduces to a single letter in $\text{Irr}(\mathbb{G})$. In other words, the class of $[y]$ is formed only by $y$ itself (notice that $y\otimes \overline{y}=\epsilon_{\mathbb{H}}\cong\epsilon_{\mathbb{G}}$ in $\text{Irr}(\mathbb{F})$). In this case, it is obvious that $s\otimes y'=sy'\in \text{Irr}(\mathbb{F})$, for every $s\in \text{Irr}(\mathbb{G})$.
		\end{enumerate}
	\end{proof}
	\begin{corSec}\label{cor.StabilityTorsionFree}
		The following stability properties for torsion-freeness hold.
		\begin{enumerate}[i)]
			\item Let $\Gamma$ be a discrete group, $\mathbb{G}$ a compact quantum group. Assume that $\Gamma$ acts on $\mathbb{G}$ by quantum automorphisms, $\alpha$. Then $\widehat{\Gamma\underset{\alpha}{\ltimes}\mathbb{G}}\mbox{ is torsion-free } \Leftrightarrow \Gamma \mbox{ and } \widehat{\mathbb{G}} \mbox{ are torsion-free.}$
			\item Let $\Gamma$ be a discrete group, $G$ a compact group. Assume that $(\Gamma, G)$ is a matched pair with actions $\Gamma \overset{\alpha}{\curvearrowright} G$ and $G \overset{\beta}{\curvearrowright} \Gamma$. Then $\widehat{\Gamma {_\alpha}\bowtie_{\beta}G}\mbox{ is torsion-free } \Rightarrow \Gamma \mbox{ and } \widehat{G} \mbox{ are torsion-free.}$
			\item Let $\mathbb{G}$ and $\mathbb{H}$ be two compact quantum groups. Then $\widehat{\mathbb{G}\times \mathbb{H}} \mbox{ is torsion-free } \Leftrightarrow \widehat{\mathbb{G}} \mbox{ and } \widehat{\mathbb{H}} \mbox{ are torsion-free.}$
			\item Let $\mathbb{G}$ and $\mathbb{H}$ be two compact quantum groups. Then $\widehat{\mathbb{G}\ast \mathbb{H}} \mbox{ is torsion-free } \Leftrightarrow \widehat{\mathbb{G}} \mbox{ and } \widehat{\mathbb{H}} \mbox{ are torsion-free.}$
		\end{enumerate}
	\end{corSec}
	\begin{proof}
		Implication \fbox{$\Rightarrow$} in $(i)$ is a consequence of Proposition \ref{pro.DivisibleConstructions} and Theorem \ref{theo.TorsionFreenessDivisible}. Implication \fbox{$\Leftarrow$} is contained in \cite[Theorem 3.1.1]{RubenSemiDirect}. Implication \fbox{$\Rightarrow$} in $(ii)$ is a consequence of \cite[Proposition 4.2]{RubenSemiDirect} (by which we know that if $\widehat{\Gamma {_\alpha}\bowtie_{\beta}G}$ is torsion-free, then $\beta$ is trivial) and property $(i)$ above. 
		Implication \fbox{$\Rightarrow$} in $(iii)$ is a consequence of Proposition \ref{pro.DivisibleConstructions} and Theorem \ref{theo.TorsionFreenessDivisible}. Implication \fbox{$\Leftarrow$} is contained in \cite[Theorem 3.17]{YukiKenny}. Implication \fbox{$\Rightarrow$} in $(iv)$ is a consequence of Proposition \ref{pro.DivisibleConstructions} and Theorem \ref{theo.TorsionFreenessDivisible}. Implication \fbox{$\Leftarrow$} is contained in \cite[Theorem 3.16]{YukiKenny}.
	\end{proof}

\section{\textsc{Baum-Connes property for discrete quantum subgroups}}

	In this last section we make some observations about the permanence of the quantum Baum-Connes conjecture by discrete quantum subgroups. We refer to \cite{MeyerNest} or \cite{Jorgensen} for a complete presentation of the categorical framework for the formulation of the Baum-Connes conjecture according to the Meyer-Nest approach.
	
	Let $\widehat{\mathbb{G}}$ be a discrete quantum group and consider the corresponding equivariant Kasparov category, $\mathscr{K}\mathscr{K}^{\widehat{\mathbb{G}}}$, which is a triangulated category with canonical suspension functor denoted by $\Sigma$. The word \emph{homomorphism (resp., isomorphism)} will mean \emph{homomorphism (resp., isomorphism) in the corresponding Kasparov category}; it will be a true homomorphism (resp., isomorphism) between C$^*$-algebras or any Kasparov triple between C$^*$-algebras (resp., any $KK$-equivalence between C$^*$-algebras). If $\mathcal{S}$ is a collection of objects in $\mathscr{K}\mathscr{K}^{\widehat{\mathbb{G}}}$, we denote by $\langle \mathcal{S} \rangle$ the \emph{localising subcategory of $\mathscr{K}\mathscr{K}^{\widehat{\mathbb{G}}}$ generated by $\mathcal{S}$}, i.e. the smallest subcategory of $\mathscr{K}\mathscr{K}^{\widehat{\mathbb{G}}}$ containing $\mathcal{S}$. Consider the following localising subcategories of $\mathscr{K}\mathscr{K}^{\widehat{\mathbb{G}}}$:
	$$\mathscr{L}_{\widehat{\mathbb{G}}}:=\langle \{\mathbb{G}\underset{r}{\ltimes} T\otimes C\ |\  C\in Obj.(\mathscr{K}\mathscr{K})\mbox{, } T\in\text{Tor}(\widehat{\mathbb{G}})\}\rangle,$$
	$$\mathscr{N}_{\widehat{\mathbb{G}}}:=\mathscr{L}^{\dashv}_{\widehat{\mathbb{G}}}=\{A\in Obj(\mathscr{K}\mathscr{K}^{\widehat{\mathbb{G}}})\ |\ KK^{\widehat{\mathbb{G}}}(L, A)=0\mbox{, $\forall$ $L\in Obj(\mathscr{L}_{\widehat{\mathbb{G}}})$}\}.$$
	
	\begin{remSec}
		We put $\widehat{\mathscr{L}}_{\widehat{\mathbb{G}}}:=\langle\{T\otimes C\ |\  C\in Obj.(\mathscr{K}\mathscr{K})\mbox{, } T\in\text{Tor}(\widehat{\mathbb{G}})\}\rangle$, so that we have $\mathbb{G}\ltimes \widehat{\mathscr{L}}_{\widehat{\mathbb{G}}}=\mathscr{L}_{\widehat{\mathbb{G}}}$ by definition.
	\end{remSec}
	
	A recent result by Y. Arano and A. Skalski \cite{YukiBCTorsion} shows that these two subcategories form indeed a complementary pair of localizing subcategories in $\mathscr{K}\mathscr{K}^{\widehat{\mathbb{G}}}$. The author together with K. De Commer and R. Nest \cite{KennyNestRubenBCProjective} have obtained the same conclusion for permutation torsion-free discrete quantum groups through different considerations by studying the projective representation theory of a compact quantum group. Consequently, it is possible now to formulate a general Baum-Connes conjecture for arbitrary discrete quantum groups (without torsion-freeness assumption), being $(\mathscr{L}_{\widehat{\mathbb{G}}}, \mathscr{N}_{\widehat{\mathbb{G}}})$ above the complementary pair that allows to define a quantum assembly map for $\widehat{\mathbb{G}}$. More precisely, we denote by $(L,N)$ the canonical triangulated functors associated to this complementary pair. Next, consider the homological functor $F:\mathscr{K}\mathscr{K}^{\widehat{\mathbb{G}}} \longrightarrow \mathscr{A}b^{\mathbb{Z}/2}$, $(A,\delta)  \longmapsto K_{*}(\widehat{\mathbb{G}}\underset{\delta, r}{\ltimes} A)$, where $\mathscr{A}b^{\mathbb{Z}/2}$ denotes the abelian category of $\mathbb{Z}/2$-graded groups of $\mathscr{A}b$. The quantum assembly map for $\widehat{\mathbb{G}}$ is given by the natural transformation $\eta^{\widehat{\mathbb{G}}}: \mathbb{L}F\longrightarrow F$. 
	\begin{defiSec}
		Let $\widehat{\mathbb{G}}$ be a discrete quantum group. We say that $\widehat{\mathbb{G}}$ satisfies the quantum Baum-Connes property (with coefficients) if the natural transformation $\eta^{\widehat{\mathbb{G}}}: \mathbb{L}F\longrightarrow F$ is a natural equivalence. We say that $\widehat{\mathbb{G}}$ satisfies the \emph{$\mathscr{L}_{\widehat{\mathbb{G}}}$-strong} Baum-Connes property if $\mathscr{K}\mathscr{K}^{\widehat{\mathbb{G}}}=\mathscr{L}_{\widehat{\mathbb{G}}}$.
	\end{defiSec}

	Let us analyse the permanence of the (resp. strong) Baum-Connes property by discrete quantum subgroups. To begin with, let $\widehat{\mathbb{H}}<\widehat{\mathbb{G}}$ be a discrete quantum subgroup of $\widehat{\mathbb{G}}$. We have two relevant functors: restriction, which is obvious, and induction, which has been studied by S. Vaes \cite{VaesInduction} in the framework of quantum groups: $\text{Res}^{\widehat{\mathbb{G}}}_{\widehat{\mathbb{H}}}:\mathscr{K}\mathscr{K}^{\widehat{\mathbb{G}}}\longrightarrow \mathscr{K}\mathscr{K}^{\widehat{\mathbb{H}}}\mbox{, }\text{Ind}^{\widehat{\mathbb{G}}}_{\widehat{\mathbb{H}}}:\mathscr{K}\mathscr{K}^{\widehat{\mathbb{H}}}\longrightarrow \mathscr{K}\mathscr{K}^{\widehat{\mathbb{G}}}$. It is well-known that restriction and induction are \emph{triangulated} functors by virtue of the universal property of the Kasparov category (see \cite{VoigtPoincareDuality} for more details). Denote by $(L',N')$ the canonical triangulated functors associated to the complementary pair $(\mathscr{L}_{\widehat{\mathbb{H}}}, \mathscr{N}_{\widehat{\mathbb{H}}})$ and by $F'$ the homological functor defining the quantum Baum-Connes assembly map for $\widehat{\mathbb{H}}$. 
	
	It is shown in \cite[Lemma 2.4.3]{RubenSemiDirect} that the restriction (resp. induction) functor transforms the assembly map for $\widehat{\mathbb{G}}$ (resp. for $\widehat{\mathbb{H}}$) into the assembly map for $\widehat{\mathbb{H}}$ (resp. for $\widehat{\mathbb{G}}$) whenever $\widehat{\mathbb{G}}$ is torsion-free and $\widehat{\mathbb{H}}$ is torsion-free and divisible. As a consequence, it is shown in \cite[Proposition 2.4.4]{RubenSemiDirect} by using the quantum Green's Imprimitivity theorem \cite[Theorem 7.3]{VaesInduction} that $\widehat{\mathbb{G}}$ satisfies the quantum Baum-Connes property if and only if every divisible torsion-free discrete quantum subgroup $\widehat{\mathbb{H}}<\widehat{\mathbb{G}}$ satisfies the quantum Baum-Connes property. Observe that thanks to Theorem \ref{theo.TorsionFreenessDivisible} it is enough to assume that $\widehat{\mathbb{G}}$ is torsion-free and $\widehat{\mathbb{H}}$ is divisible in $\widehat{\mathbb{G}}$, which simplifies the assumptions for \cite[Lemma 2.4.3]{RubenSemiDirect} and \cite[Proposition 2.4.4]{RubenSemiDirect}.
	
	One wishes to remove the divisibility assumptions in these results too. Thus, consider both $\widehat{\mathbb{G}}$ and $\widehat{\mathbb{H}}$ not to be necessarily torsion-free and $\widehat{\mathbb{H}}$ not to be necessarily divisible in $\widehat{\mathbb{G}}$. Consider the localizing subcategories:
		$$\mathscr{L}_{\widehat{\mathbb{G}}}:=\langle\{\mathbb{G}\underset{r}{\ltimes} T\otimes C\ |\  C\in Obj.(\mathscr{K}\mathscr{K})\mbox{, } T\in\text{Tor}(\widehat{\mathbb{G}})\}\rangle,$$
		$$\mathscr{L}_{\widehat{\mathbb{H}}}:=\langle\{\mathbb{H}\underset{r}{\ltimes} S\otimes D\ |\  D\in Obj.(\mathscr{K}\mathscr{K})\mbox{, } S\in\text{Tor}(\widehat{\mathbb{H}})\}\rangle,$$
		which are complemented by the corresponding right orthogonals:
		$$\mathscr{N}_{\widehat{\mathbb{G}}}:=\mathscr{L}^{\dashv}_{\widehat{\mathbb{G}}}=\{A\in Obj(\mathscr{K}\mathscr{K}^{\widehat{\mathbb{G}}})\ |\ KK^{\widehat{\mathbb{G}}}(L, A)=0\mbox{, $\forall$ $L\in Obj(\mathscr{L}_{\widehat{\mathbb{G}}})$}\},$$
		$$\mathscr{N}_{\widehat{\mathbb{H}}}:=\mathscr{L}^{\dashv}_{\widehat{\mathbb{H}}}=\{B\in Obj(\mathscr{K}\mathscr{K}^{\widehat{\mathbb{H}}})\ |\ KK^{\widehat{\mathbb{H}}}(L', B)=0\mbox{, $\forall$ $L'\in Obj(\mathscr{L}_{\widehat{\mathbb{H}}})$}\}.$$
		
	\begin{lemSec}
		Let $\widehat{\mathbb{G}}$ be a discrete quantum group and $\widehat{\mathbb{H}}<\widehat{\mathbb{G}}$ a discrete quantum subgroup (both $\widehat{\mathbb{G}}$ and $\widehat{\mathbb{H}}$ are not necessarily torsion-free and $\widehat{\mathbb{H}}$ is not necessarily divisible in $\widehat{\mathbb{G}}$): $i)$ $\text{Res}^{\widehat{\mathbb{G}}}_{\widehat{\mathbb{H}}}(\mathscr{N}_{\widehat{\mathbb{G}}})\subset \mathscr{N}_{\widehat{\mathbb{H}}}$ if and only if $\text{Ind}^{\widehat{\mathbb{G}}}_{\widehat{\mathbb{H}}}(\mathscr{L}_{\widehat{\mathbb{H}}})\subset \mathscr{L}_{\widehat{\mathbb{G}}}$; $ii)$ $\text{Ind}^{\widehat{\mathbb{G}}}_{\widehat{\mathbb{H}}}(\mathscr{N}_{\widehat{\mathbb{H}}})\subset \mathscr{N}_{\widehat{\mathbb{G}}}$ if and only if $\text{Res}^{\widehat{\mathbb{G}}}_{\widehat{\mathbb{H}}}(\mathscr{L}_{\widehat{\mathbb{G}}})\subset \mathscr{L}_{\widehat{\mathbb{H}}}$.
	\end{lemSec}
	\begin{proof}
		This is a consequence of the orthogonality condition of the complementary pairs. Indeed, recall that the functors $\text{Res}^{\widehat{\mathbb{G}}}_{\widehat{\mathbb{H}}}$ and $\text{Ind}^{\widehat{\mathbb{G}}}_{\widehat{\mathbb{H}}}$ are adjoint meaning that $KK^{\widehat{\mathbb{G}}}(A, \text{Ind}^{\widehat{\mathbb{G}}}_{\widehat{\mathbb{H}}}(B))\cong KK^{\widehat{\mathbb{H}}}(\text{Res}^{\widehat{\mathbb{G}}}_{\widehat{\mathbb{H}}}(A), B)$ and  \\$KK^{\widehat{\mathbb{G}}}(\text{Ind}^{\widehat{\mathbb{G}}}_{\widehat{\mathbb{H}}}(B), A)\cong KK^{\widehat{\mathbb{H}}}(B, \text{Res}^{\widehat{\mathbb{G}}}_{\widehat{\mathbb{H}}}(A)),$ for all $A\in Obj(\mathscr{K}\mathscr{K}^{\widehat{\mathbb{G}}})$ and all $B\in Obj(\mathscr{K}\mathscr{K}^{\widehat{\mathbb{H}}})$ (we refer to \cite{VoigtPoincareDuality} or \cite{VoigtBaumConnesUnitaryFree} for a proof). These relations yield the following: $i)$ given $N\in \mathscr{N}_{\widehat{\mathbb{G}}}$, $\text{Res}^{\widehat{\mathbb{G}}}_{\widehat{\mathbb{H}}}(N)\in \mathscr{N}_{\widehat{\mathbb{H}}} \Leftrightarrow 0=KK^{\widehat{\mathbb{H}}}(L', \text{Res}^{\widehat{\mathbb{G}}}_{\widehat{\mathbb{H}}}(N))\cong KK^{\widehat{\mathbb{G}}}(\text{Ind}^{\widehat{\mathbb{G}}}_{\widehat{\mathbb{H}}}(L'), N)\ \forall L'\in \mathscr{L}_{\widehat{\mathbb{H}}} \Leftrightarrow \text{Ind}^{\widehat{\mathbb{G}}}_{\widehat{\mathbb{H}}}(\mathscr{L}_{\widehat{\mathbb{H}}})\subset \mathscr{L}_{\widehat{\mathbb{G}}}$; and $ii)$ given $N'\in \mathscr{N}_{\widehat{\mathbb{H}}}$, $\text{Ind}^{\widehat{\mathbb{G}}}_{\widehat{\mathbb{H}}}(N')\in \mathscr{N}_{\widehat{\mathbb{G}}} \Leftrightarrow 0=KK^{\widehat{\mathbb{G}}}(L, \text{Ind}^{\widehat{\mathbb{G}}}_{\widehat{\mathbb{H}}}(N'))\cong KK^{\widehat{\mathbb{H}}}(\text{Res}^{\widehat{\mathbb{G}}}_{\widehat{\mathbb{H}}}(L), N')\ \forall L\in \mathscr{L}_{\widehat{\mathbb{G}}} \Leftrightarrow \text{Res}^{\widehat{\mathbb{G}}}_{\widehat{\mathbb{H}}}(\mathscr{L}_{\widehat{\mathbb{G}}})\subset \mathscr{L}_{\widehat{\mathbb{H}}}.$
	\end{proof}
	
	\begin{proSec}\label{pro.StabilityResIndSubgroups}
		Let $\widehat{\mathbb{G}}$ be a discrete quantum group and $\widehat{\mathbb{H}}<\widehat{\mathbb{G}}$ a discrete quantum subgroup (both $\widehat{\mathbb{G}}$ and $\widehat{\mathbb{H}}$ are not necessarily torsion-free and $\widehat{\mathbb{H}}$ is not necessarily divisible in $\widehat{\mathbb{G}}$). The inclusions $\text{Ind}^{\widehat{\mathbb{G}}}_{\widehat{\mathbb{H}}}(\mathscr{L}_{\widehat{\mathbb{H}}})\subset \mathscr{L}_{\widehat{\mathbb{G}}}$ and $\text{Res}^{\widehat{\mathbb{G}}}_{\widehat{\mathbb{H}}}(\mathscr{L}_{\widehat{\mathbb{G}}})\subset \mathscr{L}_{\widehat{\mathbb{H}}}$ hold.
	\end{proSec}
	\begin{proof}
		On the one hand, as showed in Proposition \ref{pro.CorrespondenceTorsionActions}, the restriction to $\mathbb{H}$ of any torsion action of $\mathbb{G}$ decomposes as a direct sum of torsion actions of $\mathbb{H}$. Moreover, the restriction functor is triangulated and compatible with countable direct sums. The inclusion $\text{Res}^{\widehat{\mathbb{G}}}_{\widehat{\mathbb{H}}}(\mathscr{L}_{\widehat{\mathbb{G}}})\subset \mathscr{L}_{\widehat{\mathbb{H}}}$ follows. On the other hand,  in order to show the inclusion $\text{Ind}^{\widehat{\mathbb{G}}}_{\widehat{\mathbb{H}}}(\mathscr{L}_{\widehat{\mathbb{H}}})\subset \mathscr{L}_{\widehat{\mathbb{G}}}$ we are going to show that the following diagram is commutative:
		\begin{equation}\label{eq.DiagramStabilityIndSubgroup}
		\begin{split}
			\xymatrix@!C=20mm@R=15mm{
				\mbox{$\mathscr{K}\mathscr{K}^{\mathbb{H}}$}\ar[d]_{\mbox{$\mathbb{H}\underset{r}{\ltimes} \cdot$}}\ar[r]^{\mbox{$\text{Ind}^{\mathbb{G}}_{\mathbb{H}}(\cdot)$}}&\mbox{$\mathscr{K}\mathscr{K}^{\mathbb{G}}$}\ar[d]^{\mbox{$\mathbb{G}\underset{r}{\ltimes} \cdot$}}\\
				\mbox{$\mathscr{K}\mathscr{K}^{\widehat{\mathbb{H}}}$}\ar[r]_{\mbox{$\text{Ind}^{\widehat{\mathbb{G}}}_{\widehat{\mathbb{H}}}(\cdot)$}}&\mbox{$\mathscr{K}\mathscr{K}^{\widehat{\mathbb{G}}}$}&}
		\end{split}
		\end{equation}
		
		Then given a torsion action of $\mathbb{H}$, $(B,\beta)\in\text{Tor}(\widehat{\mathbb{H}})$, $(\text{Ind}^{\mathbb{G}}_{\mathbb{H}}(B), \widetilde{\beta})$ is a torsion action of $\mathbb{G}$ (it might be the trivial one; recall the discussion of Section \ref{sec.StabilityTorsionFreeness}). Hence, by definition we have $\mathbb{G}\underset{r}{\ltimes} \text{Ind}^{\mathbb{G}}_{\mathbb{H}}(B, \beta)\in\mathscr{L}_{\widehat{\mathbb{G}}}$. Also by definition we have that $\mathbb{H}\underset{r}{\ltimes} B\in\mathscr{L}_{\widehat{\mathbb{H}}}$. If Diagram (\ref{eq.DiagramStabilityIndSubgroup}) commutes, then we will have $\text{Ind}^{\widehat{\mathbb{G}}}_{\widehat{\mathbb{H}}}\Big(\mathbb{H}\underset{r}{\ltimes} B\Big)\cong \mathbb{G}\underset{r}{\ltimes} \text{Ind}^{\mathbb{G}}_{\mathbb{H}}(B, \beta)\in \mathscr{L}_{\widehat{\mathbb{G}}}$ (where the isomorphism is in the Kasparov category $\mathscr{K}\mathscr{K}^{\widehat{\mathbb{G}}}$), which would yield then the desired inclusion. To prove commutativity of Diagram (\ref{eq.DiagramStabilityIndSubgroup}) observe that $\mathscr{K}\mathscr{K}^{\mathbb{H}}\overset{\mathbb{H}\ltimes \cdot}{\cong} \mathscr{K}\mathscr{K}^{\widehat{\mathbb{H}}}$ and $\mathscr{K}\mathscr{K}^{\mathbb{G}}\overset{\mathbb{G}\ltimes \cdot}{\cong} \mathscr{K}\mathscr{K}^{\widehat{\mathbb{G}}}$ by virtue of Baaj-Skandalis duality. Therefore, it is enough to show that $\widehat{\mathbb{G}}\underset{r}{\ltimes} \text{Ind}^{\widehat{\mathbb{G}}}_{\widehat{\mathbb{H}}}\Big(\mathbb{H}\underset{r}{\ltimes} B\Big)\cong \widehat{\mathbb{G}}\underset{r}{\ltimes} \Big(\mathbb{G}\underset{r}{\ltimes} \text{Ind}^{\mathbb{G}}_{\mathbb{H}}(B, \beta)\Big)$ as objects in $\mathscr{K}\mathscr{K}^{\mathbb{G}}$. In other words, by applying the quantum version of Green's imprimitivity theorem (see \cite{VaesInduction} for more details), we have to show that $\widehat{\mathbb{H}}\underset{r}{\ltimes} \Big(\mathbb{H}\underset{r}{\ltimes} B\Big)\cong \widehat{\mathbb{G}}\underset{r}{\ltimes} \Big(\mathbb{G}\underset{r}{\ltimes} \text{Ind}^{\mathbb{G}}_{\mathbb{H}}(B, \beta)\Big)\Leftrightarrow B\cong \text{Ind}^{\mathbb{G}}_{\mathbb{H}}(B, \beta)$ 
		as objects in $\mathscr{K}\mathscr{K}^{\mathbb{G}}$, which is true.
	\end{proof}
	
	The first direct consequence from these claims is a generalisation of \cite[Lemma 2.4.3]{RubenSemiDirect}:
	\begin{corSec}\label{cor.PropertiesResInd}
		Let $\mathbb{G}$, $\mathbb{H}$ be two compact quantum groups such that $\widehat{\mathbb{H}}<\widehat{\mathbb{G}}$. The following properties hold.
		\begin{enumerate}[i)]
			\item $\text{Res}^{\widehat{\mathbb{G}}}_{\widehat{\mathbb{H}}}(\mathscr{L}_{\widehat{\mathbb{G}}})\subset \mathscr{L}_{\widehat{\mathbb{H}}}$ and $\text{Res}^{\widehat{\mathbb{G}}}_{\widehat{\mathbb{H}}}(\mathscr{N}_{\widehat{\mathbb{G}}})\subset \mathscr{N}_{\widehat{\mathbb{H}}}$. Hence, we have the following natural isomorphisms $\text{Res}^{\widehat{\mathbb{G}}}_{\widehat{\mathbb{H}}}\circ L\cong L'\circ \text{Res}^{\widehat{\mathbb{G}}}_{\widehat{\mathbb{H}}}$ and $\text{Res}^{\widehat{\mathbb{G}}}_{\widehat{\mathbb{H}}}\circ N\cong N'\circ \text{Res}^{\widehat{\mathbb{G}}}_{\widehat{\mathbb{H}}}$.
			\item $\text{Ind}^{\widehat{\mathbb{G}}}_{\widehat{\mathbb{H}}}(\mathscr{L}_{\widehat{\mathbb{H}}})\subset \mathscr{L}_{\widehat{\mathbb{G}}}$ and $\text{Ind}^{\widehat{\mathbb{G}}}_{\widehat{\mathbb{H}}}(\mathscr{N}_{\widehat{\mathbb{H}}})\subset \mathscr{N}_{\widehat{\mathbb{G}}}$. Hence, we have the following natural isomorphisms $\text{Ind}^{\widehat{\mathbb{G}}}_{\widehat{\mathbb{H}}}\circ L'\cong L\circ \text{Ind}^{\widehat{\mathbb{G}}}_{\widehat{\mathbb{H}}}$ and $\text{Ind}^{\widehat{\mathbb{G}}}_{\widehat{\mathbb{H}}}\circ N'\cong N\circ \text{Ind}^{\widehat{\mathbb{G}}}_{\widehat{\mathbb{H}}}$.
		\end{enumerate}
		
		Consequently, $\text{Res}^{\widehat{\mathbb{G}}}_{\widehat{\mathbb{H}}}$ transforms the assembly map for $\widehat{\mathbb{G}}$ into the assembly map for $\widehat{\mathbb{H}}$ and $\text{Ind}^{\widehat{\mathbb{G}}}_{\widehat{\mathbb{H}}}$ transforms the assembly map for $\widehat{\mathbb{H}}$ into the assembly map for $\widehat{\mathbb{G}}$.
	\end{corSec}
	
	Accordingly, we generalise \cite[Proposition 2.4.4]{RubenSemiDirect}:
	\begin{proSec}\label{cor.BCQuantumSubgroups}
		Let $\mathbb{G}$, $\mathbb{H}$ be two compact quantum groups such that $\widehat{\mathbb{H}}<\widehat{\mathbb{G}}$.
		\begin{enumerate}[i)]
			\item $\widehat{\mathbb{G}}$ satisfies the quantum Baum-Connes property if and only if every discrete quantum subgroup $\widehat{\mathbb{H}}< \widehat{\mathbb{G}}$ satisfies the quantum Baum-Connes property.
			\item If $\widehat{\mathbb{G}}$ satisfies the $\mathscr{L}_{\widehat{\mathbb{G}}}$-strong Baum-Connes property, then $\widehat{\mathbb{H}}$ satisfies the $\mathscr{L}_{\widehat{\mathbb{H}}}$-strong Baum-Connes property.
		\end{enumerate}
	\end{proSec}
	\begin{proof}
		The proof of the statement in $(i)$ is done verbatim the proof appearing in \cite[Proposition 2.4.4]{RubenSemiDirect} by applying now Corollary \ref{cor.PropertiesResInd}. For the statement in $(ii)$, we observe that we always have $\text{Res}^{\widehat{\mathbb{G}}}_{\widehat{\mathbb{H}}}(\mathscr{L}_{\widehat{\mathbb{G}}})\subset \mathscr{L}_{\widehat{\mathbb{H}}}$ as showed in Proposition \ref{pro.StabilityResIndSubgroups}. Assume that $\widehat{\mathbb{G}}$ satisfies the $\mathscr{L}_{\widehat{\mathbb{G}}}$-strong Baum-Connes property, which means that $\mathscr{L}_{\widehat{\mathbb{G}}}=\mathscr{K}\mathscr{K}^{\widehat{\mathbb{G}}}$. Given $(B,\beta)\in Obj(\mathscr{K}\mathscr{K}^{\widehat{\mathbb{H}}})$, we have thus $\text{Ind}^{\widehat{\mathbb{G}}}_{\widehat{\mathbb{H}}}(B)\in \mathscr{K}\mathscr{K}^{\widehat{\mathbb{G}}}=\mathscr{L}_{\widehat{\mathbb{G}}}$. Hence, $\text{Res}^{\widehat{\mathbb{G}}}_{\widehat{\mathbb{H}}}\Big(\text{Ind}^{\widehat{\mathbb{G}}}_{\widehat{\mathbb{H}}}(B)\Big)\in \mathscr{L}_{\widehat{\mathbb{H}}}$. To conclude recall that $B$ is a retract of $\text{Res}^{\widehat{\mathbb{G}}}_{\widehat{\mathbb{H}}}\Big(\text{Ind}^{\widehat{\mathbb{G}}}_{\widehat{\mathbb{H}}}(B)\Big)$ and that localization subcategories are closed under retracts, which implies that $B\in\mathscr{L}_{\widehat{\mathbb{H}}}$. So $\widehat{\mathbb{H}}$ satisfies the $\mathscr{L}_{\widehat{\mathbb{H}}}$-strong Baum-Connes property.
	\end{proof}
	
	\begin{remSec}
		It is well-known that the \emph{strong} Baum-Connes property is preserved by \emph{divisible} discrete quantum subgroups whenever $\widehat{\mathbb{G}}$ is torsion-free (see \cite[Lemma 6.7]{VoigtBaumConnesUnitaryFree} for a proof). The above corollary generalizes this property for any discrete quantum subgroup (not necessarily divisible and without any torsion-freeness assumption), which has been possible thanks to the more abstract approach in terms of fusion rings and module C$^*$-categories from Section \ref{sec.PreparatoryObs}. Moreover, it is important to mention that property $(ii)$ of the previous proposition appears already in \cite[Remark 3.20]{RubenAmauryTorsion}. The argument appearing there has been expanded in the present paper with Corollary \ref{cor.IndResTorsion} and Proposition \ref{pro.CorrespondenceTorsionActions}. However, property $(i)$ of the previous proposition concerning the usual Baum-Connes property is new and legitimate to consider now thanks to \cite{YukiBCTorsion}, \cite{KennyNestRubenBCProjective}.  
	\end{remSec}
	
	Moreover, Proposition \ref{cor.BCQuantumSubgroups} allows to improve some stability results for the quantum Baum-Connes property with respect to some relevant constructions, which already appear in \cite{RubenSemiDirect}, by removing torsion-freeness and divisibility assumptions.
	\begin{corSec}
		The following stability properties for the quantum (resp. strong) Baum-Connes property hold.
		\begin{enumerate}[i)]
			\item Let $\Gamma$ be a discrete group, $\mathbb{G}$ a compact quantum group. Assume that $\Gamma$ acts on $\mathbb{G}$ by quantum automorphisms, $\alpha$. Then $\Gamma$ and $\widehat{\mathbb{G}}$ satisfy the (resp. strong) Baum-Connes property whenever $\widehat{\Gamma\underset{\alpha}{\ltimes}\mathbb{G}}$ satisfies the (resp. strong) Baum-Connes property.
			\item Let $\Gamma$ be a discrete group, $G$ a compact group. Assume that $(\Gamma, G)$ is a matched pair with actions $\Gamma \overset{\alpha}{\curvearrowright} G$ and $G \overset{\beta}{\curvearrowright} \Gamma$. Then $\widehat{G}$ satisfies the (resp. strong) Baum-Connes property whenever $\widehat{\Gamma {_\alpha}\bowtie_{\beta}G}$ satisfies the (resp. strong) Baum-Connes property.
			\item Let $\mathbb{G}$ and $\mathbb{H}$ be two compact quantum groups. Then $\widehat{\mathbb{G}}$ and $\widehat{\mathbb{H}}$ satisfy the (resp. strong) Baum-Connes property whenever $\widehat{\mathbb{G}\times \mathbb{H}}$ satisfies the (resp. strong) Baum-Connes property.
			\item Let $\mathbb{G}$ and $\mathbb{H}$ be two compact quantum groups. Then $\widehat{\mathbb{G}}$ and $\widehat{\mathbb{H}}$ satisfy the (resp. strong) Baum-Connes property whenever $\widehat{\mathbb{G}\ast \mathbb{H}}$ satisfies the (resp. strong) Baum-Connes property.
		\end{enumerate}
	\end{corSec}

\bibliographystyle{acm}
\bibliography{TorsionFreenessDivisibleDiscreteQuantumSubgroups}



 \end{document}